\documentclass[a4paper,twoside,reqno]{amsart}
\usepackage{amssymb,amscd,amsxtra,verbatim}
\usepackage{lscape,enumerate}

\usepackage[mathscr]{euscript} 
\usepackage{mathrsfs}
\usepackage[hyperref]{xcolor} 
\usepackage[colorlinks=true]{hyperref}
\hypersetup{linktocpage} \hypersetup{urlcolor=violet!90!black}
\hypersetup{linkcolor=red!65!black}
\hypersetup{citecolor=blue!55!black} 
\usepackage{breakurl}
\usepackage{graphicx} 
\usepackage{pdfpages}

\definecolor{note}{rgb}{0.1,0.1,0.4}

\usepackage[french,english]{babel} 
\usepackage[T1]{fontenc}
\usepackage{xspace}

\usepackage{setspace} 
\usepackage[headings]{fullpage}

\makeatletter
\newcommand{\Appendix}[2][]{%
\refstepcounter{section}%
\addcontentsline{toc}{section}%
{\appendixname #1}%
\par
{\centering\normalfont\normalsize\scshape\appendixname
#1\par}\vskip6pt\@afterheading}

\newcommand{\sAppendix}[1]{%
\addcontentsline{toc}{section}%
{\appendixname #1}%
\par
{\centering\normalfont\normalsize\scshape\appendixname #1\par}%
\vskip6pt\@afterheading}
\makeatother

\numberwithin{equation}{section} 

\theoremstyle{plain}
\newtheorem{thm}[equation]{Theorem}
\newtheorem{lem}[equation]{Lemma}
\newtheorem{prop}[equation]{Proposition}
\newtheorem{cor}[equation]{Corollary}

 \newtheorem{hypotheses}[equation]{Hypotheses}

\theoremstyle{definition}
\newtheorem{defn}[equation]{Definition}

\newtheorem{Notation}[equation]{Notation}

\newtheorem{rem}[equation]{Remark}

\theoremstyle{remark}

\newtheorem*{example}{Example}

\theoremstyle{plain}

\newenvironment{french}{\begin{otherlanguage}{french}}{\end{otherlanguage}}

  {\end{itemize}}

  {\end{enumerate}}

\newcommand{\Q}{{\mathbb{Q}}} 
\newcommand{\N}{{\mathbb{N}}}
\newcommand{\C}{{\mathbb{C}}} 
\newcommand{\R}{{\mathbb{R}}}
 
\newcommand{\Z}{{\mathbb{Z}}}
 
\newcommand{\K}{{\mathbb{K}}}
 
\newcommand{\I}{{\mathbb{I}}}

\newcommand{\calC}{{\mathcal{C}}} 
\newcommand{\calO}{{\mathcal{O}}}
\newcommand{\cO}{{\mathcal{O}}} 
\newcommand{\calV}{{\mathcal{V}}}
\newcommand{\calA}{{\mathcal{A}}} 

\newcommand{\cB}{{\mathcal{B}}} 
\newcommand{\calD}{{\mathcal{D}}}
\newcommand{\cD}{{\mathcal{D}}}

\newcommand{\cF}{{\mathcal{F}}}

 \newcommand{\cH}{{\mathcal{H}}}
 
\newcommand{\calI}{{\mathcal{I}}} 
 \newcommand{\cJ}{{\mathcal{J}}}
\newcommand{\calL}{{\mathcal{L}}} \newcommand{\cL}{{\mathcal{L}}}
\newcommand{\calM}{{\mathcal{M}}} \newcommand{\cM}{{\mathcal{M}}}

\newcommand{\cQ}{{\mathcal{Q}}}

\newcommand{\calS}{{\mathcal{S}}}
\newcommand{\cS}{{\mathcal{S}}} 

\newcommand{\calU}{{\mathcal{U}}}
\newcommand{\calIr}{{\mathcal{I}_r}}  

\DeclareMathAlphabet{\euls}{U}{eus}{m}{n}

\newcommand{\fa}{{\mathfrak{a}}} 
 \newcommand{\g}{{\mathfrak{g}}}
\newcommand{\fh}{{\mathfrak{h}}} \newcommand{\h}{{\mathfrak{h}}}
\newcommand{\fm}{{\mathfrak{m}}} \newcommand{\fn}{{\mathfrak{n}}}
\newcommand{\fb}{{\mathfrak{b}}} \newcommand{\fr}{{\mathfrak{r}}}
\newcommand{\fp}{{\mathfrak{p}}} \newcommand{\p}{{\mathfrak{p}}}
 \newcommand{\fk}{{\mathfrak{k}}}
 
\newcommand{\fsl}{{\mathfrak{sl}}} \newcommand{\fso}{{\mathfrak{so}}}
 
\newcommand{\fgl}{{\mathfrak{gl}}} \newcommand{\fsp}{{\mathfrak{sp}}}
 \newcommand{\fs}{{\mathfrak{s}}}

 \newcommand{\mbO}{{\mathbf{O}}}

 \newcommand{\Et}{{\mathtt{E}}}
\newcommand{\Ft}{{\mathtt{F}}}

\newcommand{\Asf}{{\mathsf{A}}} \newcommand{\Bsf}{{\mathsf{B}}}
 \newcommand{\Dsf}{{\mathsf{D}}}

 \newcommand{\Psf}{{\mathsf{P}}}

\DeclareMathAlphabet{\mathpzc}{OT1}{pzc}{m}{it}

\DeclareMathSymbol{\Ima}{\mathord}{symbols}{"3D}

\def\preisomto{\vbox{\hbox to
    14pt{\hfill$\sim$\hfill}\nointerlineskip\vskip -0.3pt \hbox to
    14pt{\rightarrowfill}}}

\def\isomto{\mathop{\preisomto}}

\def\prelongisomto{\vbox{\hbox to
    17pt{\hfill$\sim$\hfill}\nointerlineskip\vskip -0.3pt \hbox to
    17pt{\rightarrowfill}}}

\def\longisomto{\mathop{\prelongisomto}}

\newcommand{\stoo}{\longrightarrow \kern-15pt
  \longrightarrow} 
\newcommand{\lto}{\longrightarrow}

\newcommand{\ito}{\hookrightarrow}

\def\trait{\hbox to 4mm{\hrulefill}}

\def\Aut{\operatorname {Aut}} 
\def\ad{\operatorname {ad}}
 
\def\ann{\operatorname {ann}}
\def\Ker{\operatorname {Ker}} 
\def\Im{\operatorname {Im}}
\def\Hom{\operatorname {Hom}} 
\def\End{\operatorname {End}}

\def\SO{\operatorname {SO}}

\def\Ext{\operatorname {Ext}}

\def\GKdim{\operatorname {GKdim}}
 
\def\Lie{\operatorname {Lie}}
\def\Der{\operatorname {Der}}

\def\rk{\operatorname {rk}}

\def\Ort{\operatorname {O}} 
\def\SO{\operatorname {SO}}

\DeclareMathOperator{\dalembert}{\Box}
\DeclareMathOperator{\Soc}{Soc}

\newcommand{\socRs}{Z_s}
\newcommand{\socRa}{Z_a}
\newcommand{\socRb}{Z_b}
\newcommand{\socRr}{Z_r}

\def\boplus{\mathbin{\boldsymbol{\oplus}}}

\newcommand{\UU}{\overline{U}}

\newcommand{\dual}[2]{{\langle {#1} \, ,
    {#2} \rangle}}

\newcommand{\ascal}[2]{{\langle {#1} \mid {#2} \rangle}}
\newcommand{\corth}[1]{{{#1}^{\circ}}}

\newcommand{\moins}{\smallsetminus}

\newcommand{\dpar}[2]{{\frac{\partial {#1}}{\partial {#2}}}}
\newcommand{\dpartial}[1]{{\partial_{#1}}}

\newcommand{\vt}{\vartheta} 
\newcommand{\vpi}{\varpi}
\newcommand{\vepsilon}{\varepsilon} 
\newcommand{\vphi}{\varphi}

\newcommand{\half}{{\frac{1}{2}}}

\usepackage{upgreek}

\newcommand{\susbet}{\subset}

{\selectlanguage{english} \hyphenation{equi-va-ri-ant equi-va-ri-ant
    Lapla-cian Lap-lacian} }

\newcommand{\sontwo}{{\fso(n+2,\C)}} 

\newcommand{\son}{{\fso(n,\C)}} 

\newcommand{\ddx}[1]{\partial_{X_{{#1}}}}

\newcommand{\ddu}[1]{\partial_{U_{{#1}}}}

\newcommand{\halfn}{{\frac{n}{2}}}

\newcommand{\halfZ}{{\half\Z}}
\newcommand{\dt}{\dpartial{t}}
\newcommand{\Omin}{{\mathbf{O}_\mathrm{min}}}
\newcommand{\Ominbar}{{\overline{\mathbf{O}}_\mathrm{min}}}
 
\newcommand{\rone}{{\sqrt{-1}}}
\newcommand{\barC}{{\overline{\mathsf{C}}}}

\newcommand{\Scr}{{\mathscr{S}}}
\newcommand{\Sscr}[1]{{\mathscr{S}({#1})}}

\newcommand{\Qt}{{\mathtt{Q}}}
\newcommand{\novertwo}{\textstyle{\frac{n}{2}}}

\DeclareMathOperator{\wt}{wt}
\DeclareMathOperator{\eval}{ev}
\DeclareMathOperator{\ev0}{\eval_0}

\newcommand{\cterm}[1]{{{{#1}{\raisebox{-3pt}{|}}}_{0}}}
\newcommand{\Ftt}{{\widetilde{\mathtt{F}}}}

\newcommand{\St}{{\widetilde{S}}}
\newcommand{\gt}{{\widetilde{\mathfrak{g}}}}

\newcommand{\Pt}{{\widetilde{P}}} 
\newcommand{\Dt}{{\widetilde{D}}}
\newcommand{\psit}{{\widetilde{\psi}}}
\newcommand{\vphit}{{\widetilde{\vphi}}}

\newcommand{\Sbar}[1]{\overline{S}({#1})}
\newcommand{\Deltabar}{\overline{\Delta}}
\newcommand{\Qtbar}{{\overline{\Qt}}}

\begin{document}
 
% temporary spacing
\begin{spacing}{1.4}
%%%%%%%%%%%

  \title[Higher symmetries and differential operators]{Higher symmetries of
    powers of the Laplacian and rings of differential operators}
 %  \author{T.~Levasseur} 
% \email{Thierry.Levasseur@univ-brest.fr}
% \address{Laboratoire de Math\'ematiques de Bretagne
%     Atlantique, CNRS~-~UMR6205, Universit\'e de Brest, 29238 Brest cedex~3,
%     France.}  
% \author{J. T. Stafford}
%  \email{Toby.Stafford@manchester.ac.uk}
%   \address{School of Mathematics, Alan Turing Building, The University of
%     Manchester, Oxford Road, Manchester M13 9PL, England.}
   \author{T.~Levasseur} \address{Laboratoire de Math\'ematiques de Bretagne
    Atlantique, CNRS~-~UMR6205, Universit\'e de Brest, 29238 Brest cedex~3,
    France.} \email{Thierry.Levasseur@univ-brest.fr} \author{J. T. Stafford}
  \address{School of Mathematics, Alan Turing Building, The University of
    Manchester, Oxford Road, Manchester M13 9PL, England.}
  \email{Toby.Stafford@manchester.ac.uk}

  \keywords{Higher symmetries, Laplacian, rings of differential operators,
 harmonic polynomials, primitive ideals, scalar singletons, conformal geometry.} 
\subjclass[2010]{Primary 16S32,  58J70, 17B08;  Secondary 17B81,
    53A30, 70S10.}

\thanks{The   second author was partially supported  by EPSERC grant  EP/L018322/1}    

\begin{abstract}  
  We study the interplay between the  minimal representations of
  the orthogonal Lie algebra  
$\g=\fso(n+2,\C)$ and the
  \emph{algebra of symmetries} $\Sscr{\dalembert^r}$ of
  powers of the Laplacian $\dalembert$ on $\C^{n}$. The
  connection is made through the construction of a highest
  weight representation of $\g$ via the ring of differential
  operators $\cD(X)$ on the singular scheme
  $X=(\Ft^r=0)\subset \C^n$, for
  $\Ft=\sum_{j=1}^nX_i^2\in \C[X_1,\dots,X_n]$.
  
  In particular we prove that
  $U(\g)/K_r \cong \Sscr{\dalembert^r}\cong \cD(X)$ for a
  certain primitive ideal $K_r$.  Interestingly, if (and
  only if) $n$ is even with $r\geq \halfn$ then both
  $\Sscr{\dalembert^r}$ and its natural module
  $\mathcal{A}=\C\bigl[ \frac{\!\!\!\!\partial}{\partial
    X_n},\dots, \frac{\!\!\!\!\partial}{\partial
    X_n}\bigr] \big/(\dalembert^r)$
  have a finite dimensional factor.  The same holds for the
  $\cD(X)$-module $\cO(X)$.
    
  We also study higher dimensional analogues
  $M_r=\{x\in A:\dalembert^r (x)=0\}$ of the module of
  harmonic elements in $A=\C[X_1,\dots,X_n]$ and of the
  space of ``harmonic densities''.  In both cases we obtain a
  minimal $\g$-representation that is closely related to the
  $\g$-modules $\cO(X)$ and $\mathcal{A}$.
    
  Essentially all these results have real analogues, with
  the Laplacian replaced by the d'Alembertian $\dalembert_p$
  on the pseudo-Euclidean space $\R^{p,q}$ and $\g$ replaced
  by the real Lie algebra $\fso(p+1,q+1)$.
 \end{abstract}
   
\maketitle

\tableofcontents

%\clearpage

\section{Introduction}

  In this paper we explain  a close relationship between
  two extensively studied subjects: the construction of ``minimal
  representations'' of the orthogonal Lie algebra $\fso(p+1,q+1)$ and the
  ``algebra of symmetries'' of powers of the Laplacian on the pseudo-Euclidean
  space $\R^{p,q}$. The connection is made through the construction of a
  highest weight representation of $\fso(n+2,\C)$ via differential operators
  on a singular subscheme of $\C^n$.  This approach, and its relation to
  Howe duality, goes back to \cite{LSS, LS} (see also \cite{jo881}).

  Let $\R[X_1,\dots,X_n]$ be a polynomial ring in $n \ge 3$ variables over
  the real numbers, with ring of differential operators (also known as the
  Weyl algebra) $A_n(\R)= \R[X_j, \ddx{k}; 1 \le j, k \le n]$, where
  $\ddx{k}= \frac{\!\!\!\!\partial}{\partial X_k}$.  An important technique
  for the group-theoretical analysis of the differential equations
  associated to a differential operator $\Theta\in A_n(\R)$ is to determine
  the higher symmetries of $\Theta$. Here, an operator $P\in A_n(\R)$ is a
  \emph{(higher) symmetry} of $\Theta$ provided $\Theta P=P'\Theta$ for
  some second operator $P'\in A_n(\R)$. The operator $P$ is a trivial
  symmetry when $P \in A_n(\R)\Theta$ and factoring the space of symmetries
  by the trivial ones gives the \emph{algebra of symmetries}
  $\Scr(\Theta )$ (see~\S\ref{symmetries} for more details).  These
  symmetries are important to differential equations through  exact
  integrability and  separation of variables (see for example  \cite{bagrov, BKM, miller})
  and have further applications in general relativity and  the theory of
  higher spin fields (see 
  \cite{ea05} and the references therein). 
   
  Finding these symmetries is, however, nontrivial and has been the focus of
  considerable research. In this paper we are concerned with the case when
  $\Theta=\dalembert_p^r$ is some power of the d'Alembertian
  $\dalembert_p= \sum_{j=1}^p \ddx{j}^{\,2} -
  \sum_{j=p+1}^{p+q}\ddx{j}^{\,2}$
  defined on the pseudo-Euclidean space $\R^{p,q}$, where $n=p+q$, equipped with the
  metric $\Ft_p= \sum_{j=1}^p X_{j}^2 - \sum_{j=p+1}^{n} X_{j}^{2}$.  In
  this case the algebra of symmetries is known and, moreover, forms a
  factor algebra of the enveloping algebra $U(\fso(p+1,q+1))$ of the real
  orthogonal Lie algebra $\fso(p+1,q+1)$; see, \cite{ea05} for the first
  definitive study for $r=1$ and then \cite{eale08,gosi12,bekgri,mi11} for
  more general results proved using a variety of methods. Earlier results,
  albeit without full proofs, appear in \cite{ss92}.

 We give a more algebraic approach to this problem which
  also permits a much more detailed description of the
  $\fso(p+1,q+1)$-module structure of the space of symmetries, which in turn
  has some interesting consequences. In particular, in
  Theorem~\ref{thm8922} and Corollary~\ref{cor8922} we prove:

  \begin{thm}
    \label{thmB} {\rm (1)} For an appropriate primitive  ideal $K_r$   there exists an
    isomorphism
    \[
      \frac{U(\fso(p+1,q+1))}{K_r} \ \longisomto\ \Sscr{\dalembert_p^r},
    \]
  where $n=p+q$.   In particular, $\Sscr{\dalembert_p^r}$ is a finitely
    generated, prime noetherian ring of Goldie rank $r$ {\rm
      (}see Remark~\ref{Joseph-remark2}{\rm )}.

    {\rm (2)} Consider
    $\mathcal{A}=\R[\partial_{X_1},\dots, \partial_{X_n}]/(\dalembert_p^r)$
    under its natural $\Sscr{\dalembert_p^r}$-module structure.  If $n$ is
    even with $r\geq \halfn$, then $\mathcal{A}$ has a unique proper
    factor $\Sscr{\dalembert_p^r}$-module $L$, which is finite dimensional.
Otherwise $\mathcal{A}$ is an irreducible
    $\Sscr{\dalembert_p^r}$-module.
  \end{thm}

\begin{rem}
  Theorem~\ref{thm8922} also give explicit sets of
  generators for the algebra $\Sscr{\dalembert_p^r}$.  For
  the weights of the corresponding highest weight
  $\sontwo$-module $\mathcal{A}_\C = \mathcal{A} \otimes_\R
  \C$ and its factor $L_\C$, 
  see Theorems~\ref{thmC} and \ref{thm26}.
  
    The  significance of the  finite dimensional module $L$ in this theorem (and of its
  annihilator, which will be an ideal of finite codimension
  in  $\Sscr{\dalembert_p^r}$) needs further study.   
\end{rem}

Let us now describe our approach. First, results for the real algebras
follow easily for the corresponding results for complex algebras and so,
for the rest of the introduction we concentrate on the latter case; thus we
work with the polynomial ring $A= \C[X_1,\dots,X_n]$ and Weyl algebra
$A_n(\C)= \C[X_j, \ddx{k}; 1 \le j, k \le n]$ and consider the complex algebra
of symmetries $\Scr(\Delta^r_1)$ for the Laplacian
$\Delta_1=\sum_{j=1}^n \ddx{j}^2$. Secondly, let $\cF$ be the Fourier
transform of the Weyl algebra $A_n = A_n(\C)$ which interchanges the $X_j$ and
$\ddx{j}$. This maps the Laplacian $\Delta_1$ onto
$\Ft= \sum_{j=1}^n X_j^2$.  A quick examination of the definitions shows
that $\cF$ gives an anti-isomorphism from
$\Scr(\Delta^r_1)$ onto $\I\bigl(\Ft^r A_n)/\Ft^r A_n$, where
$\I\bigl(\Ft^r A_n) = \{\theta\in A_n : \theta \Ft^r A_n\subseteq \Ft^r
A_n\}$
is the \emph{idealizer} of $\Ft^rA_n$.  Moreover, it is standard that
$\I\bigl(\Ft^r A_n)/\Ft^r A_n\cong \cD(R_r)$, the ring of differential
operators on the commutative ring $R_r=A/\Ft^rA$ (see Section~\ref{sec1}).

Thus, we need to understand the ring $\cD(R_r)$ and its canonical module
$R_r$.  A remarkable property of the algebra $R_1$ is that it carries a
representation of the orthogonal Lie algebra $\g=\sontwo$, and the starting
point of this paper is a result proved in \cite[Theorem~\ref{thm22}]{LSS}:
there exists a maximal ideal $J_1$, called the Joseph ideal, in the
enveloping algebra $U(\g)$ such that $\cD(R_1)\cong U(\g)/J_1$. Via the
Fourier transform, this also gives an alternative proof of Eastwood's
result \cite{ea05} on $\Scr(\dalembert_p)$.

  Now suppose that $r\geq 1$.
The algebra $R_r$ also carries a representation of $\g$ (see
Section~\ref{sec2}) and the main aim of this paper is to use this to
describe $\cD(R_r)$:

\begin{thm}[see Theorems~\ref{mainthm1} and~\ref{thm892}]
% \emph{[See Theorems~\ref{mainthm1} and~\ref{thm892}.]}
  \label{thmA} 
  There exist a primitive ideal $J_r \subset U(\g)$ and isomorphisms
  \[ 
U(\g)/ J_r \ \longisomto \ \cD(R_r) \ \longisomto \ \Sscr{\Delta_1^r}.
  \]
  \end{thm}
  
 \begin{rem}\label{Joseph-remark2}
(1)   The primitive ideal $J_r \subset U(\g)$ from
   Theorem~\ref{thmA} is also the annihilator of the
   $\g$-module $R_r$.  This representation of $\g$ is
   \emph{minimal} in the sense that the associated variety
   of $J_r$ is the closure of the minimal non-zero nilpotent
   orbit $\mathbf{O}_{\mathrm{min}}$ in $\g$. At least for
   $n>3$, the Joseph ideal $J_1$ is the unique completely
   prime ideal with this associated variety, which explains
   why it necessarily appears in this theorem. In contrast,
   for $r>1$ the ideal $J_r$ is not completely prime; in
   fact $U(\g)/J_r$ has
 Goldie rank $r$, in the sense that
   its simple artinian ring of fractions is an $r\times r$
   matrix ring over a division ring. 

(2) The ideal $K_r$ of Theorem~\ref{thmB} is the intersection  $K_r=J_r \cap U(\fso(p+1,q+1))$.
 
  (3) Results like Theorems~\ref{thmB} and \ref{thmA} are very
  sensitive to the precise operator $\Delta_1$ or
  $\dalembert_p$. For example, if $\Ft$ is replaced by
  $\vphi=\sum X_j^3\in A=\C[X_1,X_2,X_3]$ then
  $\cD\bigl(A/\vphi A\bigr)$
  is neither finitely generated nor noetherian, and even has an
  infinite ascending chain of ideals \cite{bgg}. By applying the
  Fourier transform $\cF$, the same properties hold for the
  algebra of symmetries $\Sscr{\Theta}$, for
  $\Theta=\sum_{j=1}^3 \partial^3_{X_j}\in A_3.$ 
\end{rem}

The idea behind the proof of Theorem~\ref{thmA} is to relate $\cD(R_r)$ to
$\cD(R_1)$ and hence to reduce this theorem to the known case
$U(\g)/J_1 \cong \cD(R_1)$. This is achieved by showing that the set of
differential operators $\cD(R_r,R_1)$ from $R_r$ to $R_1$ (which naturally
relates $\cD(R_r)$ to $\cD(R_1)$) equals the module of \emph{$\g$-finite
  vectors} $\cL(R_r,R_1)$ that connects $U(\g)/J_r$ to $U(\g)/J_1$. Showing
that, in fact, $\cD(R_r,R_1) = \cL(R_r,R_1)$ is the key step in the proof
of the theorem, since it reduces the problem to understanding algebras in
the category~$\mathcal{O}$ of highest weight $\g$-modules; see
Sections~\ref{Orings} and \ref{kfinite} in particular.  (We remark that,
here and elsewhere, we use the Lie-theoretic notation from \cite{lie1,
  lie2} and \cite{jantzen}, although the relevant terms are also defined in
the body of the paper.)  

Along the way we also obtain a detailed understanding of the $\g$-module
structure of $R_r$,  and this forms the next main result.

\begin{thm}[see Theorem~\ref{thm26}]
% \emph{[See Theorem~\ref{thm26}.]}
  \label{thmC}
  The $\g$-module $N(\lambda) = R_r$ is a highest weight module, with
  highest weight $\lambda = (\frac{n}{2}-r) \varpi_1$ where $\vpi_1$ is the
  first fundamental weight of $\g$. Furthermore:
  \begin{enumerate}[\rm (i)]
  \item if $n$ is even with $r< \frac{n}{2}$, or if $n$ is odd, then
    $N(\lambda)= L(\lambda)$ is irreducible;
  \item if $n$ is even with $r \ge \frac{n}{2}$, then $N(\lambda)$ has an
    irreducible socle $S$ isomorphic to $L(\mu)$, while the quotient
    $N(\lambda)/S = L(\lambda)$ is irreducible and finite dimensional. 
    (The formula for $\mu$ is given in~\eqref{thm26-1}.)
  \end{enumerate}
\end{thm}

   One significant consequence of this proof is that
\\
\centerline{\emph{$\cD(R_r)$ is a maximal order in its simple ring of
    fractions}} 
\\
(this concept is the natural noncommutative analogue of
being integrally closed; see also Definition~\ref{order-defn}).
  
The space of harmonic polynomials 
$H=\{p\in A : \Delta_1(p)=0\}$, and its real analogues, 
are fundamental objects  with many applications; see, for example,
\cite{hss, KO} for applications to  minimal representations of the conformal group
$\Ort(p+1,q+1)$ and \cite{bek08, bek11} for applications in physics.
 Given that we have been interested in symmetries of powers
 of the Laplacian, it is therefore natural to consider
 solutions   of the $r$-th 
power of the Laplacian $\Delta^r_1$.    In other words, we are interested in the
space of ``harmonics of level~$r$'', defined as
\begin{equation*} %\label{eq3}
  M_r= \bigl\{f \in A \; : \; \Delta_1^r(f)=0\bigr\}. 
\end{equation*}

This is the topic of Section~\ref{sec9} where we show that
$M_r$  is   a $\g$-module,    indeed  an  
  $\Sscr{\Delta_1^r}$-module, that is closely  related to the
$\cD(R_r)$-module $R_r$ through the isomorphism of Theorem~\ref{thmA}:
 
\begin{thm}[see Corollary~\ref{cor915}]
% \emph{[See Corollary~\ref{cor915}.]}
  \label{thmD}
  \begin{enumerate}[{\rm (i)}]
\item In the category $\cO$ of $\g$-modules, $M_r$ is
    isomorphic to the dual $N(\lambda)\spcheck$ of $N(\lambda)=R_r$.
  \item Consequently, if $n$ is even with $r< \frac{n}{2}$ or if $n$ is
    odd, then $M_r \cong N(\lambda) \cong L(\lambda)$ is simple.
  \item If $n$ is even with $r \ge \frac{n}{2}$, then $M_r$ has an
    irreducible finite dimensional socle $E \cong L(\lambda)$.
The quotient $M_r/E \cong L(\mu)$ is an    irreducible highest weight
module. 
  \end{enumerate}
\end{thm}

\begin{rem}\label{thmD-remark}
As was true for  Theorem~\ref{thmB}, there are also
analogues of this result for solutions of powers of the
d'Alembertian $\dalembert_p$ and these are described in
Corollary~\ref{cor916}. 
\end{rem}

One should observe that the action of the Lie algebra $\g$ on $R_r$ and
$M_r$ is not given by linear vector fields.  For instance, one needs
differential operators $P_j$ of order~$2$ for $R_r$ and their Fourier
transforms $\cF(P_j)$ for the action on $M_r$; see~\eqref{eq15} and
Theorem~\ref{thm8922}. This is similar to the action on the minimal
representation of $\fso(p+1,q+1)$ in its Schr\"odinger model, as described
in \cite{kobayashi}.

The $\g$-module $M_1$ of harmonic polynomials is also an
incarnation of the scalar singleton module introduced by Dirac through
the ambient method (\cite{dirac1,dirac2, eagr, bek11}).  
   In this approach one  studies the Laplacian on $\R^n$ via a
  conformal compactification of $\R^n$  which appears as a projective quadric
  $\mathcal{Q}= \{\Qt=0\} \subset \mathbb{RP}^{n+1}$
  together with the action of the 
  Laplacian from $\R^{n+2}$ on   densities of a particular weight.   
    Similar questions arise for densities on the ``generalised light cone''
  $\{\Qt^r =0\}$ and, by \cite{eagr} and \cite{gjms}, the ambient method also
  works here; in this case  for densities of weight $- \halfn +r$.  This produces an
  $\fso(p+1,q+1)$-module, which in the case of the Minkowski
  space time corresponds to
  the higher-order singleton as defined in \cite{bekgri}.

  A more detailed review of this technique is given in the final
  Section~\ref{howe}, where we   give an algebraic version of the
  ambient construction and relate it to the $\g$-modules $R_r$ and $M_r$
  described above. Roughly speaking, in our setting the generalised light
  cone $\{\Qt^r =0\}$ is replaced by a factor $B/\Qt^rB$ of a  
    polynomial ring $B$ in $(n+2)$ variables, equipped with an appropriate
  Laplacian $\Delta \in \cD(B)$, while the densities are replaced by
  homogeneous polynomials in a finite extension $S/\Qt^rS$ of $B/\Qt^rB$.

  As such, the Laplacian $\Delta$ acts on the space
  $(S/\Qt^r S)(-\halfn +r)$ of densities of weight $-\halfn +r$ and gives
  the $\g$-module of \emph{harmonic densities}
\[
  N_\lambda = \bigl\{\bar{f} \in (S/\Qt^r S)(\textstyle{-\halfn +r}) \, :
  \, \Delta(f) = 0\bigr\}
\]
(see Definition~\ref{def642} and Proposition~\ref{thm221}).  Under this
notation and using the ideas of the ambient method, we are able to relate
all the earlier constructions by proving:

\begin{thm}[see Corollary~\ref{summary}]
% \emph{[See Corollary~\ref{summary}.]}
  \label{thmE}  
  There are $\g$-module isomorphisms
  $N_\lambda \cong M_r \cong N(\lambda)\spcheck$.
\end{thm}

   We remark that
  one consequence of this result  is that the algebra of symmetries
  $\Sscr{\dalembert_{n-1}}$ is isomorphic to the ``on-shell
  higher-spin algebra'' of~\cite[\S3.1.3,
  Corollary~3]{bek08}. See Remark~\ref{rem922} for the
  details.             
  %%%%%%%%%%%%%%%%%%%%%%%%

%%%%%%%%%%%%%%%%%
\section{Notation}
\label{notation}

Fix an integer $n \ge 3$ and set $N=n+2$. In this section we fix some
notation about the complex simple Lie algebra $\g= \sontwo$.  The rank of
$\g$ will be denoted by $\rk\g =\ell=\ell'+1$; thus, $\g$ is of type
$\Bsf_\ell$ (with $\ell= \halfn+\half$) when $n$ is odd and of type
$\Dsf_\ell$ (with $\ell= \halfn+1$) when $n$ is even. Note that
$\g \cong \fsl(4,\C)$ when $n=4$.

A convenient presentation of $\g$ is given by derivations on the polynomial
algebra in $N$ variables.  Thus, let $U_{\pm 1},\dots,U_{\pm \ell}$ be
$2\ell$ indeterminates over $\C$. If $n$ is odd we let $U_0$ denote another
indeterminate and  set $U_0=0$ when $n$ is even; thus in each case
$B= \C[U_{\pm 1},\dots,U_{\pm \ell}, U_0]$ is a polynomial ring in $N$
variables.  Write $\ddu{j}=\frac{\partial}{\partial U_j}$ with the
convention that $\ddu{0}= 0$ if $n$ is even.

As in \cite[Corollaries 1.2.7 \&~1.2.9]{GW} we identify  
$\g$ with the Lie subalgebra of elements in $\fgl(N,\C)= \End \C^N$ which
preserve the quadratic form:
\begin{equation*} %\label{eqA1}
\Qt= \sum_{j=1}^\ell U_jU_{-j} + \half U_0^2.
\end{equation*}
Define    differential operators on $B$ by 
  \begin{equation}
    \label{eq610}
\Et =U_0\dpartial{U_0}
    + \sum_{j=1}^\ell (U_j \dpartial{U_j} + U_{-j}
    \dpartial{U_{-j}}) , \qquad \Delta 
    =\partial^{\; 2}_{U_0} + 2 \sum_{j=1}^\ell
    \dpartial{U_j}\dpartial{U_{-j}} 
  \end{equation}
  and derivations of $B$ by setting:
  \begin{equation}
    \label{eq610a}
      E_{ij}= U_i \ddu{j} - U_{-j} \ddu{-i} \qquad\text{for $i,j
      =0,\pm 1,\dots,\pm \ell$}. 
  \end{equation}
  Observe that $E_{-i,-j} = -E_{j,i}$, while $E_{-i,i} = 0$
  and $E_{ij}(\Qt)= 0$.  The following
  result is classical.

  \begin{prop}
    \label{prop611}     
 {\rm (1)} 
    The $\C$-vector space   spanned by the derivations
    $E_{ij}$ is a subalgebra of the Lie algebra
    $\left(\End_\C B, [\phantom{.},\phantom{.}]\right)$ isomorphic to the
    orthogonal Lie algebra $\fso(n+2,\C)$. 

\noindent {\rm (2)} 
The Lie subalgebra $\fsp$ of
    $\bigl(\End_\C B,[\phantom{.},{.}]\bigr)$ spanned by $\Delta$, $\Qt$
    and $-\bigl(\Et + \frac{N}{2}\bigr)$ is isomorphic to
    $\fsl(2,\C)$.  
  \end{prop}

  \begin{proof}
    Under the identification of $\mathfrak{gl}(N,\C)$
    with the space of vector fields spanned by
    $\{U_i\partial_{U_j}\}$,
    part~(1) follows from \cite[\S2.3.1, p.~70]{GW}, while part (2) is in
    \cite[\S6]{howe1}.
 \end{proof}

We can and  therefore will identify $\g=\fso(n+2,\C)$ with the Lie
  algebra spanned by  
  the $\{E_{ij}\}$. 
 Notice that each of the subalgebras $\fsp$ and $\g$ is
  contained in the commutant of the other; this is an
  infinitesimal version of the existence of the dual pair
  $(\mathrm{O}(n+2,\C), \mathrm{Sp}(2,\C))$, see
  \cite[\S6]{howe1} or \cite[\S3]{howe2}.

  The space $\h= \bigoplus_{j=1}^\ell \C E_{jj}$ is a Cartan
  subalgebra of $\g$ and we set
  $\vepsilon_j = E_{jj}^* \in \h^*$. Let $\Phi$ denote the
  set of roots of $\fh$ in $\g$ and write $\g^\alpha$ for
  the space of root vectors of weight $\alpha \in \fh^*$.
  The set $\Phi$ is then given by
\[
  \{\pm (\vepsilon_a \pm \vepsilon_b) \, : \, 1 \le a < b \le \ell\} \,
  \bigsqcup \, \{\pm \vepsilon_b \, : \, \ 1 \le b \le \ell\} \ \text{ when 
    $n$ is odd}
\]
and
$\{\pm (\vepsilon_a \pm \vepsilon_b) \, : \, 1 \le a < b \le \ell\} \text{
  when $n$ is even}$. We choose positive roots  by setting 
 \[
  \Phi^+ = \{\vepsilon_1 \pm \vepsilon_{b+1}\}_{1 \le b \le \ell'} \,
  \bigsqcup \, \{\vepsilon_{a+1} \pm \vepsilon_{b+1}\}_{1 \le a < b \le
    \ell'} \, \bigsqcup \, \{\vepsilon_1\} \, \bigsqcup \,
  \{\vepsilon_{b+1}\}_{1 \le b \le \ell'} \quad \text{if $n$ is odd}
\]
and
$\Phi^+ = \{\vepsilon_1 \pm \vepsilon_{b+1}\}_{1 \le b \le
  \ell'} \, \bigsqcup \, \{\vepsilon_{a+1} \pm
\vepsilon_{a+1}\}_{1 \le a < b \le \ell'}$
if $n$ is even.  For more details, see
\cite[\S2.3.1]{GW}. In general we will use \cite{lie1, lie2} as our
basic reference for Lie-theoretic concepts. In particular
the set $\Phi^+$ is taken from \cite[Planches~II et~IV]{lie1} and we fix
the same basis 
$\Bsf = \{\alpha_1,\dotsc,\alpha_\ell\}$ as described there.
Set $\fn^+ = \bigoplus_{\alpha \in \Phi^+} \g^\alpha$ and
$\fb^+ = \fh \boplus \fn^+$, which is a Borel subalgebra of
$\g$.  A more refined \emph{Chevalley system} in $\g$ is
then given in the next proposition, see~\cite[Chap.~VIII,
\S~13]{lie2} for further details.
 
   \begin{prop}
    \label{chevalley2}
    The following set $\{y_\alpha \in \g^\alpha\}_{\alpha \in \Phi}$ of
    root vectors is a Chevalley system in $\g$ (recall that
    $[y_\alpha,y_{-\alpha}] = h_\alpha$):
    \begin{enumerate}[{\rm (a)}]
    \item  $y_{\vepsilon_{a} - \vepsilon_{b}} = E_{ab}$, \ 
      $y_{\vepsilon_{a} + \vepsilon_{b}} = E_{a,-b}$,
      $1 \le a < b \le \ell$, while $y_{\vepsilon_{b}} = \sqrt{2}E_{b,0}$,
      $1 \le b \le \ell$, if $n$ is odd;
    \item  $y_{-(\vepsilon_{a} - \vepsilon_{b})} = E_{ba}$,\
      $y_{-(\vepsilon_{a} + \vepsilon_{b})} = E_{-b,a}$,
      $1 \le a < b \le \ell$, while $y_{-\vepsilon_{b}} = \sqrt{2}E_{0,b}$,
      $1 \le b \le \ell$, if $n$ is odd;
    \item  $h_{\vepsilon_j\pm \vepsilon_k} = E_{jj} \pm E_{kk}$,
      $1 \le j < k \le \ell$ while  $h_{\vepsilon_j} = 2E_{jj}$,
      $1 \le j \le \ell$ if $n$ is odd.
    \end{enumerate}
    In particular, $\g$ has a triangular decomposition
    $\g= \fn^- \boplus \fh \boplus \fn^+$ with   
    \[
        \fn^+= \bigoplus_{1 \le i < j \le \ell} \C E_{ij} \boplus \bigoplus_{0
      \le i < j \le \ell} \C E_{i,-j}, \ \; \text{and}\ \; \fn^-= \bigoplus_{1 \le i < j
      \le \ell} \C E_{ji} \boplus \bigoplus_{0 \le i < j \le \ell} \C
    E_{-j,i}. \eqno{\qed}
    \]
  \end{prop}

  The subspace $\fk \subset \Der_\C(A)$ generated by the derivations
  $E_{ij}$ with $i,j \in \{0,\pm 2,\dots,\pm \ell\}$ is a Lie subalgebra of
  $\g$ isomorphic to $\fso(n,\C)$ and is of type
  $\mathsf{B}_{\ell-1}=\mathsf{B}_{\ell'}$ ($n$ odd) or
  $\mathsf{D}_{\ell -1}=\mathsf{D}_{\ell'}$ ($n$ even). Furthermore,
  $\fm = \C E_{11} \boplus \fk \subset \g\cong \fso(2,\C) \times
  \fso(n,\C)$ and $\g$ decomposes as: 
  \begin{equation*}  %\label{r-plus}
    \g = \fr^- \boplus \fp, \quad\text{for}\quad  \fp = \fm \boplus \fr^+ \
    \text{and}\  \fk= 
    [\fm,\fm],
  \end{equation*}
  where $\fr^- = \bigoplus_{p \ne \pm 1} \C E_{p1} \subset \fn^-$ and
  $\fr^+ = \bigoplus_{p \ne \pm 1} \C E_{1p} \subset \fn^+$. Here
  $\fp$ is a maximal parabolic subalgebra of $\g$ with abelian nilradical
  $\fr^+ \cong \C^n$ and Levi subalgebra $\fm$; for more details, see
  \cite[\S 3]{LSS}. In the notation of \cite[Table~3.1]{LSS}, $\p$ is the
  parabolic $\fp_1$ (of types~$\mathsf{B}_\ell$, respectively $\mathsf{D}_\ell$
  when $n$ is odd, respectively even) except that,   when $n=4$, $\fp$ is
  the parabolic 
  $\fp_2$ of type~$\Asf_3$.

Let $\varpi_1,\dotsc,\varpi_\ell$ denote the fundamental weights of $\g$
and $\Psf = \bigoplus_{j=1}^\ell \Z \varpi_j$ the lattice of weights. Set
$\N^*=\N\smallsetminus\{0\}$ and 
$\Psf_{++}= \bigoplus_{j=1}^\ell \N^* \varpi_j \ \subset \ \Psf_{+} =
  \bigoplus_{j=1}^\ell \N \varpi_j$.
In order to quote results from \cite{jantzen}, we will often need to shift by
$\rho = \half \sum_{\alpha \in \Phi^+} \alpha = \sum_{j=1}^\ell \varpi_j$ and
we define
\begin{align*}
  \Psf^{++} &  = - \rho + \Psf_{++} = \{\mu \in \Psf : \dual{\mu +
              \rho}{\alpha\spcheck} > 0, \ \alpha \in \Bsf\}= \{\mu \in \Psf :
              \dual{\mu}{\alpha\spcheck} \ge 0, \ \alpha \in \Bsf \}
  \\
  \Psf^+ & = - \rho + \Psf_{+} = \{\mu \in \Psf : \dual{\mu +
           \rho}{\alpha\spcheck} \ge 0, \ \alpha \in \Bsf \}= \{\mu \in \Psf :
           \dual{\mu}{\alpha\spcheck} \ge -1, \ \alpha \in \Bsf\}.
\end{align*}

The definition of a Verma module $M(\mu)$ and its unique simple quotient
$L(\mu)$, for $\mu \in \h^*$, will be  relative to our given
triangular decomposition $\g = \fn^- \boplus \fh \boplus \fn^+$.
Recall that $L(\mu)$ is a finite dimensional simple module if and only if
$\mu \in \Psf^{++}$. Finally, we let $\mathcal{O}$ denote the category 
of highest weight modules as defined, for example, in \cite[4.3]{jantzen}.

%%%%%%%%%%%%%%%%%
 
%%%%%%%%%%%%%%%%%
\section{Differential operators}
\label{sec1}

We continue the discussion of the Lie algebra $\fso(n+2,\C)$ and its
presentation as differential operators, most especially those  on the ring
$\C[X_1,\dotsc,X_n ]/(\Ft^r)$ from the introduction.  It is worth
emphasising that, in contrast to the last section, here we start with a
polynomial ring in $n$ rather than $n+2$ variables. Both approaches will be
needed in the paper but the present approach is more subtle since
$\fso(n+2,\C)$ can no longer be presented as an algebra of derivations on
$\C[X_1,\dotsc,X_n ]$; one needs to add certain second order operators.  
Many of the
results in this section come from \cite{Usl3, kobayashi}.

For the moment, let $\K$ be any field of characteristic $0$ and write
$\calD(R)$ for the ring of $\K$-linear differential operators on a
commutative $\K$-algebra $R$, as defined, for example, in \cite{EGA}. More
generally, given $R$-modules $M$ and $N$, define
$ \calD(M,N) =\bigcup_{k\geq 0} \calD_k(M,N)$, where $\calD_{-1}(M,N)=0$
and, for $k\geq 0$,   
\begin{equation}\label{lem0.15}
  \calD_k(M,N)\ = \  \bigl\{\theta\in \Hom_{\C}(M,N) : [x,\theta]\in
  \calD_{k-1}(M,N) \ \text{for all } x\in R\bigr\}. 
\end{equation}
The elements of $\calD_k(M,N)$ are called \emph{differential operators of
  order at most $k$}.  Set $\calD(M)=\calD(M,M)$.

Fix an integer $n \ge 3$ and set $A = \K[X_1,\dotsc,X_n]$ for commuting
indeterminates $X_j$. Recall that $\cD(A)$ identifies with the $n$-th Weyl
algebra $A_n(\K)= \K[X_1,\dotsc,X_n,\ddx{1},\dotsc,\ddx{n}]$ where, as before, 
$\ddx{i} = \dpar{{\phantom{P}}}{X_i}$.  Given a factor ring $R= A/\fa$, the
\emph{idealiser} of $\fa\calD(A)$ is defined to be
\begin{equation}\label{eq10}
  \I_{\cD(A)}\bigl(\fa\cD(A)\bigr) = \{ P \in \cD(A) \, : \, P \fa\cD(A) \subseteq
  \fa\cD(A).  \}
\end{equation}
As noted  in \cite[Proposition~1.6]{SmSt}, this provides the
useful description of $\calD(R)$:
\begin{equation}\label{eq11}
  \cD(R) \, \cong \, 
  \I_{\cD(A)}\bigl(\fa\cD(A)\bigr) \big/\fa \cD(A).
\end{equation}
   
Let $r$ be a positive integer and set
\begin{equation}
  \label{eq103}
  \mathtt{F} = \sum_{i=1}^n X_i^2 \in A, \qquad \text{and}\qquad R= R_r =
  A/ \Ft^r A.
\end{equation}
Define the \emph{Laplacian} $\Delta_1 \in \cD(A)$, the \emph{Euler
  operator} $\mathtt{E}_1 \in \cD(A)$ and derivations $D_{ij}$ for
$1 \le i,j \le n$ by
\begin{equation}
  \label{eq13}
  \Delta_1 = \half \sum_{i=1}^n \partial_{X_i}^2, \qquad \Et_1 = \sum_{i=1}^n
  X_i\ddx{i}, \qquad D_{ij} = -D_{ji}= X_i \ddx{j} - X_j \ddx{i}.
\end{equation}
(Note the minor differences  between these operators and   the Euler
$\mathtt{E}$ and Laplacian $\Delta$ from \eqref{eq610}.) 
The next lemma follows from straightforward calculations; see for example
\cite[Proposition~1.2.2]{Usl3}.

\begin{lem}
  \label{lem14}
  For $1\leq i,j,k,l\leq n$ one has:
  \begin{enumerate}[{\rm (i)}]
  \item $[D_{kl},X_i] = \delta_{il}X_k - \delta_{ik}X_l$ while
    $[D_{ij}, \Ft] = D_{ij}(\Ft) = 0$;
  \item
    $ [D_{ij},D_{kl}]= \delta_{jk}D_{il} + \delta_{jl}D_{ki} +
    \delta_{ik}D_{lj} + \delta_{il}D_{jk}$. 
    \qed
  \end{enumerate}
\end{lem}

Let $d \in \K$ and define second order differential operators in $\cD(A)$
by
\begin{equation}
  \label{eq15}
  P_j(d)= X_j \Delta_1 - (\Et_1 + d) \ddx{j} = X_j \Delta_1 -
  \ddx{j}(\Et_1 + d-1), \quad\text{for}\
  \quad 1 \le j \le n.
\end{equation}

Direct computations (see, in particular, \cite[Proposition~1.2.2 and
Th\'eor\`eme~2.1.3]{Usl3} or \cite[Theorem~2.4.1]{kobayashi}) yield:

\begin{prop}
  \label{prop16}
  \textrm{\rm (1)} For $1\leq i,j,k\leq n$, the elements $P_j(d)$ satisfy:
   \begin{enumerate}[{\qquad\rm (a)}]
    \item $[\Et_1 + d,P_j(d)] = -P_j(d)$ and
      $[P_j(d),\ddx{k}]= \ddx{j}\ddx{k} -\delta_{j,k} \Delta_1$;
    \item $[X_i,P_j(d)] = D_{i,j} + \delta_{i,j} (\Et_1 +d)$, and
      $[D_{k,l},P_j(d)] = \delta_{l,j}P_k(d) - \delta_{k,j} P_l(d)$;
    \item $[P_i(d), P_j(d)] = 0$.
    \end{enumerate} 
  \noindent
  \textrm{\rm (2)} Assume that $\K$ is algebraically closed.  The subspace
  $\widetilde{\g}_d$ of $\bigl(\cD(A), [\phantom{.},\phantom{.}]\bigr)$
  spanned by the elements  
  $X_i$, $\Et_1 + d$, $D_{ij}$ and $P_j(d)$ (for $1\leq i,j\leq n$) is
  isomorphic to the Lie algebra $\fso(n+2,\K)$.
    
  \noindent 
  \textrm{\rm (3)} The subspace of $\gt_d$ spanned by the $D_{ij}$ is
  isomorphic to $\fso(n,\K)$.  \qed
\end{prop}

 \begin{Notation}\label{rem165} For the rest of  the section assume (for
   simplicity) that 
 $\K = \C$ and fix $r \in \N^*=\N\smallsetminus\{0\}$.  
Recall from Section~\ref{notation} that the (abstract)  Lie algebra
$\fso(n+2,\C)$ is denoted by $\g$.  Let  
$\psit= \psit_d   :   \g \to \widetilde{\g}_d\subset \calD(A)$
be an isomorphism defined by Proposition~\ref{prop16}(2); an explicit map
can be found in Corollary~\ref{isomorphism}.
\end{Notation}

We now want to choose $d \in \C$ so that
 $P_j(d)(\Ft^r A) \subseteq \Ft^r A$, and hence so that
 $P_j(d) \in \I_{\cD(A)}\bigl(\Ft^r\cD(A)\bigr)$.

 \begin{lem}
   \label{prop17}
   Let $d \in \C$ and $1 \le j \le n$. Then for any $k\geq 1$ and function
   $u$ of class $\calC^2$ on $\R^n$ or $\C^n$ we have
   \begin{equation}
     \label{eq18}
     P_j(d)(u \Ft^k)= P_j(2k +d)(u) \Ft^k + k\bigl(n -
     2(k+d)\bigr) X_ju \Ft^{k-1}.
   \end{equation} 
   Fix $r \in \N^*$ and set $d=d(r)= \halfn - r$. Then for any $u\in A$ one
   has:
   \begin{equation}
     \label{eq18-1}
     P_j(d)(u \Ft^r) = P_j(2r + d)(u) \Ft^r.
   \end{equation}
 \end{lem}

\begin{proof}
  We first show that, for $k \ge 0$,  
    \begin{equation}\label{prop17-1} 
    \Delta_1(u\Ft^k)= \Delta_1(u) \Ft^k + 2k \Et_1(u) \Ft^{k-1} + k (n +
    2(k-1)) u 
    \Ft^{k-1}.
  \end{equation}
  To see this note that
  $\partial_{X_j}(u\Ft^k) =\partial_{X_j}(u)\Ft^k +2kX_ju\Ft^{k-1}$ and
  hence that
  \[
    \partial^2_{X_j}(u\Ft^k) =\partial^2_{X_j}(u)\Ft^k
    +2kX_j\partial_{X_j}(u)\Ft^{k-1} +
    2k\left(u\Ft^{k-1}+X_j\partial_{X_j}(u)\Ft^{k-1}+2(k-1)X^2_ju
      \Ft^{k-2}\right).  
\]
Equation~\eqref{prop17-1} follows by summing over $j$ and~\eqref{eq18} then
follows by an elementary computation.  Taking $d= \halfn -r$ and $k=r$ in
\eqref{eq18} immediately gives \eqref{eq18-1}.
\end{proof}

\begin{cor}
  \label{cor19}
  Assume that $r \in \N^*$ and set $d=\halfn - r$. Set $R=R_r= A/ (\Ft^r)$
  and let $\widetilde{\g}_d$ the Lie subalgebra of $\cD(A)$ defined by
  Proposition~\ref{prop16}. Then the elements of $\widetilde{\g}_d$ induce
  differential operators on $R_r$. This therefore defines an algebra
  morphism
  $ \psi_r : U(\g) \lto \cD(R)$.
\end{cor}

\begin{proof}
  Set $\cB= \I_{\cD(A)}\bigl(\Ft^r\cD(A)\bigr)$ in the notation of~\eqref{eq10}.
  Combining Lemma~\ref{lem14} and Lemma~\ref{prop17} shows that
  $\widetilde{\g}_d\subset \cB$; equivalently, by \eqref{eq11}, the
  elements of $\widetilde{\g}_d$ induce differential operators on
  $R_r$. Now apply Proposition~\ref{prop16}(2).
\end{proof}

\begin{rem}
  \label{rem19a} Write $\Ker \psi_r = J_r $; thus
  $U(\g)/J_r \longisomto\  \Im(\psi_r) \subset \cD(R)$.  It is easily
  seen that $\psi_r(\xi) \ne 0$ for all $\xi \in \g$. We therefore can, and
  frequently will, identify $\g$ with $\g_r = \psi_r(\g) \subset
  \cD(R)$.
  In this case the morphism constructed in Corollary~\ref{cor19} will be
  denoted $\psi_r : U(\g) \to \cD(R)$.  
\end{rem}

%%%%%%%%%%%%%%%%%%%%%%%%%%%%%%
 
\section{The  $\sontwo$-module structure of
   $A/(\Ft^r)$}
\label{sec2}

We keep the notation of the last section, especially \eqref{eq103},
Notation~\ref{rem165} and Remark~\ref{rem19a}. Set $ d=\halfn - r$, where
$n\geq 3$ and $r\geq 1$ and write $P_j = P_j(d)$ in the notation of
\eqref{eq15}.  Clearly $R$ is a module over $\g=\sontwo$ via the morphism
$\psi_r$ and the aim of the section is to examine this module structure.
In particular we show that $R$ is either a simple $\g$-module or has a
unique finite dimensional factor module; the latter occurring if and only
if $n$ is even with $d\leq 0$.  This proves Theorem~\ref{thmD} from the introduction.

We first  show that the $\g$-module $R$
is a highest weight module for our choice of Cartan and Borel subalgebras
of $\g$. The action of an element $u \in U(\g)$ on $f \in R$ will be
denoted by $u.f = \psi_r(u)(f)$.  
For the rest of this section we will identify $\g$ with $\g_r
\subset \cD(R)$ through $\psi_r$, or, equivalently, $\g$ with $\gt_d
\subset \cD(A)$ through $\psit_d$. 
 We begin by making this identification more precise (see 
  Lemma~\ref{chevalley1} and Corollary~\ref{isomorphism}).   
    
  Let $i = \rone \in \C$ and define elements of $\gt_d$ by
\[I(d)= \Et_1 + d,\qquad 
  h_1= -I(d), \quad\text{ and }\quad h_{j+1} = i D_{j,j+\ell'}, \ \; 1 \le j \le \ell'.
\]
Thus $\bigoplus_{j=1}^\ell \C h_j \equiv \h$ is a Cartan subalgebra of
$\gt_d\equiv \g$.   

Next, observe that the approaches of the last two sections are related by
the changes of variables 
\begin{equation}
  \label{nroots}
  U_{a+1}= X_{a} - i X_{a+\ell'}, \ \; U_{-(a+1)} = X_{a} + i
  X_{a+\ell'}, \ \text{for}\  1 \le a \le \ell'; \ \text{and}\ \text{$U_0 = \sqrt{2}X_n$
    when $n$ is odd.}
\end{equation}
In this notation,  $\Ft= \sum_{j=2}^{\ell}U_jU_{-j} + \half U_0^2$, and
$\Delta_1 = 2 \sum_{j=2}^{\ell} \ddu{j}{\ddu{-j}} + \ddu{0}^2$ while
$\Et_1 = U_0\ddu{0} + \sum_{\pm k =2}^\ell U_k\ddu{k}$. 
As $\fr^- = \bigoplus_{j=1}^n \C X_j$ we can and will identify $A=\C[X_1,\dots, X_n]$ with
$S(\fr^-)$, or equivalently with the algebra  
of polynomial functions on $\fr^+$.

We will also need to define a  Chevalley system
  inside $\gt_d$.  Set:
\begin{equation*}
 % \label{proots-1}
  V_{a+1}= \half(P_a- iP_{a+\ell'}), \ \; V_{-(a+1)} =   \half(P_a +
  iP_{a+\ell'}),  \ \text{for}\ 1 \le a \le \ell'; \  \text{and}\ \text{$V_0 =
    \frac{1}{\sqrt{2}}P_n$  when $n$ is odd.}
\end{equation*} 
Observe that $\boplus_j \C P_j = \boplus_j \C V_{j}$ and it  is easy to
check that  
\begin{equation}
  \label{proots-2}
  V_j= \half U_j \Delta_1 - I(d) \ddu{-j}, \ \  \text{for}\ j \in \{0, \pm 2,\dots, \pm \ell\}.
\end{equation}
The next result follows from straightforward computations.

\begin{lem}
  \label{chevalley1}
  The following set $\{x_\alpha\}_{\alpha \in \Phi}$ of root vectors is a
  Chevalley system in $\gt_d$ or $\g_r$ (recall that
  $[x_\alpha,x_{-\alpha}] = h_\alpha$):
  \begin{enumerate}[{\rm (i)}]
  \item 
    $x_{\vepsilon_{a+1} \pm \vepsilon_{b+1}} = U_{a+1}\ddu{\mp(b+1)} -
    U_{\pm(b+1)}\ddu{-(a+1)}$, $1 \le a < b \le \ell'$;
  \item 
    $x_{-(\vepsilon_{a+1} \pm \vepsilon_{b+1})} = U_{\mp(b+1)}\ddu{a+1} -
    U_{-(a+1)}\ddu{\pm(b+1)}$, $1 \le a < b \le \ell'$;

  \item  $x_{\vepsilon_{1} \pm \vepsilon_{a+1}} = V_{\pm(a+1)}$,
$1 \le a \le \ell'$,   while 
    $x_{\vepsilon_{1}} = \sqrt{2}V_0$ if $n$ is odd;
    
  \item  $x_{-(\vepsilon_{1} \pm \vepsilon_{a+1})} = U_{\mp(a+1)}$, 
      $1 \le a \le \ell'$, while  $x_{-\vepsilon_{1}} = \sqrt{2}U_0$ if $n$ is odd;
      \item   $x_{\vepsilon_{b+1}} = \sqrt{2}\bigl(U_{b+1}\ddu{0} - U_0
    \ddu{-(b+1)}\bigr)$,
      $x_{-\vepsilon_{b+1}} = \sqrt{2}\bigl(U_{0}\ddu{b+1} - U_{-(b+1)}
    \ddu{0}\bigr)$, $1 \le b \le \ell'$, if $n$ is odd;
      \item  Then $h_{\vepsilon_j\pm \vepsilon_k} = h_j \pm h_k$,
    $1 \le j < k \le \ell$, while  $h_{\vepsilon_j} = 2h_j$, $1 \le j \le \ell$ if
    $n$ is odd.  
  \end{enumerate}
In particular, one can identify  $\fr^+$ with $\boplus_j \C P_j = \boplus_j
\C V_{j}$.  \qed
\end{lem}

 \begin{cor}
    \label{isomorphism}
    Retain the notation of Proposition~\ref{chevalley2} and
    Lemma~\ref{chevalley1}.  The map $\psit_d$ from $\g \subset \cD(B)$
    onto $\gt_d \subset \cD(A)$, given by $\psit_d(y_\alpha) = x_\alpha$
    for $\alpha \in \Phi$, is an isomorphism of Lie algebras.      \qed
  \end{cor}

  \begin{rem}
    \label{involutions}
Recall that a Chevalley system in a semi-simple Lie algebra defines a
Chevalley anti-involution \cite[2.1]{jantzen}. Let $\vt$, respectively $\vt_d$, be the
anti-involution of $\g$, respectively $\gt_d$ defined by $\{y_\alpha\}_\alpha$,
respectively $\{x_\alpha\}_\alpha$. Then, by definition, the isomorphism
$\psit_d$ satisfies 
$\psit_d \circ \vt = \vt_d \circ \psit_d$.   
  \end{rem}

Let $G = \SO(n+2,\C) \supset M$ be connected algebraic groups such that
$ \Lie(G) = \g\supset \Lie(M) = \fm$, with $G$ acting on $\g$ through the
adjoint action.  The (reduced) quadratic cone $Z=\{\Ft = 0\} \subset \C^n$
can be identified with the closure of the $M$-orbit of a highest weight
vector $x_{\tilde{\alpha}} \in \fr^+\equiv \C^n$ and this then identifies
$R_1 = A/\Ft A$ with the algebra of regular functions on $Z$.  Recall that the orbit
$\mathbf{O}_{\mathrm{min}}= G.x_{\tilde{\alpha}}$ is the nonzero nilpotent
orbit of minimal dimension and that
$\half \dim \mathbf{O}_{\mathrm{min}} = \dim M.x_{\tilde{\alpha}} = n-1$.

Recall that the Verma module $M(\mu)$ with
highest weight $\mu \in \fh^*$ has a unique simple quotient $L(\mu)$ and
set $I(\mu)= \ann_{U(\g)} L(\mu)$. The associated variety,
$\calV(J) \subset \g$, of $J=I(\mu)$ is the closure of a nilpotent orbit,
cf.~\cite{jo85}. 
Let $W$ be the associated Weyl group of $\Phi$ and denote by
$w\cdot \lambda = w(\lambda + \rho) -\rho$ the ``dot'' action of $w \in W$
on $\lambda \in \h^*$, cf.~\cite[2.3]{jantzen}.

The starting point for this paper is the following result from \cite{LSS}, see
also \cite{Usl3}, which gives a construction of the Joseph ideal.
 (See \cite{jo881} for another proof of this theorem.)

\begin{thm}
  \label{thm22}
  The ideal $J_1=\Ker \psi_1$ is a completely prime maximal ideal such that
  $\calV(J_1)= \overline{\mathbf{O}}_{\mathrm{min}}$.  Moreover,
  $U(\g)/J_1 \cong \cD(R_1)$.
\end{thm}

Recall \cite{jo76} that if $\mathfrak{s}$ is a complex simple Lie algebra
not of type $\Asf_\ell$, there exits a unique completely prime ideal
$J \subset U(\mathfrak{s})$ such that
$\calV(J)= \overline{\mathbf{O}}_{\mathrm{min}}$. The primitive ideal $J$ is
called the \emph{Joseph ideal}. Thus, when $n \ne 4$, the ideal $J_1$ is
the Joseph ideal. When $n=4$, i.e.~$\g$ is of type $\Dsf_3=\Asf_3$, by a
slight abuse of notation, we will still call $J_1$ the Joseph ideal.

\begin{rem}\label{Joseph-remark} The ideal $K_1$ from Theorem~\ref{thmB}  is the
  intersection of $J_1$ with $U(\fso(p+1,q+1))$.  Eastwood
  also notes this fact, but his argument relies on an
  explicit set of generators for the Joseph ideal (see
  \cite{ESS, stromb}). In contrast, in Theorem~\ref{thm22}
  this is almost automatic; it follows almost immediately
  from the fact that $R_1$ has Krull dimension $(n-1)$; see
  for example \cite[Proposition~3.5]{LSS}.
\end{rem}

In the general case, $r \ge 1$, we begin with an easy lemma.
 
\begin{lem}
  \label{lem23}
  The $\g$-module $R=R_r$ is a highest weight module with highest weight
  $\lambda = - d\varpi_1 = -d\vepsilon_1$.
\end{lem}

\begin{rem} \label{rem23}
We typically write $R=N(\lambda)$ when thinking of
$R$ as a $\g$-module.
\end{rem}

\begin{proof}
  Since $R \subset \Im(\psi_r)\subset\calD(R)$, one has $R= U(\g).1$. By
  construction, $D_{s,t}.1 = P_j.1 = 0$ and so the choice of Chevalley
  system in Lemma~\ref{chevalley1} implies that $\fn^+.1 = 0$. Moreover,
  $h_j.1 = 0$ for $2 \le j \le \ell$ whereas
  $h_1.1 = -I(d).1 = -d.1= -d\vepsilon_1(h_1).1$.  The choice
  of $\Bsf \subset \Phi^+$ yields $\varpi_1= \vepsilon_1$, see
  \cite[Planche~II, pp. 252-253, and Planche~IV, pp. 256-257]{lie1}. Hence
  $R$ is a highest weight module with highest weight
  $\lambda= - d\varpi_1$.
\end{proof}

We need to analyze the structure of the $\g$-module
$N(\lambda)$ in more detail for which we need some notation. Grade
$A=\bigoplus_{m\geq 0}A(m)$ by total polynomial degree (denoted by
$\deg$). Let $H = \{p \in A : \Delta_1(p) = 0\}$ be the space of harmonic
polynomials. It is graded and we set $H(m)= A(m) \cap H$. The algebraic
group $K= \SO(\C^n,\Ft) \cong \SO(n,\C)$ acts naturally on $A$, through the
identification of  $A$ 
with the ring of regular functions of the standard representation $\C^n$ of
$K$. The differential of this $K$-action is given by the natural
representation of   
\[
  \fk \cong \bigoplus_{1 \le s < t \le n} \C D_{s,t} \cong \fso(n,\C) =
  \fso(2\ell',\C).
\]
It is clear that that $\C \Ft^k$ is the trivial $\mathfrak{k}$-module for
all $k$ and that each $H(m)$ is a $\mathfrak{k}$-module.  Notice that
$\h'= \bigoplus_{j=2}^\ell \C h_j$ is a Cartan subalgebra of $\fk$ and set
$\fn'= \bigoplus_{1 \le a < b \le \ell'} \g^{\vepsilon_{a+1} \pm
  \vepsilon_{b+1}} \susbet \fk$.
Then, $\fb'= \h' \boplus \fn'$ is a Borel subalgebra of $\mathfrak{k}$.
The following result is classical; see, for example, \cite[\S 5.2.3]{GW}.

\begin{thm}
  \label{thm24}
  Each space $H(k)$ is an irreducible $\fk$-module, of highest weight
  $k \varpi_1$ (in Bourbaki's notation for type
  $\mathsf{B}_{\frac{n-1}{2}}$ or $\mathsf{D}_{\halfn}$). The
  $\mathfrak{k}$-module $A(m)$ decomposes as the direct sum of irreducible
  modules:
  \[
    A(m) = \sum_{k = 0}^{[m/2]} H(m -2k) \Ft^k \cong \bigoplus_{k =
      0}^{[m/2]} H(m -2k) \otimes \C\Ft^k.
  \]
  Let $B$ be the polar form of $\Ft$ and $\xi \in \C^n$ be a non zero
  isotropic vector with respect to $\Ft$. If $\xi^k \in A(k)$ is the
  function given by $v \mapsto B(v,\xi)^k$, then
  $H(k) = U(\mathfrak{k}).\xi^{k}$.  \qed
\end{thm}

If $(v_1,\dotsc,v_n)$ is the canonical basis of $\C^n$, the vector 
$\xi = v_1 - i \, v_\ell$ is isotropic and the associated polynomial
function $\xi^k$ is simply $U_2^k =(X_1 - i \,X_\ell)^k$. 
As observed in Lemma~\ref{chevalley1}, $\xi$ considered as an element of
$\g$ has weight $-\vepsilon_1 +\vepsilon_2$ under the adjoint action of
$\g$. It follows from the choice of the Borel subalgebra
$\fb' = \h' \boplus \fn'$ of $\fk$ that $\xi^k \in H(k)$ is a highest
weight vector.  We will choose this generator for the simple finite
dimensional $\mathfrak{k}$-module $H(k) \cong L(k \varpi_1')$ (where
$\varpi_1' = {\vepsilon_2}_{\mid \h'}$ is the first fundamental weight for
$\mathfrak{k}$).  Observe the following consequence of Theorem~\ref{thm24}.

\begin{cor}
  \label{lem24b}
  Let $0 \le v \le t$ and $p = \sum_{j=v}^t p_j \Ft^j \in A(m)$, where
  $p_j \in H(m-2j)$ for each $j$. Then, for each $k \in \{v,\dotsc,t\}$,
  there exists $u \in U(\mathfrak{k})$ such that $p_k \Ft^k = u.p$.  
\end{cor}

\begin{proof}
  Since $\bigoplus_{j=v}^t H(m-2j) \Ft^j$ is a multiplicity free,
  semi-simple $U(\mathfrak{k})$-module, the claim follows from the
  Jacobson's Density Theorem, see~\cite[\S~4, \no 2, Corollaire~2]{alg}. 
\end{proof}

Let $0 \ne p \in A(m)$. By Theorem~\ref{thm24} we may uniquely write
$p = \sum_{k =v}^t p_k \Ft^k$ with $p_k \in H(m-2k)$, $v \le t$ and
$p_t\ne 0\ne p_v$.  The integer $v = v(p)$ will be called the
\emph{harmonic valuation} of $p$.  Let $P_j(d')$ be as in~\eqref{eq18}
(with $d' \in \C$ arbitrary). For $a \in H(m-2k)$ one has
$P_j(d')(a) = -(k+d')\ddx{j}(a) \in H(m-2k-1)$ (with the convention that
$H(j)= 0$ when $j<0$). Also,
$\Delta_1(X_ja) = [\Delta_1,X_j](a) + X_j\Delta_1(a) = \ddx{j}(a)$.
Therefore, by Lemma~\ref{prop17} with $d=\halfn - r$, we obtain:  
\[
P_j(d)(a\Ft^k) =
P_j(2k+d)(a) \Ft^k + k (n-2k -2d) X_j a \Ft^{k-1}
\]
with $P_j(2k+d)(a) \in H(m-2k-1)$ and $\Delta_1(X_ja) = \ddx{j}(a)$. In
particular, this gives:

\begin{lem} \label{eq25} {\rm (1)} If $\ddx{j}(a) = 0$, for $a\in A$, then
  $P_j(d)(a\Ft^k) = P_j(2k+d)(a) \Ft^k + k (n-2k -2d) X_j a \Ft^{k-1} $.
 If in addition   $a \in H(m-2k)$,  then
  $P_j(2k+d)(a) \in H(m-2k-1)$ and $X_ja \in H(m-2k+1)$.
\\    
  {\rm (2)} If $a= \xi^{m-2k}$ with $\xi = v_1 - i v_{\ell}$, and
  $j \ne 1,\ell$, then the condition $\ddx{j}(a) = 0$ is satisfied. \qed
\end{lem}

\medskip

The structure of the $\g$-module $N(\lambda)=R_r$ is given by the next
result. Recall that $\lambda = -d \varpi_1$.

\begin{thm}
  \label{thm26} {\rm (1)} Assume that either   $n$ is even
  with $r < \halfn$ or $n$ is odd. Then, $R_r=N(\lambda) \cong L(\lambda)$
  is a simple 
  $\g$-module.

  \noindent {\rm (2)} Assume that $n$ is even with $r \ge \halfn$; thus
  $d=\halfn-r\leq 0$.  Then, as a $\g$-module, $R_r = N(\lambda)$ has a
  simple socle $\socRr\cong L(\mu)$, where
  \begin{equation}\label{thm26-1}
    \mu = \lambda +(d-1)\alpha_1 =
    \begin{cases}
      (d-2)\varpi_1 + (1-d)\varpi_2 \ &\text{when $\ell \ge 4$,}
      \\
      (d-2)\varpi_1 + (1-d)(\varpi_2+\varpi_3) \ &\text{when $\ell = 3$.}
    \end{cases}
  \end{equation}
  The quotient $N(\lambda)/\socRr \cong L(\lambda)$ is an irreducible finite
  dimensional $\g$-module, isomorphic to the module of harmonic polynomials
  of degree $(-d)$ in $n+2$ variables. Also, $\socRr$ is the ideal of $R_r$
  generated by $H(1-d)$.
\end{thm}

\begin{proof} Let $0 \ne \overline{M} \subseteq N(\lambda)$ be a
  $\g$-submodule. Since $R_r \susbet \Im(\psi_r)$, the module
  $\overline{M}$ is also an ideal of $R_r$. Therefore
  $\overline{M}=M/(\Ft^r)$, for some  $\gt_d$-stable ideal $M$ of $A$.  In
  particular, by Proposition~\ref{prop16}, $(\Et_1 +d)(M) \subseteq M$,
  whence $\Et_1(M) \subseteq M$, and so $M = \bigoplus_m M \cap A(m)$ is
  homogeneous.   
    
  Let $p \in M \cap A(m)$, with $p \notin (\Ft^r)$, and write
  $p = \sum_{k=v}^t p_k \Ft^k$ for $p_k \in H(m-2k)$ with
  $p_v \ne 0\ne p_t$. Note that $ t < r$. We want to simplify our choice of
  $p$.  Recall that $\g\supset \fk \cong \fso(n,\C)$.  Thus, by
  Corollary~\ref{lem24b}, $ p_k \Ft^k \in M \cap H(m-2k) \Ft^k$ for all $k$ and
  so, as $H(m-2k) \Ft^k$ is an irreducible $\mathfrak{k}$-module,
  $H(m-2k) \Ft^k \subset M$ for all $k$ such that $p_k \ne 0$. In
  particular, $ H(m-2v) \Ft^v \subset M$ and we can replace $p$ by
  $p_v\Ft^v$.  There exists $u \in U(\fk)$ such that
  $u.p_v = \xi^{m-2v} =(X_1 - i X_\ell)^{m-2v}$. From $D_{s,t}.\Ft^k = 0$
  we then deduce $u.p = (u.p_v) \Ft^v \in M$ and, replacing $p$ by $u.p$,
  we may therefore assume that $p_v = \xi^{m-2v}$. Finally, we assume that
  the harmonic valuation $v=v(p)$ is as small as possible among such
  choices of $p$.

  Now apply the operator $P_{\ell +1}= P_{\ell +1}(d)$ to $p$; from
  Lemma~\ref{eq25} and $\ddx{\ell+1}(p_v) = 0$ we obtain
  \begin{equation}
    \label{eq27}
    P_{\ell +1}.p  \ = \  (P_{\ell+1}(2v +d).p_v) \Ft^v + v(n - 2v -2d)
    X_{\ell +1} p_v \Ft^{v-1}
  \end{equation}
  where $P_{\ell+1}(2v+d).p_v \in H(m-2v-1)$ and
  $X_{\ell+1}p_v \in H(m-2v+1)$.  Since
  $ P_{\ell +1}.p \in M \moins (\Ft^r)$ and $X_{\ell +1} p_v \ne 0$, the
  minimality of $v$, combined with~\eqref{eq27}, forces
  $v(n - 2v -2d) = 0$. In other words, either $v=0$ or $n - 2v -2d = 0$.
  The latter is equivalent to $v= r$, which is excluded. Thus we have
  proved that $M$ contains a non zero element $p =p_0 \in H(m)$.  In
  particular, $H(m)A \subseteq M$.
  
  Now pick $0 \ne p= \xi^{m} \in M \cap H(m)$ with $m $ minimal. As $p$ is
  homogeneous, if $m =0$ then $1 \in M$ and $\overline{M} = N(\lambda)$.
  Otherwise $ m \ge 1$.  In this case, applying
  $P_1 = X_1 \Delta_1 - (\Et_1 +d) \ddx{1}$ gives
  \begin{equation*}
    P_1.\xi^{m}  =-m(\Et_1+d)(\xi^{m-1}) =  -m(m -1+d) \xi^{m-1} \ \in\ M
    \cap H(m-1). 
  \end{equation*}
  The minimality of $m$ then implies that $m( m -1+d) = 0$, whence
  $m -1 + d =0$.  Thus, we have shown that $\overline{M} = N(\lambda)$
  unless $d$ satisfies $d=-m +1$ for some $m \ge 1$. In the latter case
  $0\geq d = \halfn - r$.  Since $r,m\in\mathbb{N}$, this is equivalent to
  $r$ being even with $r \ge \halfn$.  This proves (1).

  It remains to prove (2), where $r= \halfn + m -1$ for $m=1-d \ge 1$. Note
  that, in this case, we have shown that every non-zero $\g$-submodule
  $\overline{M}$ of $R_r=N(\lambda)$ contains the image $S$ of the ideal
  $\mathfrak{J}=H(m)A$.  
     
  Since $ \lambda = -d\varpi_1 = (m-1) \varpi_1 \in \Psf^{++}$ is a
  dominant integral weight, $L(\lambda)$ is finite dimensional. Moreover,
  it is isomorphic to the $\g$-module of harmonic polynomials of degree
  $m-1$ in $n+2$ variables (apply Theorem~\ref{thm24} to $\fso(n+2,\C)$).
  
  Since $D_{s,t}(H(m)) \subseteq H(m)$ and $\mathfrak{J}$ is homogeneous,
  certainly $\Theta. \mathfrak{J} \subseteq \mathfrak{J}$ for
  $\Theta = D_{s,t}, X_j$ and $(\Et_1 + d)$.  Thus, in order to show that
  $\mathfrak{J}$ is a $\g$-module, it remains to prove that
  $P_j.\mathfrak{J} \subseteq \mathfrak{J}$ or, equivalently, that
  $P_j.X^\alpha f \in \mathfrak{J}$ for any $f \in H(m)$ and monomial
  $X^\alpha = X_1^{\alpha_1}\dotsm X_n^{\alpha_n}$. We argue by induction
  on $|\alpha| = \sum_i \alpha_i$. If $|\alpha| = 0$, we have
  \begin{equation}
    \label{eq275}
    P_j.f = -(m-1 + d)\ddx{j}(f) = 0
  \end{equation} 
  by the choice of $m$. If $|\alpha| \ge 1$ write $X^\alpha =X_kX^\beta$,
  for $|\beta| = |\alpha| -1$. Then
  \[
    P_j.X^\alpha f= [P_j,X_k].X^\beta f + X_k P_j.X^\beta f.
  \]
  By induction, $P_j.X^\beta f \in \mathfrak{J}$, while
  $[P_j,X_k] = D_{j,k} - \delta_{j,k}(\Et_1 +d)$ from
  Proposition~\ref{prop16}(b), which implies that
  $[P_j,X_k].X^\beta f \in \mathfrak{J}$. Hence
  $P_j.X^\alpha f \in \mathfrak{J}$ and $\mathfrak{J}$ is indeed a
  $\g$-module.  Since $H(m)\not=A$ and $H(m)\not\subseteq (\Ft^r)$, the
  image $S$ of $\mathfrak{J}$ in $R_r$ is a nontrivial $\g$-submodule of
  $R_r=N(\lambda)$.  However, we have already noted that $S$ is contained
  in \emph{every} non-zero $\g$-submodule
  $\overline{M}\subseteq N(\lambda)$; in other words $S = \Soc(N(\lambda))=\socRr$.

  Recall that $H(m)$ contains all the functions $\zeta^{m}$ associated to
  the isotropic vectors $\zeta \in \C^n$. Since $\C^n$ has a basis
  $(\zeta_j)_j$ of such vectors, the ideal $\sum_{j=1}^n \zeta^mA$ has
  finite codimension in $A$. Therefore,
  $\dim R_r/\socRr\leq \dim A/H(m)A < \infty$.  Thus, on the one hand,
  $N(\lambda)/\socRr$ is finite dimensional and hence completely reducible, but
  on the other hand, as it is a factor of the Verma module $M(\lambda)$, it
  has a unique simple quotient.  This forces
  $N(\lambda)/\socRr \cong L(\lambda)$.

  Finally, we need to compute the highest weight of $\socRr$. Recall that the
  $\fk$-module $H(m)$ is generated by the function $\xi^m$ associated to
  the isotropic vector $\xi=v_1 - i v_\ell \in \C^n$. We have shown in
  \eqref{eq275} that $P_j.\xi^m = 0$ for all $j$ and hence
  $\fr^+.\xi^m = 0$.  Since $\xi^m = U_2^m$, Lemma~\ref{chevalley1} shows
  that $x_{\vepsilon_{a+1} \pm \vepsilon_{b+1}}.\xi^m = 0$ for
  $1 \le a < b \le \ell'$. Hence, $\fn'.\xi^m =0$ (recall that $n$ is even)
  and so that $\fn^+.\xi^m = 0$.  
  Moreover, from our choice of $h_j = iD_{j,j+\ell'}$ one obtains
  \[
    h_1.\xi^m = -\xi^m, \quad h_2.\xi^m = m \xi^m \quad\text{and}\quad
    h_j.\xi^m = 0 \ \; 
    \text{for $j \ge 3$}.
  \]
  Thus $\xi^m \in \socRr$ is a highest weight vector, with weight
  $\mu = -\vepsilon_1 + m \vepsilon_2$.
    Hence $\socRr \cong L(\mu)$. The fact that $\mu$ equals $\lambda -m\alpha_1$
  and satisfies \eqref{thm26-1} are easy exercises using
  \cite[Planche~IV]{lie1}.  
\end{proof}

Recall that $s_{\alpha_1} \in W$ is the
transposition $(1,2)$. In the notation of Theorem~\ref{thm26}, an easy
calculation shows that
$ \mu = \lambda +(d-1)\alpha_1 = s_{\alpha_1}\cdot \lambda.  $ We can
therefore unify the cases in that result by defining
\begin{equation*}
 % \label{eq210}
  \omega =
  \begin{cases}
    s_{\alpha_1}\cdot \lambda = \lambda - (r-\halfn+1) \alpha_1& \ \text{if
      $n$ is even
      with $r \geq \halfn $}\\
    \lambda & \ \text{otherwise.}
  \end{cases}
\end{equation*}

Write $\GKdim M= \GKdim_U M$ for the \emph{Gelfand-Kirillov dimension} of a
module $M$ over an algebra $U$, whenever it is defined, see
\cite[\S8.1.11]{MR} for more details.

\begin{cor}
  \label{cor211} {\rm (1)} The ideal $J_r$ is equal to
  $I(\omega) = \ann_{U(\g)} L(\omega)$.
  \\
  {\rm (2)} The associated variety of $J_r$ is $\calV(J_r) = \Ominbar$.
\end{cor}

\begin{proof} 
  (1) By Theorem~\ref{thm26}, $\socRr=\Soc N(\lambda)\cong L(\omega)$ while
  $N(\lambda)/\socRr$ is finite dimensional (possibly zero). Thus, by
  \cite[Lemma~8.14]{jantzen}, $\ann_{U(\g)} \socRr = \ann_{U(\g)}N(\lambda)$;
  equivalently $J_r= I(\lambda) = I(\omega)$.

  (2) Since $\GKdim L(\omega) = \GKdim R_r= n-1$,
  \cite[Satz~10.9]{jantzen}  implies that
  $\GKdim U(\g)/J(\omega) = 2(n-1)$. Thus $\calV(J_r) = \overline{\mbO}$
  for a nilpotent orbit $\mbO$ with $\dim \mbO= 2(n-1)$. This forces
  $\mbO= \Omin$.
\end{proof}

The ideal $J_r$ contains the kernel of the character $\chi_\lambda$ in the
centre of~$U(\g)$.  Since $\lambda = -d\vpi_1 \in \Q\vpi_1$,  there  
exists a unique weight $\nu$ such that $\nu \in W(\lambda+\rho)$, and
$\dual{\nu}{\alpha\spcheck} \ge 0$ for all simple roots $\alpha$;
equivalently $\nu$ is in the dominant chamber $\overline{\mathsf{C}}$ for
the given choice of positive roots $\Phi^+$. Moreover,
$\chi_{\lambda} = \chi_{\nu -\rho}$, since $w\cdot\lambda = \nu -\rho$.
For these and related Lie-theoretic facts, see \cite{lie1, lie2},
especially \cite[Chap.~V.3.3, VI.1.10]{lie1} and
\cite[Chap.~VIII.8.5]{lie2}.

\begin{prop}
  \label{prop212}
  Let $\lambda = -d \vpi_1$, so that
  $\lambda + \rho = (1-d) \vpi_1 + \vpi_2 + \dotsb + \vpi_\ell$.
  \\
  {\rm (1)} If $d \le 1$, then the weight $\nu$ in the dominant chamber is
  given by $\nu = \lambda+\rho \in \barC$.
  \\
  {\rm (2)} Assume that $n$ is odd ($\g$ of type $\mathsf{B}_\ell$) and
  $d >1$ (i.e.~$1 \le r \le \ell -2 = \frac{n-3}{2}$). Then
  \[
    \nu \ = \ \vpi_1+\vpi_2+\dotsb+ \vpi_{\ell - r -2}+ \half
    \vpi_{\ell-r-1} + \half \vpi_{\ell-r} + \vpi_{\ell-r+1} + \dotsb +
    \vpi_\ell \ \in\ \barC \cap W(\lambda + \rho).
  \]
  {\rm (3)} Assume that $n$ is even ($\g$ of type $\mathsf{D}_\ell$) and
  $d >1$ (i.e.~$1 \le r \le \ell -3 = \frac{n-4}{2}$). Then
  \[
    \nu \ = \ \vpi_1+\vpi_2+\dotsb+ \vpi_{\ell - r -2}+ \vpi_{\ell-r}+
    \vpi_{\ell-r+1} + \dotsb + \vpi_\ell \ \in\ \barC \cap W(\lambda +
    \rho).
  \]
\end{prop}

\begin{proof}
  (1) This is obvious.
  
  (2) Observe that
  $1-d < 0 \iff r < \halfn- 1 \iff r \le \ell -2 = \frac{n-3}{2}$. Let
  $w^{-1}$ be the cycle $(\ell-r,\ell-r-1,\dotsc,2,1) \in
  W$. Then, $\nu = w^{-1}(\lambda+\rho)$ has the desired form.  
  
  (3) Here, $1-d < 0 \iff r < \ell -2$. If $w^{-1}$ is the cycle
  $(\ell -r -1,\ell -r -2,,\dotsc,2,1) \in W$, the weight
  $\nu =w^{-1}(\lambda + \rho)$ has the desired form.
\end{proof}

\begin{cor}
  \label{cor213} {\rm (i)} If $n$ is even with $r \geq \halfn $ then $J_r$
  is not maximal. The unique primitive ideal containing $J_r$ is the finite
  codimensional (maximal) ideal $I((r-\halfn)\vpi_1)$.
  \\
  {\rm (ii)} In all other cases, $J_r$ is maximal.
\end{cor}

\begin{proof}   
  Since $\calV(J_r) = \Ominbar$, any proper factor of $U(\g)/J_r$ must be
  finite dimensional and that factor then has to be unique.  On the other
  hand, if $U(\g)/J_r$ is a simple ring, then it cannot have a nonzero,
  finite dimensional module. Thus part (i) already follows from
  Theorem~\ref{thm26}.

  Conversely, suppose that $J_r=I(\lambda)$ is not maximal; thus it has a
  finite dimensional quotient. This
  implies that there exists $\nu \in \Psf_{++} \subset \barC$ such that
  $\nu \in W(\lambda + \rho)$ (see \cite[Chap.~VIII.7.2,
  Corollaire~1  and VIII.8.5, Corollaire~1]{lie2}).
  Proposition~\ref{prop212} gives the unique weight $\nu$ in
  $W(\lambda + \rho)\cap \barC$. By inspection, $\nu \in \Psf_{++}$ if and
  only if $1-d \in \N^*$. As usual this forces $n$ to be even and
  $r - \halfn \ge 0$.
\end{proof}
 
%%%%%%%%%%%%%%%%%%%%%%%%%%%%%%%

\section{Rings in category  $\mathcal{O}$}
\label{Orings}

In this section we study rings $R$ on which a reductive Lie algebra $\g$
acts as differential operators, abstracting the situation from the
introduction. Although we need a considerable number of hypotheses, they do
hold occur frequently and the consequences are surprisingly strong. In
particular, they force a factor ring of $U(\g)$ to be a ring of $\g$-finite
vectors (see Theorem~\ref{thm1.4} and Definition~\ref{k-finite-defn}).

We are concerned with the situation
described in the following hypotheses, for which we need a definition.
A module $M$ over a ring $U$ of finite Gelfand-Kirillov dimension is called
\emph{quasi-simple} if its socle $\Soc(M)$ is simple, with
$\GKdim_U(M/\Soc(M))<\GKdim_U M$.
 
\begin{hypotheses}\label{ring-hyp} Let $\g$ be a finite dimensional,
  complex reductive Lie algebra with triangular decomposition
  $\g=\fn^-\boplus \fh\boplus \fn^+$ and assume that $R$ is a commutative,
  finitely generated $\C$-algebra that is a $\g$-module via a morphism
  $\chi: U(\g)\to \calD(R)$. Set $I=\Ker(\chi)$ and identify $\UU=U(\g)/I$
  with its image $\chi(U(\g)) $ in $\calD(R)$.  Assume moreover that:
  \begin{enumerate}[{\indent \rm (1)}]
  \item under the above action, $R$ is a quasi-simple, highest weight
    $\g$-module;
  \item $R$ is generated by $\chi(\fr)$ for some Lie subalgebra
    $\fr\subseteq \fn^-$;
  \item if $N=\Soc_{U(\g)}(R)$, then $\End_R(N)=R$.
  \end{enumerate}
\end{hypotheses}

Although these hypotheses are clearly strong, they do occur for the rings
relevant to this paper (see Lemma~\ref{lem1.1} for (3)). It is
important however, that we do not require that $\g$ act by derivations and
so, in that sense at least, our definition is weaker than Joseph's concept
of $\calO$-rings from \cite{J-Oring}.

We remark that, by (2), the $\g$-socle $N$ of $R$ is also an $R$-module and
so (3) makes sense. Moreover, $R$ has finite length as a $U(\g)$-module.  A
routine argument (see, for example, \cite[Lemma~8.14]{jantzen}) shows that
$I=\ann_{U(\g)}(R)=\ann_{U(\g)}(N)$.  Hence $\UU$ is primitive.

First we clarify when Hypothesis~\ref{ring-hyp}(3) holds.

\begin{lem}
  \label{lem1.1}
  Let $C$ be a finitely generated commutative $\C$-algebra and $f \in C$.
  Suppose that $C$ is a Cohen-Macaulay domain and set $R= C/(f)$. Let $M$
  be an ideal of $R$ such that $\GKdim(R/M)\leq \GKdim(R)-2$. Then
  Hypothesis~\ref{ring-hyp}(3) holds for $M$ in the sense that
  $\End_R(M)=R$.
\end{lem}

\begin{proof} 
  The hypothesis on $C$ implies that $R$ is Cohen-Macaulay and that the
  {grade} of the $R$-module $R/M$ is at least $ 2$, see~\cite[(16.A),
  (16.B)]{Mats}; that is, $\Hom_R(R/M,R)= \Ext^1_R(R/M,R)= 0$.  Applying
  $\Hom_R(-,R)$ to the short exact sequence
  $0\lto M \lto R \lto R/M \lto 0$ then gives $R= \End_R(R)= \Hom_R(M,R)$.
  It follows that $\Hom_R(M,R)= \End_R(M)$ and so $\End_R(M)=R$.
\end{proof}

\begin{defn}\label{k-finite-defn}
Let $R,\g$ satisfy Hypotheses~\ref{ring-hyp} and let $M, N$ be (left)
$\g$-modules.  Define the adjoint action of $x\in\g$ on $\theta\in  \Hom_{\C}(M,N)$ by
 $\ad(x)( \theta) = x\theta-\theta x$. Set  
\[\calA(M,N) =\{\theta\in \Hom_{\C}(M,N) : \fr \text{ acts ad-nilpotently
    on }\theta\}.\]
Analogously to the order filtration for $\calD(R)$ in \eqref{lem0.15}, we   
filter  the spaces  $\calA(M,N)$ by $ \calA(M,N) =\bigcup_{k\geq 0} \calA_k(M,N)$, where
$\calA_{-1}(M,N)=0$ and, for $k\geq 0$,
\begin{equation}\label{lem1.15}
  \calA_k(M,N)\ = \  \bigl\{\theta\in \Hom_{\C}(M,N) : [x,\theta]=x\theta-\theta x\in
  \calA_{k-1}(M,N) \ \text{for all}\  x\in \fr\bigr\}. 
\end{equation}

Recall, for example from  \cite[\S6.8]{jantzen}, that
one  also has  the set of \emph{$\g$-finite elements} $\calL(M,N)\subseteq \Hom_\C(M,N)$, 
on which $\g$ acts locally finitely  under    the
adjoint action. If $M=N$, then \cite[\S6.8(5)]{jantzen} implies that $\calL(M,M)$ contains
$U(\g)/\ann(M)$ as a subalgebra.  
\end{defn}
  
\begin{prop}\label{lem1.2} Assume that  
  $(R,\g)$ satisfies Hypotheses~\ref{ring-hyp} and let $M,N$ be non-zero
  left $\g$-modules.
   \begin{enumerate}[{\indent \rm (i)}]
  \item Under the natural $R$-module action induced from
    \eqref{ring-hyp}(2), each $\calA_k(M,N)$ is an 
    $R$-bimodule.   

  \item $\calD(M,N)=\calA(M,N)\supseteq \calL(M,N)$. Indeed
    $\calD_k(M,N)=\calA_k(M,N)$ for all $k\geq 0$.
  \item  Assume that $M\subseteq R$ with $\GKdim(R/M)<\GKdim(R)$ and
    $\End_R(M)=R$. Then $\calL(M,M)= \UU$.
  \end{enumerate} 
\end{prop}

\begin{proof} (i) By induction, assume that $\calA_{k-1}(M,N)$ is an
  $R$-bimodule.  Let $\theta\in \calA_k(M,N)$, $y\in R$ and $x\in\fr$.
  Then, by induction, and the fact that $R$ is commutative,
\[
[x, \theta y] = x\theta y- \theta xy  =(x\theta -\theta x)y \in
\calA_{k-1}(M,N)R=\calA_{k-1}(M,N).
\]
Thus, $\theta y\in \calA_k(M,N)$ and, by symmetry, the result follows.

(ii) As $\fn^-$ acts ad-nilpotently on any finite dimensional $\g$-module
and hence on $\calL(M,N)$, Hypotheses~\ref{ring-hyp}(2) implies that
$\calA(M,N)\supseteq \calL(M,N)$. Since the inclusion
$\calA_k(M,N)\supseteq \calD_k(M,N)$ is obvious, we are reduced to proving
that $\calA_k(M,N)\subseteq \calD_k(M,N)$ for each $k$.  By induction and
for some $m$, assume that the result holds for all $k<m$ and fix
$\theta\in \calA_m(M,N)$.

Filter $R=\bigcup \Lambda_pR$, by taking $\Lambda_0=\C$ and
$\Lambda_pR=\chi(\fr)^p+\Lambda_{p-1}R$ for $p>0$. For some $q$, assume
that $[\theta, z]\in \calD_{m-1}(M,N)$ for all $z\in \Lambda_sR$ and $s<q$.
(The case $s=1$ holds automatically by the definition of $\calA(M,N)$, so
the induction does start.)  Let $y\in \Lambda_qR$ and suppose that
$y=x_1x_2$ for $x_1\in \chi(\fr)$ and $x_2\in \Lambda_{q-1}R$. Then, by
part~(i) and the induction on $q$,
\[
  \begin{array}{rl}
    \theta y-y\theta \ = \ \theta x_1x_2-x_1x_2\theta & =\ 
    \theta x_1x_2 -x_1\theta x_2+x_1\theta x_2 -x_1x_2\theta \\
    \noalign{\vspace{6pt}} 
  &=\ (\theta x_1-x_1\theta) x_2 +x_1(\theta x_2-x_2\theta) \\
    \noalign{\vspace{6pt}} 
   &\in\ \calA_{m-1}(M,N)x_2+ x_1 \calA_{m-1}(M,N)\ \subseteq\
   \calA_{m-1}(M,N).  
  \end{array}
\]
Hence $[\theta,y]\in \calA_{m-1}(M,N)$. By linearity and our inductive
hypotheses, $[\theta,z]\in \calA_{m-1}(M,N)=\calD_{m-1}(M,N)$ for all
$z\in R$. Hence $\calA_m(M,N)\subseteq \calD_m(M,N)$.
    
(iii) Since $R$ is a quasi-simple $U(\g)$-module,
\cite[Lemma~8.14]{jantzen} implies that $\ann_{U(\g)}(M)= \ann_{U(\g)}(R)$
and so $\UU\hookrightarrow \calL(M,M)$ by \cite[\S6.8(5)]{jantzen}.  By
definition, $\calM=\calL(M,M)$ is a locally finite $\g$-module under the
adjoint action and hence is semi-simple. Thus, we may write
$\calM=\UU\oplus P$ for some complementary $(\ad\g)$-module $P$.  Assume
that $P\not=0$.

Pick $p\in P\smallsetminus \{0\}$ with $p\in \calA_k(M,M)$ for $k$ as small
as possible. Then, for $x\in \fr$, we have both $[x,p]\in \calA_{k-1}(M,M)$
by the definition of the order filtration  \eqref{lem1.15}, and $[x,p]\in P$ as $P$ is an
$(\ad\fr)$-module. Thus by the choice of $k$ we have $[x,p]=0$ for all
$x\in\fr$.  Therefore, by hypothesis and Part~(ii),
\[
p\in \calA_0(M,M) =  \calD_0(M)=\End_R(M)=R.
\]
     However, by Hypothesis~\ref{ring-hyp}(2), $R\subseteq \UU$ as subrings
     of $\calD(R)$. Hence $p\in \UU$, giving the required contradiction.
   \end{proof}

\begin{defn}\label{order-defn}
  If $A,B$ are prime Goldie rings with the same simple artinian ring of
  fractions $Q$, then $A$ and $B$ are \emph{equivalent orders} if
  $a_1Ba_2\subseteq A$ and $b_1Ab_2\subseteq B$ for some regular elements
  $a_i\in A$ and $b_j\in B$. The ring $A$ is called a \emph{maximal order}
  if it is not equivalent to any $B\supsetneq A$.  
\end{defn}

The property of being a maximal order is the appropriate noncommutative analogue of 
being an integrally closed commutative ring and has a number of useful consequences; 
see for example \cite{js84} or \cite{MaR}.

\begin{thm}\label{thm1.4}
  Let $R$ and $\g$ satisfy Hypotheses~\ref{ring-hyp}.  Then
  \[
    \calL(R,R)=\calL(N,N)=\UU.
  \]
  In particular, $\calL(R,R)$ is a maximal order.
\end{thm}

\begin{rem}\label{kostant}  In the sense of, say, \cite{jo80} this solves
  the Kostant problem for the ring $\UU$. 
\end{rem}
 
\begin{proof} Both $M=R$ and $M=N$ satisfy the hypotheses of
  Proposition~\ref{lem1.2} and so the displayed equation follows. The final
  assertion then follows from \cite[Theorem~2.9]{js84}.
\end{proof}

Here is a simple  example that illustrates the restrictions of
Hypothesis~\ref{ring-hyp}(3), even if one assumes that $R/N(R)$ has
pleasant properties.

\begin{example} 
  Take $A =\C[x,y,xz,yz, z^2,z^3] \subset \C[x,y,z]=C$; thus $C=A+\C z$.
  Set $I =A\cap z^2C$ and $R=A/I$.  Clearly the nil radical
  $N(R)$ equals $(zC\cap A)/I = ((xz,yz)+I)/I$ and so $R/N(R)\cong \C[x,y]$.
  However, $N(R)/I = (xz,yz) \cong (x,y)$ as modules over
  $k[x,y]\cong R/N(R)$.

  Now let $M=(x,y,zx,zy,z^2,z^3)/I$ be the augmentation ideal of $R$.  The
  maximal ring of fractions $\mathrm{Quot}(R)$ clearly contains
  $z=(zx)x^{-1}$.  However $z\not\in R$ yet, by the computations of the
  last paragraph, $zM \subseteq R$. Indeed $zM \subseteq M$ and so
  $\End_R(M)\supsetneq R$.
\end{example}

%%%%%%%%%%%%%%%%%%

\section{Modules of  $\sontwo$-finite vectors}
\label{kfinite}

As we noted in Theorem~\ref{thm22},
$\mathcal{D}(R_1)=U(\mathfrak{so}(n+2))/J$ for the Joseph ideal $J$.  We
would like to pass from $\mathcal{D}(R_1)$ to $\mathcal{D}(R_r)$ for any
$r$ to obtain the analogous result for $\mathcal{D}(R_r)$.  There are two
potential ways of creating such a passage; through differential operators
$\mathcal{D}(R_r,R_1)$ and through the $\mathfrak{g}$-finite vectors
$\mathcal{L}(R_r,R_1)$.  The first step, however, is to prove that the
latter is non-zero. This we accomplish in this section and then we use it
in the next section to prove  the major part of Theorem~\ref{thmA} from the
introduction. 
  
There are a number of different rings involved in this discussion and so we
need to refine the earlier notation.  
As in Section~\ref{sec1}, set
$A=\C[X_1,\dots, X_n] \ni \Ft=\sum_{i=1}^n X^2_i$ with $R_r=A/(\Ft^r)$ for
any $r\geq 1$. Recall the elements $\Delta_1, \Et_1,D_{k\ell} \in \calD(A)$
from \eqref{eq13} as well as the elements
\[ P_{kj}=P_j(\textstyle{\frac{n}{2}}-k) =
  X_j\Delta_1-\bigl(\Et_1+(\textstyle{\frac{n}{2}}-k)\bigr)\partial_j\in
  \calD(A) \quad\text{for}\ 1\leq j\leq n \ \text{and} \ k\geq 1,
\]
defined by \eqref{eq15}. By Corollary~\ref{cor19}, the images of
$\{P_{rj}, D_{ij}, (\Et_1+\novertwo-r), X_i : 1\leq i,j\leq n\}$ span a
copy $\g_r$ of $\g=\fso(n+2,\C)$ inside $\calD(R_r)$.  Moreover, for
$r,s\geq 1$, Proposition~\ref{prop16} proves that the map
$P_{rj}\mapsto P_{s j}$ and $\Et_1+\novertwo-r\mapsto \Et_1+\novertwo-s$,
with $D_{ij}\mapsto D_{ij}$ and $X_i\mapsto X_i$ induces a Lie algebra
isomorphism $\g_r\to \g_s$. We may therefore write
$\{\widetilde{P}_j, D_{ij}, \widetilde{\Et}_1, X_i : 1\leq i,j\leq n\}$ for
the corresponding generators of the abstract copy $\g$ of
$\mathfrak{so}(n+2,\C)$ and rephrase Corollary~\ref{cor19} as saying that
there is a Lie algebra isomorphism $\psi_r : \g\to \g_r$ given by
$\widetilde{P}_j\mapsto P_{rj}$ and
$\widetilde{\Et}_1\mapsto \Et_1+\novertwo-r$, etc.

\begin{lem}\label{kfinite1}  Fix $r\geq 2$ and recall the notation
  $\calD(M,N)$ from \eqref{lem0.15}. 
  Then, regarded as differential operators acting on $A$, the elements
  $\Delta_1$, $\partial_j$ and $1$ induce differential operators in
  $\calD(R_r,\, R_{r-1})$.
\end{lem}

\begin{proof}
  Suppose first that $\phi\in \calD(A)$ satisfies
  $\phi(\Ft^rA)\subseteq \Ft^{r-1}A$ and so induces a function
  $\widetilde{\phi}:R_r\to R_{r-1}$. Then we claim that
  $\widetilde{\phi}\in \calD(R_r,\,R_{r-1})$.  To see this, assume that
  $\phi\in \calD_t(A)$ and that the analogous statement is true for all
  $\phi'\in \calD_{t-1}(A)$. Then, for $a\in A$, the function
  $\phi'=[\phi,a] $ also satisfies $\phi'(\Ft^rA)\subseteq \Ft^{r-1}A$, but
  now $\phi'\in \calD_{t-1}(A)$. Hence
  $\phi'\in \calD_{t-1}(R_r,\,R_{r-1})$.  Thus,
  $\phi\in \calD_{t}(R_r,\,R_{r-1})$, proving the claim.

 In order to prove the lemma it therefore suffices to show that each of our
  elements $\phi$ satisfies $\phi(\Ft^rA)\subseteq \Ft^{r-1}A$.  This is
  obvious for $1$ and $\partial_i$, while for $\Delta_1$ it follows
  from~\eqref{prop17-1}.
\end{proof}
 
\begin{Notation}\label{not-kfinite} For a fixed value of $r$, the images of
$\Delta_1$, $\partial_j$ and $1$ in $\calD(R_r,\, R_{r-1})$ will be
written, respectively, $\overline{\Delta}_1$, $\overline{\partial}_j$ and
$\overline{1}$.  For concreteness, we note that the \emph{adjoint action}
$\widetilde{\psi}$ of $\g=\fso(n+2,\C)$ on $\calD=\calD(R_r,\,R_{r-1})$ is
defined by
$\widetilde{\psi}(x)(\theta) = \psi_{r-1}(x)\circ\theta -
\theta\circ\psi_r(x) $
for $x\in\g$ and $\theta\in \calD$.  
\end{Notation}

\begin{prop}\label{kfinite-2} Fix $r\geq 2$ and 
  let
  $V\ = \ \C\overline{1}+\sum_i\C \overline{\partial}_i +\C
  \overline{\Delta}_1 \subset \calD(R_r,\,R_{r-1})$.
  Then, under the adjoint action $\widetilde{\psi}$, $V$ is a
  $\g$-submodule of $\calD(R_r,\,R_{r-1})$.  In particular,
  $V\subseteq \mathcal{L}(R_r,\, R_{r-1})$.
\end{prop}

\begin{proof} By the definition of $\mathcal{L}(R_r,\, R_{r-1})$ it
  suffices to prove that $V$ is a $\g$-module. This requires  some
  explicit computations, for which the following three formul\ae\ in $\g$
  will prove useful and will be used frequently without comment: for all
  $1\leq i,j\leq n$ we have
  $[\Delta_1,\,\partial_j] = [\Delta_1,\, D_{ij}] = 0$, while
  \[ [\Delta_1,\, X_j] \ = \ [\textstyle{\frac{1}{2}} \partial_j^2,\, X_j]
  \ = \ \partial_j \quad\text{and}\quad [\Et_1,\, \Delta_1] \ = \
  \textstyle{\frac{1}{2}} \sum_j[\Et_1,\, \partial_j^2]\ = \ -2\Delta_1.
  \]
  First, computing in $\calD(R_r,\,R_{r-1})$ we have
  \[
   \begin{aligned}
    \widetilde{\psi}(\widetilde{P}_j)(\overline{\Delta}_1)  & =  
 P_j(\novertwo-r+1)\overline{\Delta}_1 -  \overline{\Delta}_1P_j(\novertwo-r) 
 \\   \noalign{\vspace{2pt}} 
 & =  \Bigl( 
  X_j\Delta_1\overline{\Delta}_1 - \overline{\Delta}_1X_j\Delta_1 \Bigr)  +
   \Bigl(  - \Et_1\partial_j
 \overline{\Delta}_1   +\overline{\Delta}_1 \Et_1 \partial_j \Bigr) 
 \\  & \phantom{ \Bigl( =
  X_j\Delta_1\overline{\Delta}_1 - \overline{\Delta}_1X_j\Delta_1 \Bigr)  +\quad}
+  \Bigl(
-  (\novertwo
  -r+1)\partial_j\overline{\Delta}_1 
    +(\novertwo-r)\overline{\Delta}_1\partial_j   \Bigr)    \\  
   \noalign{\vspace{1pt}} 
 & =    -\partial_j \overline{\Delta}_1 +2 \partial_j \overline{\Delta}_1
  - \partial_j \overline{\Delta}_1 \ = \ 0. \end{aligned}
 \]
 Next,
 \[
 \widetilde{\psi}(\widetilde{\Et}_1)(\overline{\Delta}_1) \ = \
  (\Et_1+\novertwo
 -r+1)\overline{\Delta}_1 -   \overline{\Delta}_1 (\Et_1+\novertwo -r) 
 \ = \ -\overline{\Delta}_1 +2 \overline{\Delta}_1
 \ = \ -\overline{\Delta}_1.
\]
Similarly,
$\widetilde{\psi}(D_{ij}) (\overline{\Delta}_1) \ =\ [D_{ij} ,
\overline{\Delta}_1]=0$
and
$\widetilde{\psi}(X_j) (\overline{\Delta}_1) \ = \ [D_{ij} ,
\overline{\Delta}_1]=0$.  
 Therefore, $\widetilde{\psi}(\g)
(\overline{\Delta}_1)\subseteq V$. 
\\
Next, we compute that
\[ \begin{aligned}
  \widetilde{\psi}(\widetilde{P}_i)(\overline{\partial}_i) &=
  P_i(\novertwo-r+1)\overline{\partial}_i  -   \overline{\partial}_iP_i(\novertwo-r) 
     \\
     \noalign{\vspace{5pt}} &= 
     \Bigl(  X_i\Delta_1\overline{\partial}_i -
                                 \overline{\partial}_iX_i\Delta_1 \Bigr) 
      +\Bigl(  -\Et_1\partial_i  \overline{\partial}_i
                                 +\overline{\partial}_i   \Et_1 \partial_i
                                 \Bigr)  
      +\Bigl( -(\novertwo  -r+1)\partial_i\overline{\partial}_i
                                 +(\novertwo-r)\overline{\partial}_i\partial_i
                                 \Bigr) 
     \\
     \noalign{\vspace{5pt}} &= -\overline{\Delta}_1
         - \partial_i\overline{\partial}_i
                                +\overline{\partial}_i\partial_i \ = \ 
                               - \overline{\Delta}_1.\end{aligned}
   \]
   If $i\not=j$ we obtain
\[ \begin{aligned}
 \widetilde{\psi}(\widetilde{P}_j)(\overline{\partial}_i)
&= \Bigl(   X_j\Delta_1\overline{\partial}_i -
    \overline{\partial}_iX_j\Delta_1 \Bigr)  
+\Bigl(   -   \Et_1\partial_j \overline{\partial}_i  +
    \overline{\partial}_i \Et_1 \partial_j \Bigr)  
 +\Bigl( -    (\novertwo -r+1)\partial_j\overline{\partial}_i  +
    (\novertwo-r)\overline{\partial}_i\partial_j \Bigr) 
 \\
 \noalign{\vspace{5pt}}
&=   -  \partial_j\overline{\partial}_i + \overline{\partial}_i\partial_j  
    \ = 0.\end{aligned}  
\]
On the other hand,
$ \widetilde{\psi}(\widetilde{\Et}_1)(\overline{\partial}_i) \ = \
(\Et_1+\novertwo -r+1)\overline{\partial}_i - \overline{\partial}_i
(\Et_1+\novertwo -r) \ = \ \overline{\partial}_i - \overline{\partial}_i \
= 0 $ while
 \[
 \widetilde{\psi}(D_{ab})\overline{\partial}_i = 
   (X_a\partial_b - X_b\partial_a)\overline{\partial}_i -
   \overline{\partial}_i(X_a\partial_b - X_b\partial_a)   
   \ = \ \sum_u \alpha_u
 \overline{\partial}_u,
  \]
  for any $1\leq a,b,i\leq n$, and the appropriate scalars
  $\alpha_u\in \C$.  Since
  $\widetilde{\psi}(X_k)\overline{\partial}_i = [ X_k ,
  \overline{\partial}_i ] \ = -\delta_{k,i}\overline{1}$,
  it follows that
  $\widetilde{\psi}(\g) (\overline{\partial}_i)\subseteq V$.  The easy fact
  that $\widetilde{\psi}(\g) (\overline{1})\subseteq V$ is left to the
  reader and completes the proof.
   \end{proof}

     \begin{cor}\label{kfinite-3}
      Fix $r> s\geq 1$.  Then  the following hold.
      \begin{enumerate} 
      \item[{\indent {\rm (1)}}] Both the projection map $\pi_{rs}: R_r\to R_s$ and
        the operator $\Delta_1^{r-s}$ belong to
        $\mathcal{L}(R_r,\,R_s)$.
      \item[{\indent {\rm (2)}}] Set $\socRs=\Soc_{U(\g)}R_s$.  Then
        $\calL(\socRr,\socRs)\not=0\not=\calL(\socRs,\socRr)$.
      \end{enumerate} 
\end{cor}
 
 \begin{proof}
   (1) If $\theta\in \mathcal{L}(A,B)$ and $\phi\in \mathcal{L}(B,C)$ for
   $\g$-modules $A,B,C$ then $\phi\theta\in \mathcal{L}(A,C)$
   (see\cite[6.8(3)]{jantzen}). Thus the result follows from
   Proposition~\ref{kfinite-2} and induction on $r-s$.
 
   (2) By Theorem~\ref{thm26} each $R_s/\socRs$ is finite dimensional (or
   zero), say with $\fa_s=\ann_{U(\g)}(R_s/\socRs)$. Since $\socRr$ is
   infinite dimensional, the projection map
   $\pi_{rs}\in \mathcal{L}(R_r,\,R_s)$ has an infinite dimensional image,
   as does its restriction $\pi'_{rs}: \socRr\to R_s$.  Hence 
   $\fa_s\pi'_{rs}\not=0$ and so, for some $a\in \fa_s$, one has
   $0\not= a\circ\pi'_{rs}\in \calL(\socRr,\socRs)$. Thus $\calL(\socRr,\socRs)\not=0$.
   By \cite[\S6.9(5)]{jantzen}, this also implies that
   $\mathcal{L}(\socRs,\, \socRr) \not=0$.
 \end{proof}

%%%%%%%%%%%%%%%%%%

\section{Differential operators and primitive factor rings of   $U(\sontwo)$}
\label{maintheorem}

As usual, we write $R_r=A/(\Ft^r)$ for $A=\C[X_1,\dots,X_n]$ and some
$r\geq 1$.  The aim of this section is to combine the earlier results to
prove that $\calD(R_r)= \calL(R_r,R_r)= U(\g)/J_r$, where $\g=\sontwo$ and
$J_r=\ann_{U(\g)}(R_r)$.  In particular, this proves  the major part of
Theorem~\ref{thmA} from the introduction.

We begin with some elementary results, the first of which is a minor
generalisation of \cite[Corollary~2.10]{js84}.

 \begin{lem}\label{main1}  If $M$ is a simple $U(\g)$-module with
   $I=\ann_{U(\g)}(M)$, then $\calL(M,M)$ is the unique maximal
   object among orders containing and equivalent to $U=U(\g)/I$.
 \end{lem}

 \begin{proof} By \cite[Corollary~2.10]{js84}, $\calL(M,M)$ is a maximal order
   and it is noetherian since it is finitely generated as  both a left and
   a right $U$-module.  So it remains to prove uniqueness.  Suppose that
   $U\subseteq T$ for some other order $T$ equivalent to $U$. Thus,
   $aTb\subseteq U$ for some regular elements
   $a,b\in U$. Let $T'=U aT+U$; this is an overring of $U$
   with $T'b\subseteq U$.  So,  since $b$ is regular, $T'$ is certainly finitely generated as
   a left $U$-module. 
   Therefore, by \cite[Theorem~2.9]{js84}, $T'\subseteq \calL(M,M)$. In
   particular, $T'$ is also finitely generated as a right $U$-module.
   Since $aT\subseteq T'$ it follows that $T$ is also a finitely generated
   right $U$-module. Hence \cite[Theorem~2.9]{js84} implies that
   $T\subseteq \calL(M,M)$.
 \end{proof}
 
\begin{lem}\label{main11}   
  Let $a,b,c\in\mathbb{N}^*$ and set  $\socRa=\Soc_{U(\g)}(R_a)$.  %Then:
  \begin{enumerate}[{\rm (i)}]
  \item  $\calL(R_a,R_a)=\calL(\socRa,\socRa)= U(\g)/J_a$ is a primitive
    ring.
  \item $\calL(R_a,R_b)$ and $\calL(\socRa,\socRb)$ are non-zero and
    torsion-free   both as  left $U(\g)/J_b$-modules and as  right
    $U(\g)/J_a$-modules.
  \item  Under composition of operators, both
    $\calL(R_b,R_a) \calL(R_a,R_b)$ and $\calL(\socRb,\socRa)\calL(\socRa,\socRb)$ are
    non-zero ideals of $\calL(R_a)$.
  \end{enumerate}
\end{lem}
 
\begin{proof} (i) By Theorem~\ref{thm26}, Hypotheses~\ref{ring-hyp} is
  satisfied and so the result follows from Theorem~\ref{thm1.4}.

  (ii) By \cite[Lemma~8.14]{jantzen}, $\ann_{U(\g)}(R_b)=\ann_{U(\g)}(\socRb)=J_b$
  and so $R_b$ is \emph{ann-homogeneous} in the sense of
  \cite[\S2.7]{js84}.  Thus, by \cite[Lemma~2.8]{js84}, $\calL(R_a,R_b)$ is
  \emph{GK-homogeneous} in the sense that any non-zero, left or right
  $\g$-submodule $N$ of 
  $ \calL(R_a,R_b)$ has $\GKdim(N)=\GKdim( \calL(R_a,R_b))$.    
    Since each $U(\g)/J_c$ is a prime ring
  by~(i), this is equivalent to $ \calL(R_a,R_b)$ being torsion-free on
  both sides.  Exactly the same argument works for $\calL(\socRa,\socRb)$.
  
  It therefore remains to prove that these modules are non-zero. For
  $\calL(\socRa,\socRb)$ this is Corollary~\ref{kfinite-3}.  In particular,
  $\calL(\socRa,\socRb)$ is infinite dimensional, and hence
  $\calL(\socRa,\socRb)\fa_a\not=0$, where $\fa_a$ is the annihilator of the
  finite dimensional $\UU_a$-module $R_a/\socRa$. Thus
  $0\not= \calL(\socRa,\socRb)\fa_a\subseteq \calL(R_a,\socRb) \subseteq
  \calL(R_a,R_b)$.
 
  (iii) This is automatic from (ii).
\end{proof}

We also have analogous elementary results about differential operators,
although as we do not (yet) know that these rings are noetherian the proof
is a little more involved.

  \begin{lem}\label{main2}  
    Set $\{1,r\}=\{a,b\}$.  Then:
    \begin{enumerate}[{\rm (i)}]
    \item  $\calD(R_a)$ is a prime Goldie
      ring;
  \item  $\calD(R_a)$ has \emph{Goldie rank $a$}; indeed the simple
    artinian ring of fractions 
    $\mathrm{Quot}(\calD(R_a))$ is  even  isomorphic to the $a\times a$ matrix
    ring $M_a(\mathrm{Quot}(\calD(R_1)))$;
    \item  $\calD(R_a,R_b)$ is non-zero and torsion-free  both  as a
      left $\calD(R_b)$-module and as a right $\calD(R_a)$-module.
     
    \end{enumerate}
  \end{lem}
  
\begin{proof}
  (i) Clearly $R_a$ has a local artinian  ring of fractions $(R_a)_{(\Ft)}$. Now
  apply \cite[Corollary~2.6]{Musson}.

(ii)  This follows from   \cite[Lemma~2.4 and (2.5)]{Musson}.

  (iii) The projection map $\pi:R_r\to R_1$ is an $A$-module map and so
  belongs to $\calD(R_r,R_1)$.  Similarly, the $A$-module isomorphism
  $\phi:R_1\to \Ft^{r-1}R_r$ belongs to $\calD(R_1,R_r)$. Thus,
  $\calD(R_r,R_1)\not=0\not=\calD(R_1,R_r)$.

  As right $\calD(R_r)$-modules,
  $\calD(R_r,R_1)\cong \calD(R_r, \Ft^{r-1}R_r)\hookrightarrow
  \calD(R_r,R_r)$.
  Thus by (i), it is a torsion-free right $\calD(R_r)$-module.   
 For the left action, consider $0\not=\phi\in \calD(R_r,R_1)$ and pick $s$
 maximal such that 
  $\phi(\Ft^sR_r)\not=0$.  If $\bar{r}\in R_1$, write $\bar{r}=\pi(r)$ for
  some $r\in R_r$ and define
  $\widetilde{\phi}(\bar{r})=\phi(\Ft^sr)$. Since $\phi(\Ft^{s+1}R_r)=0$,
  this is a well defined morphism and hence a differential operator; thus
  $\widetilde{\phi}\in \calD(R_1,R_1)$.  Moreover, $\widetilde{\phi}\not=0$
  by the choice of $s$, and so $\widetilde{\phi}$ is not torsion as an
  element of the left $\calD(R_1)$-module $\calD(R_1)$.  Consequently,
  $\phi$ is not torsion under the left $\calD(R_1)$ action.

Conversely, if $0\not=\theta\in \calD(R_1,R_r)$, then
  $0\not=\theta\circ\pi\in \calD(R_r,R_r)$. Thus neither $\theta\circ\pi$
  nor $\theta$ is torsion under the left $\calD(R_r)$ action. Hence
  $\calD(R_1,R_r)$ is torsion-free as a left $\calD(R_r)$-module.
  Finally, let $0\not=\phi\in \calD(R_1,R_r)$ and choose $s$ minimal such
  that $\phi(R_1)\subseteq \Ft^sR_r$.  Thus, $\phi\in
  \calD(R_1,\Ft^sR_r)$.
  Moreover, if $\rho:\Ft^sR_r\to \Ft^sR_r/\Ft^{s+1}R_r\cong R_1$, then
  $0\not=\rho\circ\phi\in \calD(R_1,R_1)$. As in the last
  paragraph, this ensures that $\phi$ is not torsion under the right
  $\calD(R_1)$  action. Hence $\calD(R_1,R_r)$ is torsion-free on the right.
\end{proof}

By Lemma~\ref{chevalley1}, we know  that
$A=\C[X_1,\dots,X_n]\hookrightarrow U(\mathfrak{n}^-)\subset U(\g)$.
As before, set  
$\socRa=\Soc_{U(\g)}(R_a)$ for $a\in \mathbb{N}^*$.
Then, by Proposition~\ref{lem1.2}(ii) and Theorem~\ref{thm26}   we have     
 \begin{equation}\label{main5}
  \calL(R_a,R_b)\subseteq \mathcal{D}(R_a,R_b) \qquad\text{and}\qquad
    \calL(\socRa,\socRb)\subseteq \mathcal{D}(\socRa,\socRb) \qquad\text{for }  \{a,b\}=\{1,r\}.
\end{equation}

    \begin{prop}\label{lem4} 
      The rings $\calD(R_r)$ and $\calL(R_r,R_r)$ are equivalent orders,
      with $\calD( R_r)\supseteq \calL(R_r,R_r)$.
    \end{prop}
  
 \begin{proof}      
   Consider
   \[
   F={\calD}(R_1,R_r){\calL}(R_r,R_1) \quad\text{and}\quad G=
   {\calL}(R_1,R_r){\calD}(R_r,R_1),
   \]
   where multiplication is composition of functions.  We claim that
   $F\not=0\not=G$. To see this it suffices, by \eqref{main5},  to prove that
   $\calL(R_a,R_b)\cdot \calL(R_a,R_b)\not=0$ (where $\{a,b\}=
   \{1,r\}$). This follows from Lemma~\ref{main11}(ii).

   Recall from Theorem~\ref{thm22} that $\calD(R_1)=\calL(R_1,R_1)$. Thus,
   repeating the argument of the last paragraph shows that
   \[
   \begin{array}{rl}
     0\ \neq \ GF  =& {\calL}(R_1,R_r){\calD}(R_r,R_1) \cdot{\calD}(R_1,R_r)
                        {\calL}(R_r,R_1)\\ 
     \ \subseteq&  {\calL}(R_1,R_r){\calD}(R_1)  {\calL}(R_r,R_1)
                  \ = \  {\calL}(R_1,R_r)\calL(R_1,R_1) {\calL}(R_r,R_1) \\
     \  =& {\calL}(R_1,R_r)   {\calL}(R_r,R_1) \  \subseteq\  {\calL}(R_r,R_r).
   \end{array}
   \]
   Both $F$ and $G$ contain ${\calL}(R_1,R_r){\calL}(R_r,R_1)$ which, by
   Lemma~\ref{main11}(i,iii), contains a regular element
   $a\in \calL(R_r,R_r)$.  Thus from the last display
   $a\calD(R_r)a\subseteq G\calD(R_r)F= GF\subseteq \calL(R_r,R_r)$, as
   required.
 \end{proof}

 Finally, putting everything together we get:
  
  \begin{thm}\label{mainthm1} One has 
    $\calD(R_r)
    =\calL(R_r,R_r)=\calL(\socRr,\socRr)=U(\fso(n+2,\C))/J_r$.
    Moreover this ring is a maximal order in its simple artinian ring of
    fractions and has Goldie rank $r$.
  \end{thm}
  
  \begin{proof} By Theorem~\ref{thm26}, Hypotheses~\ref{ring-hyp} are
    satisfied and so $U(\g)/J_r=\calL(R_r,R_r)=\calL(\socRr,\socRr)$ is a maximal
    order by Theorem~\ref{thm1.4}.  Thus, by Proposition~\ref{lem4},
    $\calD(R_r)= \calL(R_r,R_r)$.  The fact that $\calD_r$ has Goldie rank
    $r$ follows from 
    Lemma~\ref{main2}.
  \end{proof}

\begin{rem}
  \label{rigidity} There are other ways to prove aspects of this
  theorem.  For example, recall that the associated variety of
  the primitive ideal $J_r$ is equal to
  $\overline{\mathbf{O}}_{\mathrm{min}}$,
  cf.~Corollary~\ref{cor211}.  Assume that $n \ne 4$ and that
  $R_r \cong L(\lambda)$ is an irreducible $\g$-module (so we are
  in Case (1) of Theorem~\ref{thm26}).  Then the minimal orbit
  $\mathbf{O}_{\mathrm{min}}$ is rigid, in the sense that it is
  not induced from any proper parabolic subalgebra; this implies
  that the $\g$-module $R_r$ is rigid in the sense of
  \cite[1.2]{jo881}. Then, \cite[5.8]{jo881} can be applied in
  this situation to show that $\cL(R_r,R_r) = \cD(R_r)$ and one
  deduces from Lemma~\ref{main11} that $U(\g)/J_r = \cD(R_r)$.
 \end{rem}

%%%%%%%%%
\section{Higher symmetries of   powers of the Laplacian} 
\label{symmetries}

Recall that the starting point of this paper was to examine the symmetries
of powers of the Laplacian $\Delta_1=\half \sum\partial_{X_i}^2$, as
defined below, and thereby to extend the work of Eastwood and others
\cite{ea05,eale08,mi11} on this concept. In this section we translate
Theorem~\ref{mainthm1} into a result about those symmetries.  In
particular, this proves the second half of Theorem~\ref{thmA} from the
introduction, by identifying the ring of symmetries of $\Delta_1^r$ with a
factor ring of $U(\sontwo)$.

We briefly retain the notation of Section~\ref{sec1} for a
general field $\K$ of characteristic zero; in particular setting
$A= \K[X_1,\dots,X_n]\ni \Ft =\sum X_i^2$.  We begin with the definition of
higher symmetries as 
described, for example, in \cite[(1.2)]{bsss90}, \cite[(4.3)]{ss92} or
\cite[Definition~1]{ea05}.

\begin{defn}
  \label{def82} Fix an operator $P \in \cD(A)$.  An operator $Q \in \cD(A)$
  is a \emph{symmetry} of $P$ if there exists $Q' \in \cD(A)$ such that
  $PQ= Q'P$.  Equivalently, in the notation of~\eqref{eq10}, the symmetries
  of $P$ equals the idealizer $\I\bigl(\cD(A)P\bigr)$.  The elements
  $Q\in \cD(A)P$ are trivially symmetries of $P$ and so one usually factors
  them out and defines \emph{the algebra of symmetries of $P$} to be the
  factor algebra
  \[
    \Sscr{P} = \I\bigl(\cD(A)P\bigr) \big/ \cD(A)P.
  \]
\end{defn}

The definition of $\Sscr{\Delta^r_1}$ is of course similar in style to that
of the ring of differential operators
$\cD(A/\Ft^r)\cong \I\bigl(\Ft^r\cD(A)\bigr)/\Ft^r\cD(A)$ from~\eqref{eq11}; except
that one works with constant coefficient differential operators rather than
polynomials and left rather than right ideals.  However, as we show next,
one can easily pass from the one to the other by taking a Fourier
transform.

We now return to  the  field $\K=\C$.  In our applications it
will be convenient to use the Chevalley system from Lemma~\ref{chevalley1}
and so we will use the Fourier transform $\cF$ in $\cD(A)$ with respect to
the variables $U_{\pm j}$ from \eqref{nroots}. Thus we write
$A= \C[U_{\pm 2}, \dots, U_{\pm \ell},U_0]$ and define
\begin{equation}
  \label{eq80a}
  \cF(U_j)= \ddu{j}, \quad \cF(\ddu{j}) = U_j, \quad j \in \{0,\pm 2,
  \dots, \pm \ell\}. 
\end{equation}

In terms of the variables $X_j$ it is routine to check that
\begin{equation}
  \label{eq80b}
  \cF(X_j)=
  \begin{cases}
    \half \ddx{j} \ & \text{if $1 \le j \le \ell'$ or if $j=n$ when $n$ is
      odd;}
    \\
    - \half \ddx{j} \ & \text{if $\ell'+ 1 \le j \le 2\ell'$.}
  \end{cases}
\end{equation} 
Further routine properties of $\cF$ are given by the following lemma.  For
this we recall that $d= \halfn -r$ for some $r \in \N^*$ and that
$P_j=P_j(d) = X_j \Delta_1 - (\Et_1 + d) \ddx{j}$.
  
\begin{lem}
  \label{lem81} The map $\cF$ is an involutive anti-automorphism of
  $\cD(A)$. Moreover,
  \[
    \cF(\Ft) = \half \Delta_1, \ \cF(\Delta_1) = 2 {\Ft}, \
    \cF(\Et_1) = \Et_1, \ \cF(D_{jk}) = \pm D_{jk}, \ \cF(P_j)= \pm
    \bigl(\Ft\ddx{j} - 2X_j (\Et_1 + d)\bigl). %\qed
  \]
\end{lem}

\begin{proof} Use the formul\ae\
  $\ddx{j} = \half\bigl(\ddu{j+1} + \ddu{-(j+1)}\bigr)$ and 
  $\ddx{j+\ell'}= \frac{1}{2i}\bigl(\ddu{j+1} - \ddu{-(j+1)}\bigr)$, for
  $1 \le j \le \ell'$, while $\ddx{n}= \sqrt{2} \ddu{0}$.
\end{proof}

\begin{prop}
  \label{cor85}
  Let $r \ge 1$ and $R_r = A/\Ft^r A$.  Then the Fourier transform $\cF$
  induces an anti-isomorphism
  $\cF= \cF_r : \cD(R_r) \to \Sscr{\Delta_1^r}$.
\end{prop}

\begin{proof} Set $\cD=\cD(A)$. By Lemma~\ref{lem81},
  $\cF(\Ft^r) = 2^{-r} \Delta_1^r$.  Thus, in the notation of
  Definition~\ref{def82}, $ \Delta_1^rQ= Q'\Delta_1^r$ for some
  $Q,Q'\in \cD$ $\iff$ $ \cF(Q) \Ft^r= \Ft^r \cF(Q')$.  Equivalently,
  $ Q \in \I_{\cD}\bigl(\cD \Delta_1^r\bigr)$ $\iff$
  $\cF(Q) \in \I_{\cD}\bigl(\Ft^r\cD\bigr)$.  Hence $\cF$ induces an anti-isomorphism
  $ \I_{\cD}\bigl(\cD\Delta_1^r\bigr)/ \cD\Delta_1^r\to \I_{\cD}\bigl(\Ft^r\cD\bigr)/
  \Ft^r\cD$.
\end{proof}
 
Recall from Theorem~\ref{mainthm1} that we have an isomorphism
$\psi_r: U(\g)/J_r \to \cD(R_r) $ and hence an anti-isomorphism from
$U(\g)/J_r$ to $ \Sscr{\Delta_1^r}$.  We want to convert this into an
automorphism.  By \cite[5.2~(2)]{jantzen} $\vt(J_r)= J_r$ for any Chevalley
anti-involution $\vt$ of $\g$, and hence for the anti-involution $\vt$
defined by the Chevalley system $\{y_\alpha\}_\alpha \subset \g$ described
in Lemma~\ref{chevalley1}. Notice that
$\psit_d \circ \vt = \vt_d \circ \psit_d$ in the notation of
Remark~\ref{involutions}. In particular, using Theorem~\ref{mainthm1},
$\vt_d$ induces an anti-automorphism $\vt_r$ on $\cD(R_r)$ and combining
the earlier results of the paper we obtain the main result of this section.

\begin{thm}
  \label{thm892} {\rm (1)}  Fix $r\geq 1$. Then there are algebra isomorphisms:
  \[
    \cF \circ \vartheta_r \circ \psi_r:  
    U(\sontwo) \big/ J_r \ \longisomto \ \cD(R_r) \ \longisomto \
    \Sscr{\Delta_1^r}.
  \]
  In particular, the algebra of symmetries $\Sscr{\Delta_1^r}$ is a
  noetherian maximal order in its simple artinian ring of fractions and has
  Goldie rank $r$. Moreover, it is generated by the elements:
  \[
    \Et_1 + d; \quad \ddx{k}, \ 1 \le k \le n; \quad D_{j,k}, \ 1 \le j,k
    \le n; \quad \Ft \ddx{j} - 2 X_j (\Et_1 + d), \ 1 \le j \le n.
  \]
  
  {\rm (2)} Consider
  $\mathcal{A}=\C[\partial_{X_1},\dots, \partial_{X_n}]/(\Delta_1^r)$. If
  $n$ is even with $r<\halfn$ or $n$ is odd, then $\mathcal{A}$ is a simple
  $\Sscr{\Delta_1^r}$-module. If $n$ is even with $r\geq \halfn$, then
  $\mathcal{A} $ has a unique proper factor module, which is finite
  dimensional.
\end{thm}

\begin{proof} (1)   Since $\Sscr{\Delta_1^r}$ is generated by
  the Fourier transforms of the generators of $\cD(R_r)$, the assertion on
  generators is simply the translation of Proposition~\ref{prop16} and
  Lemma~\ref{lem81}.  For the remaining   assertions, combine Proposition~\ref{cor85}
  and Theorem~\ref{mainthm1}. 
  
  (2) Under the Fourier transform, this is a direct
  consequence of (1) and Theorem~\ref{thm26}.  Up to a
  change of Borel, the weights of these modules are also
  given by that theorem.
\end{proof}

   %%%%%%%%%%%%%%%%%%
 
\section{The real case}
\label{sec7} 
  
In this section we show that
Theorems~\ref{mainthm1} and~\ref{thm892} have natural analogues for
differential operators with real coefficients, thereby proving Theorem~\ref{thmB} 
from the introduction. There are in fact a number of different real forms of $\sontwo$, 
each with their own Laplacian, and these are all covered by our results. 

To make this precise, we will always write
$A_{\R}= \ \R[X_1,\dotsc,X_n]\subset A=\C[X_1,\dots,X_n]$ and set
$\cD(A_\R)\ = \ \R[X_1,\dotsc,X_n,\ddx{1},\dotsc,\ddx{n}]$.  Throughout the
section we fix $p,q \in \N$ with $p+q = n \ge 3$.  Then the elements
$\Ft=\sum X_j^2$ and $\Delta_1=\sum \half \ddx{j}^2$ can be replaced by,
respectively,
\[
  \Ftt \ = \ \Ftt_p \ = \ \sum_{j=1}^p X_j^2 - \sum_{j=p+1}^q X_j^2 \ \in
  A_\R\  
\] and
\[
  \dalembert \ = \ \dalembert_p \ = \ \half \Biggl(\sum_{j=1}^p \ddx{j}^2 -
  \sum_{j=p+1}^n \ddx{j}^2\Biggr)\ \in \ \cD(A_\R).\
\]
(Of course, $\Ft=\Ftt_n$ and $\Delta_1=\dalembert_n$.) The signature of the
quadratic form $\Ftt$ will be denoted by $(p,q)$ and the operator
$\dalembert$ is called the d'Alembertian (or Laplacian) on the quadratic
space $\R^{p,q} =(\R^n, \Ftt)$.

  For fixed $r\geq 1$,  set
$S=S_{p,r}= A_\R \, / \, \Ftt^r \hskip-1pt A_\R$ and
$\St = A \, / \,\Ftt^r \hskip-1pt A = S \otimes_\R \C$.  Define an
automorphism $\phi= \phi_{p,q}$ of $\cD(A)$, with $\phi(A) = A$, by
setting:

\begin{equation}
  \label{eq71}
  \phi(X_j) =
  \begin{cases}
    X_j \ & \text{if $1 \le j \le p$,}
    \\
    i X_j \ & \text{if $p+1 \le j \le n$;}
  \end{cases} \quad\text{and}\quad  \phi(\ddx{j})=
  \begin{cases}
    \ddx{j} \ & \text{if $1 \le j \le p$,}
    \\
    -i \ddx{j} \ & \text{if $p+1 \le j \le n$.}
  \end{cases}
\end{equation}
Note that $\phi(\Ft)=\Ftt$ and so $\cD(A/\,\Ftt^rA)\cong\cD(A/\Ft^rA)$.

Fix $d =d_n = \halfn -r$ and recall from
Notation~\ref{rem165} that we have a copy $\gt=\gt_d\subset \cD(A)$ of
$\g=\sontwo$.  We now want to describe a real analogue
$\fs \subset \cD(A_\R)$ of $\gt$.  By analogy with~\eqref{eq13} and
\eqref{eq15}, set:

(a) $\Dt_{kl} = \Dt_{lk} = X_k \ddx{l} + X_l\ddx{k}$ if
$1 \le k \le p < l \le n$, \ while $\Dt_{kk} =0$;
  
\smallskip (b) $\Pt_j = \Pt_j(d)=
\begin{cases}
  X_j \dalembert - (\Et_1+d) \ddx{j} \ & \text{if $1 \le j \le p$}
  \\
  X_j \dalembert +(\Et_1+d) \ddx{j} \ & \text{if $p+1 \le j \le n$;}
\end{cases}$.

\smallskip Observe that:
\begin{gather}
  X_j = \phi(X_j), \ \text{ for } j \le p \, ; \qquad \ X_j = -i\phi(X_j),
  \ \text{for } j \ge p+1 ; \nonumber
  \\
  \Pt_j = \phi(P_j), \ \text{for} \ j \le p \, ; \ \qquad \ \, \Pt_j =
  -i\phi(P_j), \  \text{ for }  j \ge p+1;    \label{eq74}  \\
\quad  \Dt_{kl} = i \phi(D_{kl}) = -i \phi(D_{lk}), \  \  \text{ for } 1 \le k \le p
  < l \le n. \qquad\qquad
\nonumber
\end{gather}
   
 The next result is standard.

\begin{lem}
  \label{cor76} Let $\fs =\fs(p,r)\subset \cD(A_\R)$ be the real subspace
  spanned by the elements
  \[
    \{\Et_1+d ; \ X_k; \ \Pt_j, \, 1 \le j,k \le n; \ D_{k,l},
    \, 1 \le k, l \le p, \ \text{or} \ p+1 \le k, l \le n; \ \Dt_{kl},
    \, 1 \le k \le p < l \le n\}.
  \]
  {\rm (1)} Then $\fs$ is a Lie subalgebra of
  $(\cD(A_\R),[\phantom{.},\phantom{.}])$ with
  $\fs \otimes_\R \C \cong \gt \cong \sontwo$.

  {\rm (2)} Moreover,
  $\fs \subset \I_{\cD(A_\R)}\bigl(\Ftt^r_p\cD(A_\R)\bigr)$ and hence
  induces a ring homomorphism $\vphit_\R : U(\fs) \lto \cD(S_{p,r})$.
\end{lem}

\begin{proof} (1) It is a routine exercise to see that $\fs$ is a Lie
  algebra; indeed, one can use \eqref{eq74} to reduce this claim to the
formul\ae\ in Lemma~\ref{lem14} and Proposition~\ref{prop16}.
By construction, the
  generators of $\phi(\fs)$ also span $\gt$; whence
  $\fs \otimes_\R \C \cong \gt$.
  
  (2) Mimic the proof of Lemma~\ref{prop17} and Corollary~\ref{cor19}.
\end{proof}

\begin{prop}
  \label{thm79}
  Let $d= d_n = \halfn - r$  and $\fs=\fs(p,r)$. Then $\vphit_\R$ yields an
  isomorphism 
  \[
  \psit_\R : U(\fs) \big/ (J_r \cap U(\fs)) \longisomto \cD(S_{p,r}).
  \]
\end{prop}

\begin{proof}
  Recall that $U(\fs) \otimes_\R \C \cong U(\gt)$, by Lemma~\ref{cor76}.
  Also, by construction, $\phi(\Ft^r)={\Ftt}^r$ and hence induces an
  isomorphism $\phi: \cD(R_r) \to \cD(\St)$. Thus one has a natural map
  $\vphit= \phi\circ\psi_r$ from $U(\g) $ to $\cD(\St)$ which, by
  Theorem~\ref{mainthm1}, is surjective.
    
  But, if $\vphit_\R$ is the map from Lemma~\ref{cor76}, then
  $\vphit_\R \otimes_\R \C = \vphit $.  By faithful flatness of the functor
  $-\otimes_\R \C$ this implies that $\vphit_\R$ is surjective and
  $\Ker(\vphit_\R)= \Ker(\vphit) \cap U(\fs)$.  (Notice that this also
  proves that $\cD(R_r)\cong \cD(\St)=\cD(S_{p,r})\otimes_\R \C$.)
  Finally, by Theorem~\ref{mainthm1}, again, $\Ker(\vphit)=\Ker(\psi)=J_r$,
  as required.
\end{proof}

The real forms of $\sontwo$ are classified, see
\cite[Chap.~X]{hel}, and in Helgason's notation   are   the real Lie algebras 
 $\fso(p'+1,q'+1)$ for $p'+q'=n$.  It is therefore not surprising that we
 have the following result.  
   
\begin{prop}
  \label{cor745}
  The Lie algebra $\fs=\fs(p,r)$   is isomorphic
  to $\fso(p+1,q+1)$ for $p+q=n$.
\end{prop}

\begin{proof} The proof is omitted since it follows from straightforward,
  but not particularly illuminating examination of the relations in $\gt$.
\end{proof}
 
 Combined with Proposition~\ref{thm79} this gives:
\begin{thm}
  \label{thm749}
  Let $p+q= n \ge 3$, $r \ge 1$ and $\Ftt_p= \sum_{j=1}^p X_j^2 -
  \sum_{j=p+1}^n X_j^2 \in A_\R= \R[X_1,\dotsc,X_n]$. Then,
  \[ 
\cD\bigl(A_\R/\,\Ftt^r_p A_\R\bigr) \ \cong \ \frac{U(\fso(p+1,q+1))}{J_r \cap
  U(\fso(p+1,q+1))}
\]

\noindent
  where $J_r$ is the primitive ideal of $U(\sontwo)$ described in
  Corollary~\ref{cor211}.\qed
\end{thm}

In Section~\ref{symmetries} we showed that the higher symmetries of the
Laplacian were equal to a factor of the enveloping algebra of
$\sontwo$. However, in 
applications (see, for example, \cite{bsss90,ea05}), one is interested in
the real case.  To end this section we  show
that those  results from  Section~\ref{symmetries} also  descend readily to the real case.

We are therefore interested in symmetries of powers of the d'Alembertian
$\dalembert=\dalembert_p$ for a fixed integer $p$. Now the basic results
from Section~\ref{symmetries} do restrict to the real field. In particular,
the Fourier transform $\cF$ from \eqref{eq80b} is well-defined on
$\cD(A_\R)$ and, by Lemma~\ref{lem81}, satisfies
$\cF(\dalembert) = 2{\Ftt_p}$. Therefore, the proof of
Proposition~\ref{cor85} can be used mutatis mutandis to show that $\cF$
induces an anti-isomorphism
\begin{equation}\label{final1}
      \cD(S)  \longrightarrow  \Sscr{\dalembert^r}   \qquad\text{for}\
  S=S_{p,r}=A_\R/\Ftt^r_p A_\R. 
  \end{equation}

  It is known, see for example \cite[2.2 and 2.3]{dfdg}, that for each real
  form $\fso(p+1,q+1)$ of $\g$ there exists a Chevalley anti-involution
  $\varkappa$ on $\g$ which stabilizes this real form. Moreover $\varkappa$
  fixes $J_r$ by \cite[5.2~(2)]{jantzen}.    It then follows from
  Theorem~\ref{thm749} that $\varkappa$ induces an anti-automorphism
  $\varkappa$ on $\cD(S)$. 
We therefore obtain the following analogue of
  Theorem~\ref{thm892}:

  \begin{thm}
    \label{thm8922}
    Let $n = p+q \ge 3$, $r \ge 1$ and $\dalembert=\dalembert_p$ be the
    d'Alembertian on $\R^{p,q}$. There exists a primitive ideal $J_r$ of
    $U(\sontwo)$ and  algebra isomorphisms:    
    \[
   \cF\circ\varkappa \circ \psit_\R   \ : \  \frac{U(\fso(p+1,q+1)) }{ (J_r \cap U(\fso(p+1,q+1)))} 
     \ \longisomto\  \cD(S_{p,r})     \ \longisomto\ 
      \Sscr{\dalembert^r}.
    \]
    Consequently, $\Sscr{\dalembert^r}$ is a primitive  noetherian ring that is
    a maximal order in its simple artinian quotient ring.  Moreover it is
    generated as an algebra by the elements
    \[
      \{\Et_1+d ; \ \partial_{X_k}; \quad \widetilde{Q}_j, 1 \le j,k \le
      n; \ D_{k,l}, 1 \le k, l \le p \ \text{or}\ p+1 \le k, l \le n;
      \ \Dt_{kl}, 1 \le k \le p < l \le n\},
    \]
    thought of as differential operators on $S_{p,r}$.  Here
  \[
    \widetilde{Q}_j=\cF(\Pt_j)= \begin{cases}  \Ftt\, \partial_{X_j} -
     2 X_j(\Et_1+d) & \text{if $1 \le j \le p$}
      \\
       \Ftt\, \partial_{X_j} + 2X_j(\Et_1+d) \ & \text{if
        $p+1 \le j \le n$.}
  \end{cases}
\]
\end{thm}

\begin{proof} The isomorphisms follow by combining Theorem~\ref{thm749}
  with \eqref{final1}.  It follows immediately that
  $\Scr=\Sscr{\dalembert^r}$ is noetherian. The fact that $\Scr$ has the
  specified generators then follows from the Fourier transform applied to
  Lemma~\ref{cor76}.

  It remains to prove that $\Scr$ (or, equivalently $\cD(S)$) is a
  primitive maximal order.  Recall from the proof of
  Proposition~\ref{thm79} that $\cD(S)\otimes_\R \C =\cD(\St)\cong
  \cD(R_r)$. Thus  $\cD(\St)$ is primitive and it follows, for example from
  \cite[10.1.9]{MR}, that $\cD(S)$ is primitive.
  If $\cD(S)$ is not a maximal order, there exists an overring
  $\cD(S)\subset T$ with $aTb\subseteq \cD(S)$, for some regular elements
  $a,b\in \cD(S)$. Tensoring with $\C$ shows that
  $a(T\otimes\C)b\subseteq \cD(\St)$ and hence, by Theorem~\ref{mainthm1},
  that $T\otimes\C\subseteq \cD(\St)$. By the faithful flatness of
  $-\otimes_\R \C$, this forces $T=\cD(S)$. Thus, $\cD(S)$ is a maximal
  order.
\end{proof}

\begin{rem} The theorem recovers the generators for
  $\Sscr{\dalembert^r}$ given in \cite[(4.8)]{ss92}.
\end{rem}
 
\begin{cor}\label{cor8922}  Keep the notation of the theorem
  and consider
  $\mathcal{A}=\R[\partial_{X_1},\dots, \partial_{X_n}]/(\dalembert^r_p)$
  under its natural $\Sscr{\dalembert_p^r}$-module
  structure.  If $n$ is even with $r\geq \halfn$, then
  $\mathcal{A} $ has a unique proper factor
 module, which is finite
  dimensional. 
  Otherwise $\mathcal{A}$ is an irreducible module.
  \end{cor}

  \begin{proof} Passing from the $\Sscr{\dalembert_p^r}$-module
    $\mathcal{A}$ to the $\Sscr{\Delta_1^r}$-module
    $\mathcal{A}_\C=\C[\partial_{X_1},\dots, \partial_{X_n}]/(\Delta_1^r)$
    is given by $\phi^{-1}\circ(-\otimes_\R \C)$, where the automorphism $\phi$ of 
    $\Sscr{\Delta_1^r}$ is induced by \eqref{eq71}.
    We claim that the corollary follows by the faithful
    flatness of  $\phi^{-1}\circ(-\otimes_\R \C)$ and Theorem~\ref{thm892}(2).
    
    In order to prove the claim, we first note that, as a
    $\Sscr{\Delta_1^r}$-module, $\mathcal{A}_\C$ satisfies the claimed
    results by Theorem~\ref{thm892}(2).  If $ \mathcal{A}$ is not a simple
    $\Sscr{\Delta_1^r} $-module, then faithful flatness implies that the
    same is true of $\mathcal{A}_\C$ as an $\Sscr{\dalembert_p^r}$-module
    and hence $n$ must be even with $r\geq \halfn$. Moreover, faithful
    flatness and Theorem~\ref{thm892}(2) further imply that $\mathcal{A}$
    will then have a simple infinite dimensional submodule with a simple
    finite dimensional factor module.

  Conversely, if $\mathcal{A}_\C$ is not simple as an
  $ \Sscr{\Delta_1^r}$-module, then it has a finite dimensional factor
  module, as an $\Sscr{\Delta_1^r}$-module and hence as an
  $\Sscr{\dalembert_p^r}$-module. Since
  $\mathcal{A}_\C \cong \mathcal{A}^{(2)}$ as
  $\Sscr{\dalembert_p^r}$-modules, this implies that $\mathcal{A}^{(2)}$
  and hence $\mathcal{A}$ also has a finite dimensional factor module and
  so is certainly not simple.
  \end{proof}

%%%%%%%%%%%%%%%
  \section{Harmonic polynomials and $\cO$-duality}
  \label{sec9}

  The  set of  harmonic polynomials $H=\{p\in A : \Delta_1(p)=0\}$  and its real analogues 
are  fundamental objects with many
  applications, notably  in Lie theory (see, for example,
  \cite{GW,hss,KO}) and physics (see, for example, 
  \cite{bek08,bek11}).  In the latter subject, 
  Dirac \cite{dirac1,dirac2}  constructed a remarkable unitary irreducible representation
  $D\bigl(\frac{n}{2}-1,0\bigr)$ of the Lie algebra $\fso(n,2)$, known
  as the scalar singleton module.  This module can be realized in different
  ways, see \cite{bek08, bek11}: as a highest weight module; as harmonic
  scalar fields $\phi$ (in the sense that $\dalembert_{n-1}.\phi = 0$) on
  the space $\R^{n-1,1}$ which are preserved by $\fso(n,2)$; or as harmonic
  distributions $\vphi$ of weight $1 - \halfn$ on the "ambient space"
  $\R^{n,2}$ (hence $\dalembert_{n,2}.\vphi = 0$). The algebra of
  symmetries of the scalar singleton, as defined in \cite[4.4]{bek08}, acts
  on $D\bigl(\frac{n}{2}-1,0\bigr)$. This algebra can be identified,
  thanks to the results of \cite{ea05} and \cite{vas03}, with the algebra
  $\Scr(\dalembert_{n-1,1})$ (see also Remark~\ref{rem922}~(2)). 
  
 Replacing the condition $\dalembert_{n-1}.\phi = 0$ by
  $\dalembert_{n-1}^r.\phi = 0$, generalises the   singleton module   to
  give \emph{higher singleton modules}.    
 These have been studied in \cite{bekgri}, with
   applications to anti-de Sitter gauge fields and related
   topics.

It is therefore natural to give mathematical description of
these higher singletons or, more generally, of the space
$ \cM_{p,r} = \bigl\{a \in A_\R \ : \ \ \dalembert_p^r(a)=
0\bigr\}$.
This is the topic of this section where we describe the
structure of $\cM_{p,r}$ as a representation of the
orthogonal Lie algebra $\fso(p+1,q+1)$ and show that its
complex analogue is simply the category $\cO$ dual of
$N(\lambda)=R_r$ (see Corollaries~\ref{cor915} and
\ref{cor916} for the precise statements).

 We now make this precise.   We continue with the notation from Sections~\ref{notation} and~\ref{sec2}
  and start with the formal definitions. Since we will be able to
   prove results for both the real and complex cases, we fix
  $1\leq p\leq n$ and keep the notation $\Ftt_p\in A_\R$, and
  $\fs=\fs(p,r)\subset \cD(A_\R)$ and $\dalembert_p\in \cD(A_\R)$ from the
  beginning of Section~\ref{sec7}.   
   The spaces that interest us  are defined as follows.
 
  \begin{defn}
    \label{hharmonics}
    Let $r \in\N^*$. The elements of the subspace
    \[
      \cM_{p,r} = \bigl\{a \in A_\R \ : \ \ \dalembert_p^r(a)= 0\bigr\}
    \]
    will be called \emph{higher harmonics} or \emph{harmonics of level $r$
      in $A_\R$}.  We clearly also have the analogous complex space
    \[
      M_{p,r}= \bigl\{a \in A \ : \ \ \dalembert_p^r(a)= 0\bigr\}.
    \]
  \end{defn}

  Observe that $A= A_\R \otimes_\R \C$ and that $\dalembert_p^r$ has real
  coefficients. Hence, if $a= u + iv \in A$ with $u,v \in A_\R$, we have:
  $\dalembert_p^r(a) = \dalembert_p^r(u) + i \dalembert_p^r(v) = 0$ $\iff$
  $\dalembert_p^r(u) = \dalembert_p^r(v) =0$. Therefore,
  $ M_{p,r} = \cM_{p,r} \otimes_\R \C$.

  In the notation of Theorem~\ref{thm8922} set
  $\tau = \cF\circ\varkappa \circ \psit_\R$. Thus
  $\tau : \ \fs \to \Sscr{\dalembert_p^r}$ is a Lie algebra homomorphism
  that induces an isomorphism
  $\tau : U(\fs)/\cJ_r \cong \Sscr{\dalembert_p^r}$ for the appropriate
  ideal $\cJ_r$.  By Proposition~\ref{thm79}, upon tensoring with the
  complex numbers, we obtain an isomorphism
  $\tau\otimes_\R\C: U(\g)/J_r \cong \Sscr{\Delta_1^r}$.

  \begin{lem}
    \label{lem911} {\rm (1)} The map $\tau$ defines a representation of
    $\fs$ in the space $\cM_{p,r}$ of higher harmonics.
  
    {\rm (2)} Similarly $\tau\otimes_\R\C$ defines a representation of $\g$
    in the space $ M_{p,r}$.
  \end{lem}

  \begin{proof} It suffices to prove (1).  Let $f \in \cM_{p,r}$ and
    $Y \in \fs$. Then, $\mathcal{Y}=\tau(Y)$ is a symmetry of the operator
    $\dalembert_p^r$ whence,
    $\dalembert_p^r(\mathcal{Y}(f))=( \dalembert_p^r\cdot
    \mathcal{Y})(f) \in \cD(A_\R)\cdot \dalembert_p^r(f) = 0$.
    Therefore, $\mathcal{Y}(f) \in \cM_{p,r}$ and so $\tau$ takes values
    in $\End_\R(\cM_{p,r})$, as required.
  \end{proof}
 
\begin{rem}\label{new-module}  In fact the  $\g$-modules $M_{p,r}$ are twists of each other
  (see the proof of Corollary~\ref{cor916} for the details) and so it will
  suffice to prove results for just one of them.  We will
  use the module
\[M_r=M_{n,r}=\{a\in A : \Delta_1^r(a)=0\},
\]
as was also defined in the introduction. For most of this section we
will study this module.   
\end{rem}

We are interested in the following pairing.

\begin{defn}
  \label{pairing}
  The pairing $\ascal{}{}$ is defined by:
  \[
    \ascal{}{} : A\times A\lto \C, \quad \ascal{a}{f} =
    \cterm{\cF(a)(f)}=\cF(a)(f)(0).
  \]
    The \emph{orthogonal} of a subspace $W \subseteq A$ is written
  $W^\perp = \{f \in A : \ascal{W}{f} = 0\}$.
\end{defn}

Recall from Section~\ref{notation} that the Lie subalgebra $\fk \subset \g$
generated by the $E_{jk}$, $2 \le j,k \le \ell$, is isomorphic to $\son$.
The next result is an easy variant on classical results, but since we could
not find an appropriate reference, we include a proof in the Appendix.

\begin{lem}
  \label{lem912} Set $\calIr = \widetilde{\Ft}_p^r A$.
  \begin{enumerate}[\rm (1)]
  \item The bilinear form $\ascal{}{}$ is symmetric and non-degenerate. It
    is also $\fk$-invariant in the sense that
    $\ascal{Y.a}{f} + \ascal{a}{Y.f}= 0$ for all $a,f  \in A$ and
    $Y \in \fk$.
  \item $M_r^\perp = \calIr$ and $\calIr\hskip-2pt^\perp = M_r$.
  \item The form $\ascal{}{}$ induces a non-degenerate $\fk$-invariant
    symmetric pairing $\ascal{}{} : R_r \times M_r \lto \C$.\qed
  \end{enumerate}
\end{lem}

By Lemma~\ref{lem911}, $M_r$ has a $\g$-module structure given by
$a \mapsto \tau(Y).a = \cF(\psit(\vt(Y)))(a)$, for all $Y \in \g$
and $a \in M_r$.  On the other hand, $R_r=N(\lambda)$ is a $\g$-module
through the map $\psi_r : U(\g) \to \cD(R)$ induced by $\psit$. We set
$Y.p = \psi_r(Y)(p)$ for all $Y \in \g$ and $p \in R_r$.  These structures
are related as follows.

\begin{thm}
  \label{thm913}
  The pairing $\ascal{}{} : R_r \times M_r \lto \C$ is $\g$-invariant in
  the sense that
  \[
    \ascal{Y.p}{g} = \ascal{p}{\tau(\vt(Y)).g} =
    \ascal{p}{\cF(\psit(Y))(g)}.
  \]
  for all $Y \in \g$, $p \in R_r$ and $g \in M_r$.
\end{thm}

\begin{proof}
  We first claim that the result is equivalent to proving that
  \begin{equation}
    \label{calcul}
    \cterm{\cF[\psit(Y)(p)](g)} = \cterm{\cF(p)[\cF(\psit(Y))(g)]}.
  \end{equation}
  To see this, note that
  $\ascal{Y.p}{g} = \ascal{\psit(Y)(p)}{g} = \cF[\psit(Y)(p)](g)(0)$
  and $\tau(\vt(Y)).g= \cF(\psit(Y))(g)$. Hence
  $\ascal{p}{\tau(\vt(Y)).g} = \cF(p)[\cF(\psit(Y))(g)](0)$, as
  claimed.

  Write $V_j = \half U_j \Delta_1 - I(d) \ddu{-j}$ as in
  \eqref{proots-2}. The following formul\ae, which are easily checked, will
  be needed in the proof.  \begin{equation}
    \label{ftransforms}
    \cF(E_{ab})= E_{ba}, \quad \text{and} \quad Q_b = \cF(V_b)=  \Ft \ddu{b} - U_{-b} I(d),
    \quad  \text{ for $a,b \in \{0,\pm 2,\dots, \pm \ell\}$.}
  \end{equation}
 
  Recall from \S\ref{notation} that $\g= \fr^- \boplus \fm \boplus \fr^+$
  with $\fm= \C E_{11} \boplus \fk$.  Observe that if~\eqref{calcul} is
  true for $Y,Z \in \g$, then it is true for $[Y,Z]$. Thus, we can reduce
  the verification to the case where $Y$ is a (scalar multiple) of a root
  vector $y_\alpha$ of the Chevalley basis given in
  Proposition~\ref{chevalley2}.  We may therefore prove~\eqref{calcul} by
  considering the cases $Y \in \fk$, $Y\in \fr^+$ and $Y \in \fr^{-}$
  separately.

  Suppose first that $Y = y_\alpha \in \fk$ for $\alpha \in \Phi_1$. From
  Proposition~\ref{chevalley2} we may assume that $Y= E_{ab}$ with
  $a,b \in \{0,\pm 2,\dots, \pm \ell\}$. Then $\psit(Y) = Y= E_{ab}$ and,
  since $Y$ is a derivation, $Y(p)= [Y,p]$. Hence,
  $\cF(Y(p)) = [\cF(p),\cF(Y)]$ and we get
  $\cF(Y(p))(g) = \cF(p)(\cF(Y)(g)) - \cF(Y)(\cF(p)(g))$. But
  $\cF(Y) = E_{ba} =U_b\ddu{a} - U_{-a}\ddu{-b}$, see~\eqref{ftransforms},
  and so $\cF(Y)(q)(0) = 0$ for all $q \in A$ (in particular for
  $q=\cF(p)(g)$). We then get
  $\cterm{\cF(Y(p))(g)} = \cterm{\cF(p)(\cF(Y)(g))}$, as wanted.

  Suppose next that $Y = y_\alpha \in \fr^-$, where either
  $\alpha =-(\vepsilon_1 \pm \vepsilon_b)$ or $\alpha=-\vepsilon_1$. Here,
  since $\psit(y_\alpha)= x_\alpha$, we may assume that $\psit(Y)=U_b$ for
  some $b \in \{0,\pm 2,\dots, \pm \ell\}$; see~Lemma~\ref{chevalley1}.
  Therefore we have
  \[
    \cF[\psit(Y)(p)](g) = \cF(U_bp)(g) = [\cF(p)\cF(U_b)](g) =
    \cF(p)[\cF(\psit(Y))(g)]
  \]
  and, in particular,
  $\cterm{\cF[\psit(Y)(p)](g)} = \cterm{\cF(p)[\cF(\psit(Y))(g)]}$.

  Finally, suppose that $Y = y_\alpha \in \fr^+$, where either
  $\alpha = \vepsilon_1 \pm \vepsilon_b$ or $\alpha=\vepsilon_1$. From
  Lemma~\ref{chevalley1} we may assume that $\psit(Y)= V_b$ for some
  $b \in \{0,\pm 2,\dots, \pm \ell\}$. Recall from~\eqref{ftransforms} that
  $\cF(V_b)= Q_b= \Ft \ddu{b} - U_{-b} I(d)$.  We will need an auxiliary
  calculation: for $p,f \in A$, we claim that
  \begin{equation}\label{star}
    \begin{aligned}
      \quad \cterm{\cF(\Delta_1 p - \Delta_1(p))(f)} & = \biggl[\cF(p
      \Delta_1)(f) + 2\biggl\{U_0\cF(\ddu{0}(p)) + \sum_{\pm j=2}^\ell
      U_{-j}\cF(\ddu{j}(p))\biggr\}(f)\biggr](0)
      \\
      & = \cterm{\cF(p \Delta_1)(f)} = 2\cterm{\Ft \cF(p)(f)} =0.
    \end{aligned}
  \end{equation}
  To see this, notice that
  $\cF(\Delta_1 p - \Delta_1(p)) = \cF(p\Delta_1) +
  \cF\bigl(2\bigl\{\ddu{0}(p)\ddu{0} + \sum_{\pm j=2}^\ell
  \ddu{j}(p)\ddu{-j}\bigr\}\bigr)$.
  Then \eqref{star} follows by applying this equality to $f$ and then
  evaluating at $0$.
  
  We return now to the proof in the case where $\psit(Y) = V_b$. We have
  \begin{equation*}
    \cF[\psit(Y)(p)]  = \cF(V_b(p)) = \cF\Bigl(\half U_b
    \Delta_1(p) - I(d) (\ddu{-b}(p))\Bigr) =
    \half\cF(\Delta_1(p))\ddu{b} - \cF\bigl[I(d)(\ddu{-b}(p))\bigr].
  \end{equation*}
  On the other hand, using
  \[
    \cF\bigl(I(d)\ddu{-b}p\bigr)(g)= [\cF(p)U_{-b}\cF(I(d))](g) =
    [\cF(p)U_{-b}I(d)](g)= \cF(p)\bigl(U_{-b}I(d)(g)\bigr)
  \]
  we get:
  \begin{align*}
    \cF(p)[\cF(\psit(Y))(g)] & = \cF(p)(Q_b(g)) = \cF(p)\bigl(\Ft
                                 \ddu{b}(g) - U_{-b}I(d)(g)\bigr) 
    \\
                               & =\half\cF(\Delta_1p)(\ddu{b}(g))
                                 - \cF\bigl(I(d)\ddu{-b}p\bigr)(g).
  \end{align*}
  It follows that
  \begin{equation}\label{starstar}
    \begin{aligned}
      \quad \cF(p)[\cF(\psit(Y))(g)]- \cF[\psit(Y)(p)](g) & =
      \half\bigl\{\cF(\Delta_1p) -
      \cF(\Delta_1(p))\bigr\}\bigl(\ddu{b}(g)\bigr) \\
      & \phantom{\cF(p)} - \bigl\{\cF\bigl(I(d)\ddu{-b}p\bigr) -
      \cF\bigl[I(d)(\ddu{-b}(p))\bigr]\bigr\}(g).
    \end{aligned}
  \end{equation}
  It is easily checked that
  $I(d)\ddu{-b}p - I(d)(\ddu{-b}(p)) = I(d)p\ddu{-b} + \ddu{-b}(p)\Et_1$.
  Therefore
  \[
    \cF\bigl[I(d)\ddu{-b}p - I(d)(\ddu{-b}(p))\bigr](g) =
    U_{-b}\cF(p)(I(d)(g)) + \Et_1\bigl(\cF(\ddu{-b}(p))(g)\bigr)
  \]
  and this polynomial vanishes at $0$ since
  $\cterm{(U_{-b}q)} = \cterm{\Et_1(q)} = 0$ for all $q \in
  A$.  Equation~\eqref{starstar} then gives 
  $\bigl\{\cF(p)[\cF(\psit(Y))(g)]- \cF[\psit(Y)(p)](g)\bigr\}(0)=
  \half \cF\bigl(\Delta_1p - \Delta_1(p)\bigr)\bigl(\ddu{b}(g)\bigr)(0)$.
  Now applying~\eqref{star} to $f=\ddu{b}(g)$ yields
  $\bigl\{\cF(p)[\cF(\psit(Y))(g)]- \cF[\psit(Y)(p)](g)\bigr\}(0)= 2
  \cterm{\Ft \cF(p)(\ddu{b}(g))} =0$, as required.
\end{proof}

As in \cite[\S2.1]{jantzen}, the dual $M^*$ of a $\g$-module $M$ is endowed
with a $\g$-module structure through
\begin{equation}
  \label{gdual}
  (Y.f)(x)=  f(\vt(Y).x) \quad \text{for all $Y \in \g$, $f \in \g^*$ and $x
    \in M$}.
\end{equation}
Let $\cO$ denote the category of highest weight modules, as in
Section~\ref{notation}.  If $M\in\text{Ob}\,\cO$, then
$M= \bigoplus_{\mu \in \h^*} M^{\mu}$ where $M^\mu$ is the $\mu$-weight
space of $\h$ in $M$. As in \cite[4.10]{jantzen}, the \emph{$\cO$-dual}
$M\spcheck$ of $M$ is then defined by
\[
  M\spcheck = \bigoplus_{\mu \in \h^*} (M\spcheck)^\mu, \quad
  (M\spcheck)^\mu \cong (M^\mu)^*.
\]
By \cite[4.10~(2,3)]{jantzen}, the contravariant functor $M \to M\spcheck$
is exact and satisfies $L(\omega)\spcheck \cong L(\omega)$ for all
$\omega \in \h^*$.

For the next result, recall from Theorem~\ref{thm26} that, if
$R_r=N(\lambda)$ is not a simple $\g$-module, then it has a simple socle
$\socRr\cong L(\mu)$, for the appropriate weight $\mu$, and finite dimensional
factor $\overline{R}=R_r/\socRr\cong L(\lambda)$.

\begin{cor}
  \label{cor915} Under the $\g$-module structure of $M_r$ given by
  Lemma~\ref{lem911}, there is a $\g$-module isomorphism
  $N(\lambda)\spcheck \cong M_r$. Furthermore:
  \begin{enumerate}[{\rm \indent (1)}]
  \item if $n$ is even with $r< \halfn$ or if $n$ is odd, then
    $ M_r \cong L(\lambda)$;
  \item if $n$ is even with $r \ge \halfn$, then $M_r$ has a simple,
    finite dimensional socle $\overline{R}\spcheck \cong L(\lambda)$, with
    quotient $M_r/ \overline{R}\spcheck=\socRr\spcheck\cong L(\mu)$.  The
    formula of the weight $\mu$ is given in
    \eqref{thm26-1}. \end{enumerate}
\end{cor}

\begin{proof}
  By Lemma~\ref{lem23}, $N(\lambda)$ is an object of category $\cO$.
  By Theorem~\ref{thm913}, the pairing $\ascal{}{}$ enables us to define a
  linear injection: 
  \begin{equation*}  
    \xi \; : \; R_r \ \ito \   M_r^* \quad \text{given by $\xi(p)(a)=
      \ascal{p}{a}$ for all $p \in R_r$ and $a \in   M_r$.}
  \end{equation*}
  Moreover, the $\g$-invariance of $\ascal{}{}$ means that
  $\xi(Y.p) = Y.\xi(p)$ under the $\g$-module structure on $M_r^*$ defined
  by~\eqref{gdual}; thus, $\xi$ is $\g$-linear.  To deduce that $\xi$
  restricts to an isomorphism $\xi: R_r \isomto M_r\spcheck$, it suffices
  to check that $\xi : R_r^{\,\nu} \isomto ( M_r^\nu)^*$ for all
  $\nu \in \h^*$. Let $Y \in \fh$, $p \in R_r^{\,\nu}$ and
  $f \in M_r^{\nu'}$. From $\vt(Y)= Y$ one gets that
  $\nu(Y)\ascal{p}{f} = \nu'(Y) \ascal{p}{f}$. Therefore
  $\ascal{R_r^{\,\nu}}{ M_r^{\nu'}} = 0$ for $\nu \ne \nu'$ and it follows
  that $\ascal{}{}$ gives a non-degenerate pairing
  $R_r^{\,\nu} \times M_r^{\nu} \to \C$, as desired. Statements (1) and (2)
  are then consequences of the $\cO$-duality and the structure of the
  $\g$-module $N(\lambda) = R_r$ obtained in Theorem~\ref{thm26}.
\end{proof}

We now return to the real case, as discussed at the beginning of the
section and describe the natural analogues of Corollary~\ref{cor915}.  
 
\begin{cor} \label{cor916} Fix $r\geq 1$ and $1\leq p\leq n$ and consider
  the $\fs(p,r)$-module $\calM_{p,r}$ from Definition~\ref{hharmonics} and 
  Lemma~\ref{lem911} .  Then:
  \begin{enumerate}[{\rm \indent (1)}]
  \item if $n$ is even with $r< \halfn$ or if $n$ is odd, then $\cM_{p,r}$
    is simple;
  \item if $n$ is even with $r \ge \halfn$, then $\cM_{p,r}$ has a simple,
    finite dimensional socle $S'$,
    with an infinite dimensional simple quotient $\cM_{p,r}/S'$.
    \end{enumerate}
\end{cor}
 
\begin{proof}   We first want to obtain an analogue of
  Corollary~\ref{cor915} for any $M_{p,r}$.  To do this, recall that the
  automorphism $\phi$ from \eqref{eq71} satisfies
  $\phi(\Delta_1)=\dalembert_p$. Hence if we define the \emph{twisted
    $\cD(A)$-module} structure ${}^\phi \hskip -3pt A$ on $A$ by
  $\theta\circ a = \phi(\theta) a$, for $\theta\in \cD(A)$ and $a\in A$,
  then a simple computation shows that ${}^\phi M_{r}=M_{p,r}$.  Since
  twisting preserves submodule lattices and dimensions, it follows from
  Corollary~\ref{cor915} that the statement of the present corollary holds
  for the $\g$-modules $M_{p,r}$.
 
  As was observed in \eqref{hharmonics},
  $\calM_{p,r}\otimes_\R \C= M_{p,r}$ for each $p,r$.  Thus, just as in
  the proof of Corollary~\ref{cor8922}, faithful flatness then implies that
  the desired results for the $\calM_{p,r}$ follow from the corresponding results for the $M_{p,r}$.
\end{proof}

 %%%%%%%%%% 
\section{Conformal densities and the ambient construction}
  \label{howe}

  The $\g$-module $M_1=H$ of harmonic polynomials from
  Definition~\ref{hharmonics} is an ``incarnation'' of the
  scalar singleton module introduced by Dirac through the
  ambient method, cf.~\cite{dirac1,dirac2, eagr, bek11}.  In
  this section, we will briefly recall this construction,
  and its generalisation to the higher harmonics, and then
  show that these objects have a natural algebraic analogue
  that gives another incarnation of the module
  $M_r\cong N(\lambda)\spcheck$.  This will prove
  Theorem~\ref{thmE} from the introduction.

  The ambient method allows one to work with the standard
  representation of $\fso(p+1,q+1)$ by linear vector fields
  on $\R^{n+2}$ and leads to a conformal model of the
  representation of $\g$ on $M_r$.  To see this, one first
  constructs the following conformal compactification of
  $\R^{p,q}$.  Let $\Qt= U_{1}U_{-1} + \Ft_p$ be the
  quadratic form of signature $(p+1,q+1)$ on $\R^{n+2}$ with
  coordinate functions $U_1,U_{-1},X_1,\dots,X_n$.  and
  write $\calC = \{z \in \R^{n+2} : \Qt(z)= 0\}$ for the
  corresponding light cone.  Then the associated projective
  quadric $\cQ= \{\Qt = 0\} \subset \mathbb{RP}^{n+1}$ is a
  conformal compactification of $\R^{p,q}$, under the identification of  $\R^{p,q}$ 
    with the open cell $\cQ\cap \{U_1= 1\}$.   (The space $\cQ$ is sometimes known as the
  boundary of the anti-de Sitter space when
  $(p,q)= (n-1,1)$).  
In this setting a conformal density of weight $w$ on $\cQ$
  can be viewed as a homogeneous functions of weight $w$ on
  the light cone $\calC$.  The conformal Laplacian is the
  operator induced by the Laplacian
  $\Delta = \ddu{1}\ddu{-1} + \dalembert_p$. It acts on
  densities of weight $-\halfn +1$, as shown for example in
  \cite[Propositions~4.4 and 4.7]{eagr} and
  \cite[Proposition~2.1]{gjms}, and the corresponding space
  of harmonic densities of weight $-\halfn +1$ is then an
  $\fso(p+1,q+1)$-module.  For Minkowski space-time, when
  $(p,q)=(n-1,1)$, this $\fso(n,2)$-module corresponds to the
  \emph{scalar singleton module} $D(\halfn -1,0)$ of
  \cite{bek11}.

Analogous questions arise for homogeneous functions on the
``generalised light cone'' $\{\Qt^r =0\}$; equivalently for
densities $\vphi$ of weight $w$ for which
$\Delta\cdot\vphi =0\pmod{\Qt^r}$. By \cite{eagr} and
\cite{gjms}, again, the above construction works for
densities of weight $- \halfn +r$.  This produces an
$\fso(p+1,q+1)$-module, which in the case of the Minkowski
space time corresponds to the higher-order singleton as
defined in \cite{bekgri}.

In this section we aim to give an algebraic version of this
construction and relate it to the $\g$-modules $R_r$ and
$M_r$ obtained in the previous sections. Roughly speaking,
in our setting (and under the notation of
Section~\ref{notation}), the generalised light cone
$\{\Qt^r =0\}$ is replaced by a factor $B/\Qt^rB$ of a
polynomial ring in $(n+2)$ variables, while the ``conformal
Laplacian'' is replaced by the operator $\Delta \in \cD(B)$
and the ``densities'' by homogeneous polynomials on (a
finite extension of) $B/\Qt^rB$.

 \medskip
We now formalise this approach.
Adopt the notation of Sections~\ref{notation}
and~\ref{sec1}; in particular, we use the presentation of
$\g$ provided by Proposition~\ref{prop611}.     
Following \eqref{nroots},   
write $B=A[U_{-1},U_{1}]$ for
$A=\C[U_0,U_{\pm 2}, \dots,U_{\pm \ell}]$ and let
$B[U_{1}^{-1}]    = A[U_{-1},U_{1}^{\pm 1}] $
  be the localization of $B$ at the powers of $U_1$.  In the notation of
  \eqref{eq610}, $U_{-1} = (\Qt - \Ft)U_1^{-1}$, for
  $ \Qt= \sum_{j=1}^\ell U_jU_{-j} + \half U_0^2$, and so
  $ B[U_{1}^{-1}]=A[\Qt,U_{1}^{\pm 1}]$. We also need to adjoin a square
  root of $U_1$ to $B[U_{1}^{-1}]$; thus, formally, set
  \begin{equation*}
    S=B[U_1^{-1},T^{\pm 1}]/\bigl(T^2 - U_1\bigr)
  \end{equation*}
  for an   indeterminate $T$.  The class of $T$ in $S$ is denoted by $t$.  
  Notice that
  \begin{equation}
    \label{eq640c}
    S = A[U_{-1},t^{\pm 1}]= A[\Qt,t^{\pm 1}] = \C[t^{\pm
      1},U_{\pm 2},\dots,U_{\pm \ell},U_0,\Qt].
  \end{equation}
  Clearly, $\Qt$, $t$ are indeterminates over $A$ and
  $S = B[U_1^{-1}] \bigoplus B[U_1^{-1}]t$ is a free $B[U_1^{-1}]$-module
  of rank 2.

\begin{rem}
\label{rem641}
A derivation $D \in \Der_\C(B)$ extends to $B[U_1^{-1},T^{\pm 1}]$ by
setting $D(T)= \frac{1}{2T}D(U_1)$. Then $D((T^2 -U_1)f) = (T^2 -U_1)D(f)$
for any $f \in B[U_1^{-1},T^{\pm 1}]$ and so $D$ also defines a derivation
on $S$. In particular, the derivations $E_{ab} \in \g$, as defined in
\eqref{eq610a}, act on $S$ and this endows $S$ with a natural $\g$-module
structure.  Furthermore, the derivation $\ddu{1}$ extends to $S$ by
$\ddu{1}(t) = \frac{1}{2t}$ and so we can write
$\ddu{1} = \frac{1}{2t}\dt \in \Der_\C(S)$.  With this notation and that
from (\ref{eq610}, \ref{eq610a}), we can therefore write
\[
\Et =  \frac{t}{2}\dt + \sum_{j \ne 1} U_j\ddu{j},
\qquad E_{j1} = \frac{U_j}{2t} \dt - U_{-1}\ddu{-j} =  \frac{U_j}{2t} \dt -
\frac{1}{t^2}(\Qt - \Ft)\ddu{-j}
\]
as elements of  $\Der_\C(S)$.   
\end{rem}

We now fix an integer $r \in \N^*$ and set
\begin{equation*}
  \label{eq650} 
\overline{S} = \overline{S}_r = S / \Qt^r  S. 
\end{equation*}
Let $D \in \Der_\C(S)$ be a derivation such that $D(\Qt)= c\,\Qt$ for some
$c\in S$; for example $D= \Et$ or $E_{ab}$. Then,
$D$ induces a derivation, again denoted 
$D $, in $ \Der_\C(\overline{S})$.
In particular, the algebra $\overline{S}$ also has a $\g$-module structure. As such, 
$S$ and $ \overline{S}$ are graded by the weights of the Euler operator $\Et$ and 
we can make the following definition.

\begin{defn}
    \label{homogeneous}
    Let $p \in \Q$. An element $u \in C= S$, or $C=\overline{S}$ is said to
    be \emph{homogeneous of weight $p$} if $\Et(u) = p u$. The space of
    homogeneous elements of weight $p$ is  denoted by $C(p)$.  Set $\wt(0)= 0$ and 
    $\wt(u)= p$ if $0 \ne u \in C(p)$; thus
    $\Et(u) = \wt(u) u$. If $V$ is a subspace of $C$,  set
    $V(p)= C \cap C(p)$.
  \end{defn}

  Notice that $U_j \in S(1)$ for all $j$ and that $t \in S(\half)$.  
  Hence, $S= \bigoplus_{\beta \in \halfZ} S(\beta)$ and, since
  $\Qt^r$ is homogeneous of degree $2r$, we have: 
\begin{equation*} 
  \overline{S} =  \bigoplus_{\beta \in \halfZ} \Sbar{\beta}, \qquad \text{where}\ 
  \Sbar{\beta}= S(\beta) / \Qt^r S(\beta -2r). 
\end{equation*}
We can think of the elements  of $\Sbar{\beta}$ as ``half-densities'' on
the subscheme $\{\Qt^r = 0\} \subset \C^{n+2}$.

\begin{lem}
  \label{lem642}
Let   $\Qtbar$ denote  the class of $\Qt$ in $\overline{S}$ and 
pick $\beta \in \halfZ$. Then:
\\
{\rm (1)} $\Sbar{\beta}$ is a $\g$-submodule of the $\g$-module
$\overline{S}$;
\\
{\rm (2)} one has $\overline{S} = \bigoplus_{j=0}^{r-1} A[t^{\pm 1}]
\Qtbar^{j}$  and $\Sbar{\beta} = \bigoplus_{j=0}^{r-1}
A[t^{\pm 1}](\beta - 2j) \Qtbar^{j}$   as  $A[t^{\pm 1}]$-modules. 
\end{lem}
 
\begin{proof} This follows from~\eqref{eq640c} combined with the fact that
  $[\Et,Y] = 0$ for all $Y \in \g$.
 \end{proof}

 By Remark~\ref{rem641}, 
the Laplacian $\Delta = 2 \sum_{j=1}^\ell \dpartial{U_j}\dpartial{U_{-j}} +
\partial^{\; 2}_{U_0}$ extends trivially to $B[U_1^{-1},T]$ 
and then restricts to   $S$. This differential operator on $S$, still denoted by
$\Delta$, can also be written:
\begin{equation*} 
\Delta = \frac{1}{t} \dt\ddu{-1} + \Delta_1, \qquad\text{for}\  \ \Delta_1= 2\sum_{j \ge
    2} \dpartial{U_j}\dpartial{U_{-j}} + \partial^{\; 2}_{U_0}.
\end{equation*}
As usual, the set of   \emph{harmonic elements in $S$} is defined to be  $\cH(S) = \{f
\in S  : \Delta(f)= 0\}$.
Since $[\Et,\Delta] = -2\Delta$, this  is a graded subspace of $S$ and
$\cH(S)= \bigoplus_{\beta \in \halfZ} \cH(\beta)$.  Clearly,
 $\Delta (S(\beta) ) \subseteq S(\beta -2)$ for all
$\beta \in \halfZ$.  

We now want to find $\beta$ such that $\Delta$ induces
an operator from $\Sbar{\beta}$ to $\Sbar{\beta -2}$.
The next result can be compared with \cite[Propositions~4.4 \& 4.7]{eagr},
\cite[Proposition~2.1]{gjms} and \cite[II~Proposition]{bz91}.

\begin{prop}
  \label{prop641}
  Let $\delta = -d = - \frac{n}{2} + r$. Then, the Laplacian $\Delta$
  induces an operator
  \[
  \Deltabar_\delta = \Deltabar \ : \ \Sbar{\delta} = \frac{S(\delta) }{ \Qt^r
  S(\delta -2r)} \ \lto \  \Sbar{ \delta -2} = \frac{S(\delta -2) }{ \Qt^r S(\delta
  -2(r+1))}
  \]
  by the formula $\Deltabar(\bar{f}) = \overline{\Delta(f)}$ for all
  $f \in S(\delta)$.
\end{prop}

\begin{proof}
  As in~\eqref{prop17-1} one   shows that, for all $k \in \N^*$ and  
  homogeneous elements $u\in S$ one has
  \begin{equation}\label{eq122}
  \Delta(u\Qt^k)= \Delta(u) \Qt^k + k [n+2+2\wt(u) + 2(k-1)]u \Qt^{k-1}.
  \end{equation}
  We are looking for $\delta \in \halfZ$ such that
  $\Delta(u \Qt^r) \in \Qt^r S(\delta -2(r+1))$ for all
  $u \in S(\delta -2r)$.  
   However, by \eqref{eq122}, if $\delta =   -\halfn +r$ then  
  \[
  \Delta(u\Qt^r)= \Delta(u) \Qt^r + r [n+2 +2\delta -2(r+1)]u \Qt^{r-1}=
  \Delta(u) \Qt^r
  \]
  for all $u \in S(\delta -2r)$. This proves the proposition.
\end{proof}

Thanks to Proposition~\ref{prop641} we can define the
following subspace of ``harmonic elements of degree
$\delta$'', or ``harmonic conformal densities of weight $\delta$''.
It gives a (complex) version of the ``scalar singleton
module'' defined by Dirac; see for example
\cite[\S3.5]{bek11}.

\begin{defn}
  \label{def642}     
    Let $r \in \N^*$ and $\delta = - \frac{n}{2} + r$ and set:
   \[
  N_{\lambda} = \Ker \Deltabar_\delta = \bigl\{\bar{f} \in \Sbar{\delta} \
  : \ \Deltabar_\delta(\bar{f}) = 0\bigr\}.
  \]
 Since  $[Y,\Delta] = 0$ for all $Y \in \g$, it is easy to see that
  $N_{\lambda}$ is a $\g$-submodule of $\Sbar{\delta}$. \end{defn}

It follows from Lemma~\ref{lem642}(2) that any element
$\bar{f} \in \Sbar{\beta}$, for $f \in S(\beta)$, can be uniquely written:
\begin{equation}
  \label{eq214b}
  \bar{f} = \sum_{j=0}^{r-1} f_j \Qtbar^j \quad \text{with \ $f=
    \sum_{j=0}^{r-1} f_j \Qt^j$ and $f_j
    \in A[t^{\pm 1}](\beta - 2j)$}.
\end{equation}
  Let $a \in A$ and
$\nu \in \Z$. Since $\ddu{-1}(at^\nu)= 0$ and $\Delta(a)= \Delta_1(a)$ it
follows that $\Delta(a t^\nu)= \Delta_1(a) t^\nu$. Applying this observation to the  
$f_j$ shows  that $\Delta(f_j) = \Delta_1(f_j)$ for 
$0\leq j\leq r-1$.

\begin{prop}
\label{thm221} {\rm (1)} Let $f = \sum_{j=0}^{s} f_j\Qt^j \in S(\delta)$
with $0 \le s \le r-1$ and $f_j \in A[t^{\pm 1}](\delta -2j)$ as
in~\eqref{eq214b}. Then,
  \begin{equation*}
    \Deltabar(\bar{f})= 0 \iff   \Delta(f)= 0  \iff 
    \begin{cases}
      \Delta_1(f_j) = -2(j+1)(r-(j+1)) f_{j+1} \ \; \text{ if \, $0 \le j
        \le s-1$;} \\
           \Delta_1(f_{s}) =0.
    \end{cases}
  \end{equation*}
  Consequently,
  \[
  N_\lambda = \bigl\{\bar{f} \in \Sbar{\delta} \ : \ \Delta(f)= 0\bigr\} =
  \overline{\cH(S)(\delta)} \subset \Sbar{\delta}.
  \]
  {\rm (2)} There exists an isomorphism of $\fm$-modules
  \[
  \ev0 : \ N_\lambda \, \longisomto \, \calS_\lambda =\bigl\{a \in A[t^{\pm
    1}](\delta) \, : \, \Delta_1^r(a) = 0 \bigr\}
  \]
  given by $\ev0(\bar{f}) = f_0$, when
  $f = \sum_{j=0}^{r-1} f_j\Qt^j \in S(\delta)$ with
  $f_j \in A[t^{\pm 1}](\delta -2j)$.
\end{prop}

\begin{proof}
  (1) By~\eqref{eq122} we have:
\[
  \Delta(f)= \sum_{j=0}^{s} \Delta(f_j \Qt^j) = \Delta_1(f_j)\Qt^{s} +
  \sum_{j=0}^{s-1} \bigl\{\Delta_1(f_j) + (j+1) [2\wt(f_{j+1}) + N +2j]
  f_{j+1}\bigr\}\Qt^{j}.
\]
  It follows from $\Delta_1(f_j) \in A[t^{\pm 1}]$ that
  $\Delta(f) \in \bigoplus_{j=0}^{s} A[t^{\pm 1}] \Qt^{j}$. Since
  $\wt(f_{j+1}) = \delta -2(j+1)$ if $f_{j+1}\ne 0$, one deduces that 
  $\Delta(f)= 0$ is equivalent to the following system of equations in the
  ring $A[t,t^{-1}]$:
  \[
  \begin{cases}
    \Delta_1(f_{s}) =0
    \\
    \Delta_1(f_j) = -2(j+1)(r-(j+1)) f_{j+1} \ \; \text{if $0 \le j \le
      s-1$}.
  \end{cases}
  \]
  By definition, $\Deltabar({\bar{f}})= \overline{\Delta({f})}$ and so 
  $\Deltabar({\bar{f}})=0$ is equivalent to
  $\Delta(f) \in S(\delta - 2(r+1)) \Qt^{r}$. Recalling that
  $\Delta(f) \in \bigoplus_{j=0}^{r-1} A[t^{\pm 1}] \Qt^{j}$ and using
  $S= \bigoplus_{j \in \N} A[t^{\pm 1}] \Qt^{j}$ one deduces that
  $\Delta(f)= 0$ if and only if $\Deltabar({\bar{f}})= 0$.
  
  Finally, if $\bar{x} \in N_\lambda$ we can find $f$ as in~\eqref{eq214b} such
  that $\bar{x} = \bar{f}$. By definition, $\overline{\Delta}(\bar{f})=0$ and hence $\Delta(f) =0$.
  The equality $N_\lambda = \overline{\cH(S)(\delta)}$ follows.

   (2) By induction, the equalities $\Delta_1(f_j) =
  -2(j+1)(r-(j+1)) f_{j+1}$ for $0 \le j \le r-2$ yield $\Delta_1^p(f_0) =
  c_p f_p$, where $c_p= (-2)^p p! (r-1)\cdots(r-p)$, for $0\leq p\leq r-1$.
  It follows from~(1) that the map $\bar{f} \mapsto f_0$ with $f \in
  \cH(S)(\delta)$ takes values in $\cS_\lambda$ and is injective. This
  also implies that if $a \in \cS_\lambda$ one can obtain $f =
  \sum_{j=0}^{r-1} f_j \Qt^j \in \cH(S)(\delta)$ by setting $f_0 = a$ and 
    $f_{p} = \frac{1}{c_p} \Delta_1^p(a)$ for $1\leq p\leq r-1$. This yields
  a vector space  isomorphism $N_\lambda \isomto \cS_\lambda$.

  Recall from~\S\ref{notation} that $\fm = \C E_{11} \boplus \fk$ where
  $E_{11} = \frac{t}{2}\dt - U_{-1}\ddu{-1}$ and
  $\fk = \sum_{i,j \notin\{\pm 1\}} \C E_{ij}$. Clearly, therefore, $A[t^{\pm 1}]$ is
  $\fm$-stable.  From
  $[E_{11},\Delta_1]=0=[\fk,\Delta_1] $ it follows  that $\cS_\lambda$ is an
  $\fm$-module. Let $Y \in \fm$ and $\bar{f} \in N_\lambda$ with
  $f=\sum_{j=0}^{r-1} f_j\Qt^j$ as above. Since $Y.\Qt = Y(\Qt) =0$ one has
  $Y.\sum_{j=0}^{r-1} f_j\Qt^j = \sum_{j=0}^{r-1} (Y.f_j) \Qt^j$ with
  $Y.f_j \in A[t^{\pm 1}](\delta -2j)$.  Hence
  $\ev0(Y.f)=(Y.f)_0 = Y.f_0 = Y.\ev0(f)$ and so  the linear isomorphism
  $\ev0 : N_\lambda \isomto \cS_\lambda$ is indeed an isomorphism of $\fm$-modules.
\end{proof}

As we next show, this theorem implies that, as
$\fk$-modules, $N_\lambda$ is isomorphic to the module
$M_r =M_{n,r} = \bigl\{a \in A : \Delta_1^r(a) =0\bigr\}$  from Remark~\ref{new-module}. 
Before proving
this, note that as $\Delta_1^r(A(m)) \subset A(m-2r)$ for
all $m \in \N$, the space $M_r$ is a graded subspace of
$A$. Thus any $a \in M_r$ can be written as
$a= \sum_{m \in \N}a(m)$ with $a(m) \in M_r(m)$.  On the
other hand, by Proposition~\ref{thm221}(2) and the fact that
$\wt(t)=\half$, any $g \in \cS_\lambda$ may be written as
$g= \sum_{m\in \N} g(m)t^{2(\delta -m)}$ with
$g(m) \in A(m)$.

\begin{cor}
  \label{cor643}
The $\fk$-modules $M_r \subset A$ and $\cS_\lambda \subset A[t^{\pm 1}]$
are isomorphic via the map:
\[
\sigma \ : \ \cS_\lambda \longisomto M_r; \qquad   \sum_{m \in
  \N}a(m) t^{2(\delta - m)} \mapsto \sum_{m \in \N} a(m).
\]
In particular, 
the map $\sigma \circ \ev0 : N_\lambda \isomto M_r$ is an isomorphism of
$\fk$-modules. 
\end{cor}

\begin{proof} Let $g=\sum_{m \in \N}a(m) t^{2(\delta - m)} \in \cS_\lambda$
  and recall that $\Delta_1(at^\beta) = \Delta_1(a)t^\beta$ for all
  $a \in A, \beta \in \Z$. Then,
  $\Delta_1^r(g)= \sum_m \Delta_1(a(m)) t^{2(\delta - m)} = 0$ forces
  $\Delta_1^r(a(m)) = 0$ for all $m$, hence $\sum_m a(m) \in M_r$.
  Conversely, the same computation shows that if $\sum_m a(m) \in M_r$ then
  $\sum_{m \in \N}a(m) t^{2(\delta - m)}\in \cS_\lambda$. Thus $\sigma$ is
  a linear bijection.  The $\fk$-linearity of $\sigma$ is consequence of
  $Y.t^{\beta} = 0$ for all $\beta$ and $Y \in \fk$. The second assertion
  then follows from Proposition~\ref{thm221}(2).
\end{proof}

By Lemma~\ref{lem911}, $M_r$ is   a
$\g$-module so it is natural to ask whether  $\sigma \circ \ev0$ is actually $\g$-linear in
this corollary. This is true and forms the main result of the section.

 \begin{thm}
   \label{thm921}
   The map  $\sigma\circ\ev0 : N_\lambda \to M_r$ is
   an isomorphism of $\g$-modules. Equivalently,
   \[
\sigma(\ev0(Y.\bar{f})) = \tau(Y).g \quad \text{for all $Y \in
  \g$ and $g=\sigma(\ev0(\bar{f})) \in M_r$, for $\bar{f}\in N_\lambda$}.
\]
 \end{thm}

 \begin{proof} 
Throughout the proof we write
  $\bar{f} = \sum_{j=0}^{r-1} f_j \Qtbar^j$ with
  $f= \sum_{j=0}^{r-1} f_j \Qt^j$ as in \eqref{eq214b}; thus
  $\ev0(\bar{f})=f_0 \in \cS_\lambda$ by Proposition~\ref{thm221}.
  Let $Y \in \g$ and recall that the derivation $Y$
  satisfies $Y(\Qt)= 0$. Thus, $Y.f= \sum_j (Y.f_j) \Qt^j$
  and it follows from the definition of $\ev0$ that
  $\ev0{(Y.\bar{f})} = \ev0{(Y.f_0)}$.  Using this, the
  theorem is then obviously equivalent to proving:
   \begin{equation}
     \label{theclaim}
     \sigma(\ev0(Y.f_0))=  \tau(Y).\sigma(f_0) \quad \text{for all $Y \in
       \g$ and $f_0 \in \cS_\lambda$}.
   \end{equation}

   By  Corollary~\ref{cor643},
   $f_0 = \sum_{m \in \N} a(m)t^{2(\delta -m)}$, with $a(m) \in A(m)$ for
   all $m$, and   $\sigma(f_0)= \sum_m a(m)$. For simplicity we set
   $\gamma(m)= 2(\delta -m)$. Thus 
\[
Y.f_0 = \sum_m (Y.a(m)) t^{\gamma(m)} + \sum_m a(m) (Y.t^{\gamma(m)}) =
\sum_m (Y.a(m)) t^{\gamma(m)} + \sum_m \gamma(m)a(m) (Y.t)  t^{\gamma(m)-1} 
\]
while $\tau(Y).\sigma(f_0) = \sum_m \tau(Y).a(m)$.  As
usual, in order to prove that~\eqref{theclaim} holds, it
suffices to prove it when $Y$ is a root vector in $\g$,
hence a (scalar multiple) of some $y_\alpha$ from the
Chevalley system given in Proposition~\ref{chevalley2}. As
in the proof of Theorem~\ref{thm913}, we will consider the
three cases $Y\in \fk$, $Y \in \fr^-$ and $Y \in \fr^+$
separately.

Suppose first that $Y = y_\alpha \in \fk$ for
$\alpha \in \Phi_1$.  We may assume that $Y= E_{ab}$ with
$a,b \in \{0,\pm 2,\dots, \pm \ell\}$. Then $E_{ab}.t= 0$
and
$\tau(Y) = \cF(\psit(\vt(Y))) = \cF(\psit(E_{ba})) =
\cF(E_{ba})= E_{ab} =Y$,
by\eqref{ftransforms}. Thus \eqref{theclaim} is equivalent
to the $\fk$-linearity of $\sigma$, which follows from
Corollary~\ref{cor643}.

Suppose next that $Y = y_\alpha \in \fr^-$, where either
$\alpha =-(\vepsilon_1 \pm \vepsilon_b)$ or
$\alpha=-\vepsilon_1$. Hence, by
Proposition~\ref{chevalley2}, $Y = y_{\alpha}$ is either
equal to $E_{\mp b,1}$ when
$\alpha =-(\vepsilon_1 \pm \vepsilon_b)$ and
$b \in \{2,\dots, \ell\}$, or $Y = \sqrt{2}E_{0,1}$ when
$\alpha= -\vepsilon_1$. By the definition of $\vt$ and
$\psit$ we have $\psit(\vt(Y))= x_{-\alpha} \in
\fr^+$.
Using Lemma~\ref{chevalley1}, this root vector is either
$V_{-b}$ when $\alpha = -(\vepsilon_1 - \vepsilon_b)$, or
$V_b$ when $\alpha= -(\vepsilon_1 + \vepsilon_b)$ or
$\sqrt{2}{V_0}$ when $\alpha= -\vepsilon_1$.  Recall that
$\cF(V_a)= Q_a= \Ft \ddu{a} - U_{-a} I(d)$. Therefore, up to
the scalar $\sqrt{2}$, we may assume that $Y= E_{j1}$ and
$\tau(Y) = Q_{-j}$ with $j \in \{0,\pm 2, \dots,\pm \ell\}$.
Since $E_{j1}= \dfrac{U_j}{2t}\dt - U_{-1}\ddu{-j}$ and
$a(m) \in A= \C[U_0, U_{\pm 2}, \dots,U_{\pm \ell}]$, we
obtain
\[
E_{j1}.f_0 = - \sum_m U_{-1}\ddu{-j}(a(m)) t^{\gamma(m)} + \sum_m
\frac{\gamma(m)}{2} U_j a(m) t^{\gamma(m) -2}.  
\]
Recalling from \eqref{eq610} that $U_{-1} =  \dfrac{1}{U_1}(\Qt- \Ft)
=t^{-2}(\Qt- \Ft)$, it follows that  
\begin{align*}
E_{j1}.f_0 & =- \sum_m (\Qt - \Ft)\ddu{-j}(a(m)) t^{\gamma(m) -2} + \sum_m
\frac{\gamma(m)}{2} U_j a(m) t^{\gamma(m) -2} 
\\
& = \sum_{m} \Bigl(\frac{\gamma(m)}{2} U_j a(m) + \Ft\ddu{-j}(a(m)) \Bigr)
  t^{\gamma(m) -2} - \Qt \sum_m \ddu{-j}(a(m))  t^{\gamma(m) -2}.
\end{align*}
Thus, as $\gamma(m) =2(\delta -m) =-2(d+m)$, 
we  have  $$\ev0(E_{j1}.f_0) = \sum_{m} \bigl(-(d+m)U_j +
\Ft\ddu{-j} \bigr)a(m)t^{\gamma(m) -2}$$ and 
$\sigma(\ev0(E_{j1}.f_0)) = \sum_{m} \bigl(-(d+m) U_j +
\Ft\ddu{-j}\bigr)(a(m))$. 
Since $I(d)(a(m))= (m+d)a(m)$, observe that 
\begin{equation*}
Q_{-j}. \sigma(f_0)=Q_{-j}.\sum_m a(m) = \sum_m Q_{-j}.a(m)  = \sum_m \bigl(\Ft \ddu{-j}(a(m))
- (d+m) U_j a(m)\bigr)   
\end{equation*}
and we obtain $\sigma(\ev0(E_{j1}.f_0)) =Q_{-j}. \sigma(f_0)$, as required.

Finally, suppose that $Y = y_\alpha \in \fr^+$, where either
$\alpha = \vepsilon_1 \pm \vepsilon_b$ or
$\alpha=\vepsilon_1$. From Lemma~\ref{chevalley1} we see
that we may take $Y= E_{1b}$; using
Lemma~\ref{chevalley1} and the definition of $\cF$, we may
assume that $Y= E_{1j}$, whence $\tau(Y) = \ddu{j}$ with
$j \in \{0,\pm 2, \dots,\pm \ell\}$. Recall that 
$E_{1j} = U_1\ddu{j} - U_{-j}\ddu{-1} = t^2 \ddu{j} -
U_{-j}\ddu{-1}$;
from $E_{1j}(t)= 0$ and $\ddu{-1}(a(m))=0$, we get that
$E_{1j}.f_0 = \sum_m \ddu{j}(a(m)) t^{\gamma(m) +2}$.  On
the other hand we have
$\ddu{j}.\sigma(f_0)= \sum_{m} \ddu{j}(a(m))$ and it follows
that
$\sigma(\ev0(E_{1j}.f_0)) = \sum_m \ddu{j}(a(m)) =
\ddu{j}.\sigma(f_0)$, as desired.
\end{proof}

In summary, by combining Corollary~\ref{cor915} with Theorem~\ref{thm921},
we have the following relationship between the three main $\g$-modules
appearing in this paper: the module $N(\lambda)= R_r$ of regular functions
on the scheme $\{\Ft^r= 0\} \subset \C^n$, the module $M_r$ of higher
harmonic polynomials on $\C^n$, and the module $N_\lambda$ of harmonic conformal
densities of weight $-\halfn + r$ on $\C^{n+2}$.

\begin{cor}
  \label{summary}
There exist  isomorphisms of $\g$-modules:
\[
N(\lambda) \; \cong \; N_\lambda\spcheck \quad \text{or, equivalently,}
\quad R_r \cong M_r\spcheck. 
\]
Moreover:     
\begin{enumerate}[{\rm \indent (1)}]
  \item if $n$ is odd or if $n$ is even with $r<
\halfn$, then $R_r \cong L(\lambda) \cong M_r$;
\item if $n$ is even with $r \ge \halfn$, then $R_r$ has a simple socle $\socRr$
  isomorphic to $L(\mu)$, where $\mu$ is given  by
 \eqref{thm26-1},  and the quotient $Q_r=R_r/\socRr$ is isomorphic to the
  simple finite dimensional module $L(\lambda)$.
\item if $n$ is even with $r \ge \halfn$, then $M_r$ has a simple (finite
  dimensional) socle $Q_r\spcheck \cong L(\lambda)$ and the quotient
  $\socRr\spcheck = M_r/Q_r\spcheck\cong L(\mu)$.\qed
\end{enumerate}
\end{cor}

We conclude this section with  remarks about the module $N_\lambda$ and its
annihilator.

\begin{rem}
  \label{rem922}
  (1) Recall that $\delta= -d= -\halfn + r$. Since
  $t^{2\delta} \in \cH(S)(\delta)$, the class
  $e_\lambda=[t^{2\delta}]\in \Sbar{\delta}$ belongs to $N_\lambda$. It is
  easy to see that $e_\lambda$ is a highest weight vector in $N_\lambda$,
  with weight $\lambda =-d \vepsilon_1 =\delta \vepsilon_1$. In more
  detail:
\begin{enumerate}[\indent\rm (i)] 
\item  $E_{ab}.t^{2\delta} = 0=E_{c,0}.t^{2\delta}$ for $1 \le a < b \le \ell$ and
 $1 \le c \le \ell$. Thus,  by Proposition~\ref{chevalley2},   $\fn^+.e_\lambda=
  0$. 
\item   $E_{11}.t^{2\delta} = \frac{U_1}{2t}\dt(t^{2\delta}) = \delta
  t^{2\delta}$ and $E_{jj}.t^{2\delta} = 0$, for $2\leq j\leq\ell$; showing that
  $e_\lambda$ has weight $\lambda= \delta \vepsilon_1$.
\end{enumerate}
Therefore,   if $n$ is even with $r< \halfn$ or if $n$ is odd,  one gets
$N_\lambda= U(\g).e_\lambda$.      

Assume now that $n$ is even with $r \ge \halfn$. Then,
$F\spcheck=U(\g).e_\lambda \cong L(\lambda)$ is finite dimensional. The
quotient $N_\lambda/F\spcheck$ is isomorphic to $L(\mu)$ and, by similar
computations, it is easily seen that the class
$e_\mu \in N_\lambda/F\spcheck$ of the element
$U_2^{\delta+1}t^{-2}= U_2^{\delta +1}U_1^{-1} \in \cH(S)(\delta)$ is a
highest weight vector of  weight $\mu$. Thus,
$N_\lambda/F\spcheck = U(\g).e_\mu \cong L(\mu)$.  Through the isomorphism
$N_\lambda \isomto M_r$, the elements $e_\lambda$ and $e_\mu$ correspond,
respectively, to $1$ and $U_2^{\delta+1} $ (which, in turn,
gives the element $\xi^m$ in the proof of Theorem~\ref{thm26}).

 (2) By Corollaries~\ref{summary} and~\ref{cor211}, the
annihilator of $N_\lambda$ is equal to the primitive ideal $J_r$.  Let $\calU$ denote
the subalgebra of $\cD(B)$ generated by the elements $E_{ij} \in \g$.
  Then, $\calU$ is a factor of the enveloping algebra $U(\g)$ and
$\overline{\calU} = U(\g)/J_r$ is a primitive quotient.

When $r=1$, $\calU$ and $\overline{\calU}$ are, respectively,
complexifications of the ``off-shell higher-spin algebra'' and ``on-shell
higher-spin algebra'' as defined in \cite[\S3]{vas03},
\cite[\S3.1.1]{bek08} and \cite[\S4.1]{bek12}. In particular, on the
Minkowski space-time (i.e.~in the real case $\R^{p,q}= \R^{n-1,1}$ as in 
\S\ref{sec7}) the algebra of symmetries $\Sscr{\dalembert_{n-1}}$ is
isomorphic to the ``on-shell higher-spin algebra'' of~\cite[\S3.1.3,
Corollary~3]{bek08}.
\end{rem}

%%%%%%%%%
\medskip

\Appendix
%\section{}  %\label{appendix}

Here we provide a proof of Lemma~\ref{lem912} and repeat (a minor variant
of) the statement of the result for the reader's convenience.  

\begin{lem}
 \label{lem912a}
{\rm (1)}  Set $\calI_r= \Ft^r A$. The bilinear form $\ascal{}{}$ from
Definition~\ref{pairing} satisfies 
the following properties.
\begin{enumerate}[{\rm (i)}]
\item $\ascal{A(p)}{A(q)} =0$ for  $p \ne q$ in $\N$ and
  $\ascal{a}{\phi} = \cF(a)(\phi)= \cF(\phi)(a)$ for all $a, \phi \in A(p)$.
\item $\ascal{}{}$ is  symmetric non degenerate and  $\fk$-invariant.
\item $L^{\perp\perp} = L$ for any graded subspace $L \subset A$.
\item $M_r^\perp = \calI_r$  and $\calI_r^\perp = M_r$.
\end{enumerate}

\noindent {\rm (2)} The pairing $\ascal{}{}$ induces a non degenerate
$\fk$-invariant symmetric pairing $\ascal{}{} : R_r \times M_r \lto \C$.
\end{lem}

\begin{proof}
  (1) (i,ii) With standard notation, let
  $U^{\mathbf{j}}, U^{\mathbf{k}}$ be monomials in the $U_j$ (where
  $\mathbf{j}, \mathbf{k}$ are multi-indices). Then,
  $\cF(U^{\mathbf{j}})= \partial_U^{\, \mathbf{j}}$ and 
  $\ascal{U^{\mathbf{j}}}{U^{\mathbf{k}}} = \partial_U^{\,
    \mathbf{j}}(U^{\mathbf{k}})= c_{\mathbf{j}} \, \delta_{\mathbf{j},
    \mathbf{k}}$ {for some $c_{\mathbf{j}} \in \N^*$}.
% \[
% \ascal{U^{\mathbf{j}}}{U^{\mathbf{k}}}
%   = \partial_U^{\mathbf{j}}(U^{\mathbf{k}})=
%   c_{\mathbf{j}} \, \delta_{\mathbf{j}, \mathbf{k}} \
%     \; \text{for some $c_{\mathbf{j}} \in \N^*$}.
% \]
   This proves (i) and implies that the pairing is symmetric non
   degenerate.
   
   Let $K \cong \SO(\Ft) \subset \Aut(A)$ be the algebraic group such that
   $\Lie(K) =\fk$.  It is well known (and easy to see) that if $g \in K$,  then
   $\cF(g.a) = g.\cF(a)$ for all $a \in A$. Hence,
   $\ascal{g.a}{g.\phi} = (g.\cF(a))(g.\phi)(0) = [g.\cF(a)(\phi)](0) =
   \cF(a)(\phi)(0) = \ascal{a}{\phi}$.
   Thus, $\ascal{}{}$ is $K$-invariant and therefore $\fk$-invariant.
   
   (iii) By~(i) we have perfect pairings
   $\ascal{}{}_m: A(m) \times A(m) \to \C$ for all $m \in \N$. For a
   subspace $P \subset A(m)$, denote by $\corth{P}$ its orthogonal w.r.t
   $\ascal{}{}_m$; thus $\corth{P}= P^\perp \cap A(m)$.  As $A(m)$ is
   finite dimensional, one has $P^{\circ\circ} =P$.  It is easily seen that
   $L^\perp$ is graded and that $L^\perp(m) = \corth{L(m)}$ for all
   $m$. Let $a \in L^{\perp\perp}$ and write $a= \sum_m a(m)$ with
   $a(m) \in A(m)$. We have $\ascal{a}{p} = \ascal{a(k)}{p} = 0$ for
   all $p \in \corth{L(k)}$. Hence $a(k) \in L(k)^{\circ\circ} = L(k)$ and
   we get $a \in L$. 
   
   (iv) Since $\cF(\Ft^r)= 2^{-r}\Delta_1^r$, it is clear that
   $M_r \subset \calI_r^\perp$. Conversely, if $\phi \in \calI_r^\perp$ and
   $U^{\mathbf{j}} \in A$, one has
   $\partial_U^{\, \mathbf{j}}[\Delta_1^r(\phi)](0) =
   2^r\ascal{\Ft^rU^{\mathbf{j}}}{\phi} = 0$.
   It then follows from~(i) that $\Delta_1^r(\phi) =0$ and $\phi \in
   M_r$. Thus $\calI_r^\perp = M_r$.

   \noindent (2) Observe that $\calI_r$ and $M_r$ are $\fk$-invariant subspaces of
   $A$ (see Corollary~\ref{cor643}).  Since $R_r= A/\calI_r$, the claim therefore
   follows from~(1)(ii) and~(1)(iv). 
\end{proof}
%%%%%%%%%%%

%%%%%%%%%%%%%%%%%%%%%%%%%%%%%
%\bigskip
 
%%%%%%%%%%%%%%%%%%%%%%%%%%%%%%

%\vfill
%%%%%%%%%%%%%%%%%%%%%%%%%%%%%%%%
\end{spacing}
\end{document}